\documentclass[10pt]{article}%[10pt]{article}%\documentclass
\usepackage{amssymb,latexsym,amsmath,amsthm,amsfonts, enumerate}
\usepackage{color}

\usepackage{amsfonts}
\usepackage{amsxtra}
\usepackage{float}
\usepackage[colorlinks,linkcolor=RoyalBlue,anchorcolor=Periwinkle, citecolor=Orange,urlcolor=Green]{hyperref}
\usepackage[usenames,dvipsnames]{xcolor}
\usepackage{enumitem}
\setenumerate[1]{itemsep=0pt,partopsep=0pt,parsep=\parskip,topsep=5pt}
\setitemize[1]{itemsep=0pt,partopsep=0pt,parsep=\parskip,topsep=5pt}
\setdescription{itemsep=0pt,partopsep=0pt,parsep=\parskip,topsep=5pt}
\usepackage{geometry} 
\usepackage{graphicx}
\usepackage{subfigure}
\usepackage{tikz}\usetikzlibrary{matrix}
\usepackage{adjustbox}
\usepackage{url}
\usepackage[all]{xy}
\usepackage{mathdots}
\usepackage{yhmath}
\usepackage{dsfont}

\def\<{\langle}
\def\>{\rangle}
\def\I{\mathrm{I}}
\newcommand{\Hom}{\mathrm{Hom}}

\newcommand{\M}{\mathcal{M}}
\newcommand{\II}{\mathrm{II}}
\newcommand{\Surf}{\mathcal{S}}

\newcommand{\Pic}{{F\scriptsize{IGURE}\ }}
\newcommand{\rbullet}{{\color{red}{\circ}}}
\newcommand{\Eqcls}{{\mathrm{Eq}\text{-}\mathrm{cls}}}
\newcommand{\Htpcls}{{\mathrm{Htp}\text{-}\mathrm{cls}}}
\newcommand{\pdim}{\mathrm{proj.dim}}
\newcommand{\idim}{\mathrm{inj.dim}}
\newcommand{\gldim}{\mathrm{gl.dim}}
\newcommand{\ind}{\mathrm{ind}}
\newcommand{\soc}{\mathrm{soc}}
\newcommand{\rad}{\mathrm{rad}}
\newcommand{\Gpdim}{\mathrm{G}\text{-}\mathrm{proj.dim}}

\newcommand{\proj}{\mathrm{proj}}
\newcommand{\inj}{\mathrm{inj}}
\newcommand{\Ker}{\mathrm{Ker}}

\newcommand{\simp}{\mathrm{simp}}
\newcommand{\tgamma}{\tilde{\gamma}}

\newcommand{\finiteDelta}{\mathsf{f}\text{-}\Delta}

\newcommand{\fforbid}{\mathsf{f}\text{-}\mathcal{F}}
\begin{document}

%% ----------------------------------------------------------------------

%\def\cosilting{quasi-cotilting }
%\def\Cosilting{Quasi-cotilting }

%%%%%%%%FROM latexexam.tex

\allowdisplaybreaks[4]

\newtheorem{theorem}{Theorem}[section]
\newtheorem{lemma}[theorem]{Lemma}
\newtheorem{corollary}[theorem]{Corollary}
\newtheorem{main theorem}[theorem]{Main Theorem}
\newtheorem{proposition}[theorem]{Proposition}
\newtheorem{definition}[theorem]{Definition}
\newtheorem{construction}[theorem]{Construction}
\newtheorem{remark}[theorem]{Remark}
\newtheorem{example}[theorem]{Example}
\newtheorem{notation}[theorem]{Notation}
\newtheorem{application}[theorem]{Application}

\usetikzlibrary{arrows}

\numberwithin{equation}{section}
%\def\KET{\mr{KerExt}_A^{i\ge 1}(T,-)\setcounter{equation}{0}}
%\def\KEC{\mr{KerExt}_A^{i\ge 1}(C,-)\setcounter{equation}{0}}
%\def\KEC{\mr{KerExt}_A^{i\ge 1}(-,C)\setcounter{equation}{0}}
%%%%%%%%%%

%-----------------------------------------------------------------------------------------
\renewcommand{\thefootnote}{\fnsymbol{footnote}}
\setcounter{footnote}{0}

\title{\bf Homological Dimensions of Gentle Algebras via Geometric Models
\footnote{This research was partially supported by NSFC (Grant Nos. 11971225, 12171207).}}
\footnotetext{E-mail: yzliu3@163.com (Y.-Z. Liu), hpgao07@163.com (H. Gao), huangzy@nju.edu.cn (Z. Huang).}
\smallskip
\author{\small Yu-Zhe Liu$^1$, Hanpeng Gao$^2$, Zhaoyong Huang\thanks{Corresponding author. } $^{,1}$\\
{\it \footnotesize $^1$ Department of Mathematics, Nanjing University, Nanjing 210093, Jiangsu Province, China; }\\
{\it \footnotesize $^2$ School of Mathematical Sciences, Anhui University, Hefei 230601, Anhui Province, China. }
}
\date{}
\maketitle
\baselineskip 15pt%16pt%14pt%15.5pt%\baselineskip  25.5pt %
%%%%%%%\hskip 18pt
%
% Abstract ------------------------------------------------------
%
\begin{abstract}
Let $A=kQ/I$ be a finite dimensional basic algebra
over an algebraically closed field $k$ which is a gentle algebra with the marked ribbon surface $(\mathcal{S}_A,\mathcal{M}_A,\Gamma_A)$.
It is known that $\mathcal{S}_A$ can be divided into some elementary polygons $\{\Delta_i\mid 1\le i\le d\}$ by $\Gamma_A$ which
has exactly one side in the boundary of $\mathcal{S}_A$. Let $\mathfrak{C}(\Delta_i)$ be
the number of sides of $\Delta_i$ belonging to $\Gamma_A$ if the unmarked boundary component of $\mathcal{S}_A$ is not a side of $\Delta_i$;
otherwise, $\mathfrak{C}(\Delta_i)=\infty$, and let
$\mathsf{f}\text{-}\Delta$ be the set of all non-$\infty$-elementary polygons
and $\mathcal{F}_A$ (respectively, ${\mathsf{f}\text{-}\mathcal{F}}_A$) the set of all forbidden threads (respectively, of finite length).
Then we have
\begin{enumerate}
\item[{\rm (1)}]
The global dimension of $A=\max\limits_{1\leq i\leq d}{\mathfrak{C}(\Delta_i)}-1
=\max\limits_{\mathit{\Pi}\in\mathcal{F}_A} l(\mathit{\Pi})$, where $l(\mathit{\Pi})$ is the length of $\mathit{\Pi}$.
\item[{\rm (2)}]
The left and right self-injective dimensions of $A=$
\begin{center}
$\begin{cases}
0,\ \mbox{\text{if {\it Q} is either a point or an oriented cycle with full relations};}\\
\max\limits_{\Delta_i\in{\mathsf{f}\text{-}\Delta}}\big\{1, {\mathfrak{C}(\Delta_i)}-1 \big\}=
\max\limits_{\mathit{\Pi}\in{\mathsf{f}\text{-}\mathcal{F}}_A} l(\mathit{\Pi}),\ \mbox{\text{otherwise}.}
\end{cases}$
\end{center}
\end{enumerate}
As a consequence, we get that the finiteness of the global dimension of gentle algebras
is invariant under AG-equivalence. In addition, we get that the number of indecomposable
non-projective Gorenstein projective modules over gentle algebras is
also invariant under AG-equivalence.
%\mskip\
\vspace{0.8cm}

\noindent {\it 2020 Mathematics Subject Classification}: 16E10; 16G10.

\vspace{0.2cm}
%\sskip\

\noindent {\it Key words}: global dimension; self-injective dimension; gentle algebras;
marked ribbon surfaces; geometric models; AG-equivalence.

\end{abstract}
%\smallskip
%
\vskip 30pt
% ----------------------------------------------------------------------
%% ----------------------------------------------------------------------
%\def\baselinestretch{1}

\section{\bf Introduction}
Gentle algebras were introduced by Assem and Skowro\'{n}ski \cite{AS1987} as appropriate context for the study of
algebras derived equivalent to hereditary algebras of type $\widetilde{\mathbb{A}}_n$. Note that
every gentle algebra is a special biserial algebra and all indecomposable modules over a special biserial algebra were
described by Butler and Ringel \cite[Section 3]{BR87} and Wald and Waschb\"{a}sch \cite[Proposition 2.3]{WW1985}.
Thus we can portray all indecomposable modules over a gentle algebra.
To be precise, each indecomposable module over a gentle algebra is either a string module or a band module.
Moreover, every gentle algebra is also a string algebra and all irreducible morphisms between string and band modules
over a string algebra were studied by Butler and Ringel \cite{BR87}, Crawley-Boevey \cite{C1989} and Krause \cite{Kra91}.
%Moreover, all Auslander-Reiten sequences over a string algebra were described by \cite[II.6]{Erd1990}.
%For the derived category of gentle algebra, the indecomposable objects and morphisms between two objects have been
%studied in \cite{BM2003, AAG2008, AL2016}.
In addition, by proving that the left and right self-injective dimensions of a gentle algebra are equal to the maximal length of certain
paths starting with a gentle arrow and bounded by one if there is no such arrow, Gei{\ss} and Reiten \cite {GR2005} obtained
that every gentle algebra is Gorenstein, that is, its left self-injective dimension $\idim_AA$
and right self-injective dimension $\idim A_A$ are finite.

The geometric models of gentle algebras were introduced in \cite{BS2018, HKK2017, OPS2018} and have been extensively
studied in \cite{APS2019,HZZ2020,QZ2017} based on the works in \cite{ALP2016,BM2003,Erd1990,R97}.
They originated in triangulated surfaces which is used in the study of cluster algebras and cluster categories,
such as \cite{BZ2011,LF2009a, LF2009b} and so on.
Opper, Plamondon and Schroll introduced ribbon graphs and marked ribbon surfaces for gentle algebras \cite{OPS2018}.
One can calculate all objects, morphisms and AG-invariants (a derived invariant defined in \cite{AAG2008})
in the derived category of a gentle algebra by marked ribbon surfaces (see \cite[Theorems 2.5, 3.3 and 6.1]{OPS2018}).
Furthermore, Baur and Simoes \cite{BS2018} introduced permissible curves
(see Definition \ref{BS2018, Def 3.1} below) to describe indecomposable modules,
and construct the geometric model of the category of finitely generated modules over a gentle algebra; and they also
provided the depictions of the Auslander-Reiten translate $\tau$ and Auslander-Reiten sequences through the rotations
of permissible curves in marked ribbon surfaces. Recently,
He, Zhou and Zhu \cite{HZZ2020} studied the category of finitely generated modules
over skew-gentle algebras (a generalization of gentle algebras) by punctured marked surfaces (a generalization of marked ribbon surfaces).
They provided a dimension formula for calculating morphisms between the indecomposable modules $M$ and $\tau M$ by equivalence classes
of tagged permissible curves, tagged intersections and intersection numbers.
In addition, by using the Koszul dual, Opper, Plamondon and Schroll \cite[Section 1, 1.7]{OPS2018}
proved that the global dimension $\gldim A$ of
a gentle algebra $A$ is infinite if and only if its quiver has
at least one oriented cycle with full relations.

The aim of this paper is to give some definite formulae for calculating the global
and self-injective dimensions of a given gentle algebra by marked ribbon surfaces.

Let $A=kQ/I$ be a gentle algebra and $\Surf_A=(\Surf_A,\M_A,\Gamma_A)$ its marked ribbon surface.
Then $\Surf_A$ is divided into some elementary polygons $\{\Delta_i\mid 1\le i\le d\}$ by $\Gamma_A$ which
has exactly one side $\subseteq \partial\Surf_A$ (the boundary of $\Surf_A$). We use $\mathfrak{C}(\Delta_i)$ to denote
the number of sides of $\Delta_i$ belonging to $\Gamma_A$ if the unmarked boundary component of $\Surf_A$ is not a side of $\Delta_i$;
otherwise, $\mathfrak{C}(\Delta_i)=\infty$.
We use $\finiteDelta$ to denote the set of all non-$\infty$-elementary polygons (see Remark \ref{rk. polygon-inf}),
and use $\mathcal{F}_A$ (respectively, $\fforbid_A$) to denote the set of all forbidden threads (respectively, of finite length)
(see Definition \ref{def. per. and forb.}). For any $\mathit{\Pi}\in\mathcal{F}_A$, we use $l(\mathit{\Pi})$ to denote its length.
Our main result is as follows.

\begin{theorem}\label{thm-1.2} {\rm (Theorems \ref{thm. gl.dim} and \ref{thm. Gdim A})}
Let $A=kQ/I$ be a gentle algebra and $\Surf_A=(\Surf_A,\M_A,\Gamma_A)$ its marked ribbon surface.
Then we have
\begin{enumerate}
\item[{\rm (1)}]
$\gldim A=\max\limits_{1\leq i\leq d}{\mathfrak{C}(\Delta_i)}-1
=\max\limits_{\mathit{\Pi}\in\mathcal{F}_A} l(\mathit{\Pi})$.
\item[{\rm (2)}]
$\idim_AA=\idim A_A=$
  \begin{center}
  $\begin{cases}
    0,\ \mbox{{if $Q$ is either a point or an oriented cycle with full relations};}\\
    \max\limits_{\Delta_i\in\finiteDelta}\big\{1, {\mathfrak{C}(\Delta_i)}-1 \big\}=
    \max\limits_{\mathit{\Pi}\in\fforbid_A} l(\mathit{\Pi}),\ \mbox{{otherwise}.}
  \end{cases}$
  \end{center}
\end{enumerate}
\end{theorem}

The paper is organized as follows. In Section 2,
we recall some terminologies and some preliminary results needed in
this paper. In particular, we give the definition of gentle algebras
and some related notions related to their geometric models.
In Section 3, we give the descriptions of some short exact sequences by geometric models,
which will be used frequently in the sequel.

\begin{theorem}\label{thm-1.1} {\rm (Theorem \ref{thm. s. e. s.})}
Let $A=kQ/I$ be a gentle algebra and $\Surf_A = (\Surf_A, \M_A, \Gamma_A)$ its marked ribbon surface.
\begin{enumerate}
\item[{\rm(1)}] If $c$, $c'$ and $c''$ are permissible curves such that
the positional relationship of them is given by \Pic \ref{Pic. s. e. s.} {Case I},
then there exists an exact sequence
$$0 \longrightarrow M(c') \longrightarrow M(c) \longrightarrow M(c'') \longrightarrow 0.$$
\item[{\rm(2)}] If $c$, $c'$, $c_{\I}''$ and $c_{\II}''$ are permissible curves
such that the positional relationship of them is given by \Pic \ref{Pic. s. e. s.} {Case II},
then there exists an exact sequence
$$0 \longrightarrow M(c') \longrightarrow M(c)
\longrightarrow M(c_{\I}''\oplus c_{\II}'') \longrightarrow 0.$$
\item[\rm{(3)}] If $c$, $c_{\I}'$, $c_{\II}'$ and $c''$ are permissible curves
such that the positional relationship of them is given by \Pic \ref{Pic. s. e. s.} {Case III},
then there exists an exact sequence $$0 \longrightarrow M(c_{\I}'\oplus c_{\II}')
\longrightarrow M(c) \longrightarrow M(c'') \longrightarrow 0.$$
\end{enumerate}
\end{theorem}

\begin{figure}[htbp]
\begin{center}
%%%%%%%%%%%%%%%%%%%%%%%%%%%%%%%%%%%%%%%%%%%%%%%%%%%%%%%%%%%%%%
%%%%%%%%%%%%%%%%%%%%%%%%%%% pic 1 %%%%%%%%%%%%%%%%%%%%%%%%%%%%
%%%%%%%%%%%%%%%%%%%%%%%%%%%%%%%%%%%%%%%%%%%%%%%%%%%%%%%%%%%%%%
\begin{tikzpicture}[scale=0.7] % ÍŒÆ¬µÈ±ÈÀý·ÅËõÎª³õÊŒµÄ0.7±¶
\draw [line width=1.2pt] (-2, 1)-- (-2,-1);
\draw [line width=1.2pt] ( 2, 1)-- ( 2,-1);
\draw [line width=1.2pt] (-1, 2)-- ( 1, 2);
\draw [line width=1.2pt] (-1,-2)-- ( 1,-2);
\draw [line width=12pt] [white] (-0.2,2) -- (0.2,2);
\draw (0,-2)-- (-0.5,2);
\draw (0,-2)-- (0.5,2);
\begin{scriptsize}
\fill [color=black] (0,-2)    circle (2.2pt);
\fill [color=black] (-0.5,2)  circle (2.2pt);
\fill [color=black] (0.5,2)   circle (2.2pt);
\fill [color=black] (-2,-0.5) circle (2.2pt);
\fill [color=black] (2,0.5)   circle (2.2pt);
\end{scriptsize}
\draw[red] (-2,-0.5) to[out=0, in=180] (2,0.5); \draw[red] (-1,-0.5) node {$c$};
\draw[blue] (-2,-0.5) to (0.5,2); \draw[blue] (-1,0.7) node {$c'$};
\draw[orange] (2,0.5) to (-0.5,2); \draw[orange] (1,1.3) node {$c''$};
% number
\draw  (0,-2.2) node[below]{$\text{Case I}$};
\end{tikzpicture}
\ \
%%%%%%%%%%%%%%%%%%%%%%%%%%%%%%%%%%%%%%%%%%%%%%%%%%%%%%%%%%%%%%
%%%%%%%%%%%%%%%%%%%%%%%%%%% pic 2 %%%%%%%%%%%%%%%%%%%%%%%%%%%%
%%%%%%%%%%%%%%%%%%%%%%%%%%%%%%%%%%%%%%%%%%%%%%%%%%%%%%%%%%%%%%
\ \
\begin{tikzpicture}[scale=0.7]
\draw [line width=1.2pt] (-2,1)-- (-2,-1);
\draw [line width=1.2pt] (2,1)-- (2,-1);
\draw [line width=1.2pt] (-1,2)-- (1,2);
\draw [line width=1.2pt] (-1,-2)-- (1,-2);
\draw [line width=12pt] [white] (-2,-0.2) -- (-2,0.2);
\draw [line width=12pt] [white] ( 2,-0.2) -- ( 2,0.2);
\draw [line width=12pt] [white] ( 0.2, 2) -- ( 0.6, 2);
\draw [line width=12pt] [white] (-0.2,-2) -- (-0.6,-2);
\draw (0,-2)-- (0, 2);
\draw (0, 2)-- (-0.8,-2.0); \draw (0,-2)-- ( 0.8, 2.0);
\draw (0, 2)-- (-1.5,-1.5); \draw (0,-2)-- ( 1.5, 1.5);
\draw (0, 2)-- (-2.0,-0.8); \draw (0,-2)-- ( 2.0, 0.8);
\draw (0, 2)-- (-2.0, 0.4); \draw (0,-2)-- ( 2.0,-0.4);
\begin{scriptsize}
\fill [color=black] ( 0.0,-2.0) circle (2.2pt); \fill [color=black] ( 0.0, 2.0) circle (2.2pt);
\fill [color=black] ( 0.8, 2.0) circle (2.2pt); \fill [color=black] (-0.8,-2.0) circle (2.2pt);
\fill [color=black] (-2.0,-0.8) circle (2.2pt); \fill [color=black] ( 2.0, 0.8) circle (2.2pt);
\end{scriptsize}
\draw[red] (-2,1.5) to[out=0, in=180] ( 2,-1.5); \draw[red] (-1, 1.1) node[above] {$c$};
\draw[blue] (-2, 0.4) to[out=0, in=180] ( 2,-0.4); \draw[blue] (-1,0.4) node {$c'$};
\draw[orange] (-2.0,-0.8) to[out=70, in=0] (-2, 1.45); \draw[orange] (-1.8, 1.4) node[below] {$c_{\I}''$};
\draw[orange] ( 2.0, 0.8) to[out=-110, in=180] ( 2,-1.45); \draw[orange] ( 1.8,-1.4) node[above] {$c_{\II}''$};
\fill [color=black] (-2, 0.4) circle (2.2pt); \fill [color=black] ( 2,-0.4) circle (2.2pt);
% number
\draw  (0,-2.2) node[below]{$\text{Case II}$};
\end{tikzpicture}
\ \
%%%%%%%%%%%%%%%%%%%%%%%%%%%%%%%%%%%%%%%%%%%%%%%%%%%%%%%%%%%%%%
%%%%%%%%%%%%%%%%%%%%%%%%%%% pic 3 %%%%%%%%%%%%%%%%%%%%%%%%%%%%
%%%%%%%%%%%%%%%%%%%%%%%%%%%%%%%%%%%%%%%%%%%%%%%%%%%%%%%%%%%%%%
\ \
\begin{tikzpicture}[scale=0.7]
\draw [line width=1.2pt] (-2,1)-- (-2,-1);
\draw [line width=1.2pt] (2,1)-- (2,-1);
\draw [line width=1.2pt] (-1,2)-- (1,2);
\draw [line width=1.2pt] (-1,-2)-- (1,-2);
\draw [line width=12pt] [white] (-2,-0.2) -- (-2,0.2);
\draw [line width=12pt] [white] ( 2,-0.2) -- ( 2,0.2);
\draw [line width=12pt] [white] (-0.2, 2) -- (-0.6, 2);
\draw [line width=12pt] [white] ( 0.2,-2) -- ( 0.6,-2);
\draw (0,-2)-- (0, 2);
\draw (0,-2)-- (-0.8, 2.0); \draw (0, 2)-- ( 0.8,-2.0);
\draw (0,-2)-- (-1.5, 1.5); \draw (0, 2)-- ( 1.5,-1.5);
\draw (0,-2)-- (-2.0,-0.8); \draw (0, 2)-- ( 2.0, 0.8);
\draw (0,-2)-- (-2.0, 0.4); \draw (0, 2)-- ( 2.0,-0.4);
\begin{scriptsize}
\fill [color=black] ( 0.0,-2.0) circle (2.2pt); \fill [color=black] ( 0.0, 2.0) circle (2.2pt);
\fill [color=black] ( 0.8,-2.0) circle (2.2pt); \fill [color=black] (-0.8, 2.0) circle (2.2pt);
\fill [color=black] (-2.0,-0.8) circle (2.2pt); \fill [color=black] ( 2.0, 0.8) circle (2.2pt);
\end{scriptsize}
\draw[red] (-2,1.5) to[out=0, in=180] ( 2,-1.5); \draw[red] (-1, 1.1) node[above] {$c$};
\draw[orange] (-2, 0.4) to[out=0, in=180] ( 2,-0.4); \draw[orange] (-1,0.4) node {$c''$};
\draw[blue] (-2.0,-0.8) to[out=70, in=0] (-2, 1.45); \draw[blue] (-1.8, 1.4) node[below] {$c_{\I}'$};
\draw[blue] ( 2.0, 0.8) to[out=-110, in=180] ( 2,-1.45); \draw[blue] ( 1.8,-1.4) node[above] {$c_{\II}'$};
\fill [color=black] (-2, 0.4) circle (2.2pt); \fill [color=black] ( 2,-0.4) circle (2.2pt);
% number
\draw  (0,-2.2) node[below]{$\text{Case III}$};
\end{tikzpicture}
\caption{For each case, the module $M'$ corresponding to blue permissible curve(s) is a submodule of $M$ corresponding
to the red permissible curve, and the module $M''$ corresponding to orange permissible curve(s) is isomorphic to the quotient $M/M'$.}
\label{Pic. s. e. s.}
\end{center}
\end{figure}
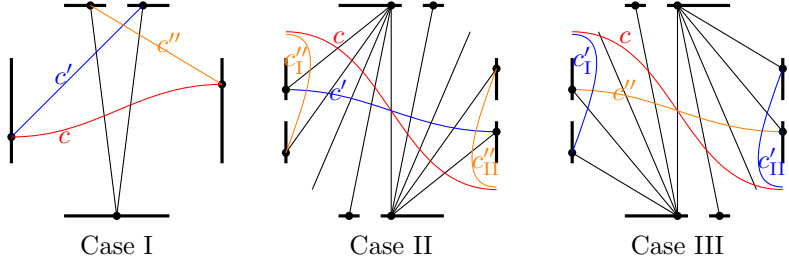

In Section 4, we describe the projective covers and injective envelopes of indecomposable modules
over gentle algebras by marked ribbon surfaces (Theorems \ref{thm. proj. cover} and \ref{thm. inj. envelope}).
Then in Section 5, we give the proof of Theorem \ref{thm-1.2}(1).

In Section \ref{Section Gproj}, by describing all Gorenstein projective modules in geometric models,
we give a proof of the following equalities
$$\idim_AA=\idim A_A=\begin{cases}
0,\ \mbox{{if $Q$ is either a point or an oriented cycle with full relations};}\\
\max\limits_{\Delta_i\in\finiteDelta}\big\{1, {\mathfrak{C}(\Delta_i)}-1 \big\},\ \mbox{{otherwise}.}
\end{cases}$$
Then by describing minimal projective resolutions of all injective modules, we obtain a formula for calculating
the left and right self-injective dimensions of $A$ by marked ribbon surfaces (Proposition \ref{prop. pdim. of inj. mod.}).
As a consequence, we give the proof of Theorem \ref{thm-1.2}(2) including another proof of
the above equalities. In addition, we prove that the number of indecomposable
non-projective Gorenstein projective modules over gentle algebras is
invariant under AG-equivalence (Proposition \ref{prop. CM-corresponding}).
In Section 7, we give some examples to illustrate the obtained results; that is,
we apply Theorem \ref{thm-1.2} to calculate the global and self-injective dimensions of some gentle algebras.

\section{\bf The geometric models of gentle algebras}

In this paper, assume that $A=kQ/I$ is a finite dimensional basic algebra
over an algebraically closed field $k$, where $I$ is an admissible ideal of $kQ$
and $Q=(Q_0, Q_1)$ is a finite quiver with $Q_0,Q_1$ the sets of all vertices and arrows,
respectively. We use $s,t$ to denote two functions from $Q_1$ to $Q_0$ which send each arrow
to its source and target, respectively. The multiplication $\alpha_1\alpha_2$ of two arrows
$\alpha_1$ and $\alpha_2$ in $Q_1$ is defined by the concatenation if $t(\alpha_1)=s(\alpha_2)$
or zero if $t(\alpha_1)\ne s(\alpha_2)$. For a set $X$, the number of elements of $X$ is denoted by $\sharp X$.
We use $\bmod A$ to denote the category of finitely generated right $A$-modules, and use
$\gldim A$ to denote the global dimension of $A$. For a module $M\in\bmod A$, we use
$\pdim M$ and $\idim M$ to denote the projective and injective dimensions of $M$, respectively, and
use ${\rm top} M$, $\soc M$ and $P(M)$ to denote the top, socle and projective cover of $M$, respectively.
We denote by $N \le M$ (respectively, $N\le_{\oplus} M$) if $N$ is a submodule (respectively, direct summand) of $M$.

\subsection{Marked ribbon surfaces}

In gentle algebras, Opper, Plamondon and Schrollin \cite{OPS2018} introduced the notion of marked ribbon surfaces
by defining ribbon graphs first. Marked ribbon surfaces are often referred to as marked surfaces,
such as \cite{APS2019, BS2018, HZZ2020}.
In this paper we still use the terminology from \cite{OPS2018} (but we do not need the definition of ribbon graphs).

\begin{definition}{\rm (Marked ribbon surfaces) \cite{APS2019,BS2018,OPS2018} \label{def. m. r. s.}
A {\it marked ribbon surface} is a triple $(\Surf, \M, \Gamma)$, where
\begin{itemize}
\item $\Surf$ is an oriented connected surface with non-empty boundary, and we use $\partial \Surf$ to denote its boundary.
\item $\M$ is a finite set of points on $\partial \Surf$, and
each element in $\M$ is called a {\it marked point} and denoted by the symbol $\bullet$.
\item $\Gamma$ is called a {\it full formal $\bullet$-arc system} of $\Surf$ such that the following hold:
\begin{itemize}
\item Each element in $\Gamma$, called a {\it $\bullet$-arc} (or an arc), is a curve whose ending points belong to $\M$.
\item For any two $\bullet$-arcs $\gamma$ and $\gamma'$, there is no intersection in the inner of $\Surf$.
\item $\Surf$ is divided into some parts $\{\Delta_i\mid 1\le i\le d\}$ by $\Gamma$
such that each part is a polygon which has exactly a side, say a {\it single boundary arc}, not belonging to $\Gamma$.
The set of all polygons $\Delta = \{ \Delta_i\mid 1\le i\le d \}$ is called an {\it original $\bullet$-dissection}
(or {\it original dissection}) of $\Surf$ \cite{APS2019}, and each polygon $\Delta_i$ above is said to be {\it elementary}.
We denote by $\mathfrak{S}(\Delta_i)$, say the {\it arc set}, the set of all sides of $\Delta_i$ belong to $\Gamma$.
\end{itemize}
\end{itemize}

We give some supplements to Definition \ref{def. m. r. s.} as follows.
\begin{enumerate}
\item[(1)] Each arc in $\Gamma$ can be seen as a pair $(\gamma, \gamma^-)$, where $\gamma$ is a continuous function
$\gamma: (0,1)\to \Surf\backslash \partial\Surf$ such that $\gamma(0^+)=\lim\limits_{x\to 0^+}\gamma(x)$ and
$\gamma(1^-)=\lim\limits_{x\to 1^-}\gamma(x)$ lie in $\M$, and $\gamma^-$ is the formal inverse of $\gamma$,
that is, $\gamma^-(x) = \gamma(1-x)$.
Based on this, $\gamma = \gamma(x)$ and $\gamma^- = \gamma(1-x)$ correspond to the same arc in $\Gamma$.
For the sake of simplicity, we suppose that $\gamma$ and $\gamma^-$ both are elements in $\Gamma$ and
$\gamma \simeq \gamma^- \simeq (\gamma, \gamma^-)$. This hypothesis can be used to simplify some descriptions,
such as Theorem \ref{prop. pdM and idM} and Example \ref{exp. 1}.

\item[(2)] If there is some $i$ with $1\leq i \leq d$ such that $\Delta_i$ is a two-gons,
then $\Delta_i$ is one of the two forms (i) and (ii) shown in \Pic \ref{fig:2-gon}.
If $\Delta_i$ is of the form (i), then adding a new point (will be denoted by the symbol $\rbullet$ and called an {\it extra point})
on the side not belonging to $\Gamma$.  The set of all extra points is denoted by $E$.

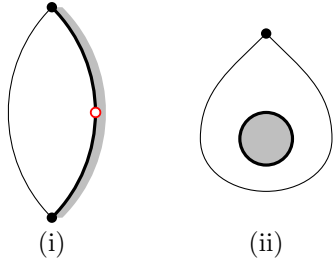
\begin{figure}[H]
\begin{center}
%%%%%%%%%%%%%%%%%%%%%%%%%%%%%%%%%%%%%%%%%%%%%%%%%%%%%%%%%%%%%%
%%%%%%%%%%%%%%%%%%%%%%%%%%% pic 1 %%%%%%%%%%%%%%%%%%%%%%%%%%%%
%%%%%%%%%%%%%%%%%%%%%%%%%%%%%%%%%%%%%%%%%%%%%%%%%%%%%%%%%%%%%%
\begin{tikzpicture}[scale=0.7] %
\draw (0,-2) to[out=135, in=225] (0,2);
\fill[black!25] (0,-2) to[out=45, in=-45] (0,2) -- (0.2,2) to[out=-45, in=45] (0.2,-2) -- (0,-2);
\draw (0,-2) to[out=45, in=-45] (0,2)[line width=1.2pt];
\begin{scriptsize}
\fill [color=black] (0, 2) circle (2.8pt);
\fill [color=black] (0,-2) circle (2.8pt);
\end{scriptsize}
\filldraw[color=red] (0.83, 0) circle (2.8pt);
\filldraw[color=white] (0.83, 0) circle (1.8pt);
\draw (0,-2.5) node{(i)};
\end{tikzpicture}
\ \ \ \ \ \ \ \
%%%%%%%%%%%%%%%%%%%%%%%%%%%%%%%%%%%%%%%%%%%%%%%%%%%%%%%%%%%%%%
%%%%%%%%%%%%%%%%%%%%%%%%%%% pic 2 %%%%%%%%%%%%%%%%%%%%%%%%%%%%
%%%%%%%%%%%%%%%%%%%%%%%%%%%%%%%%%%%%%%%%%%%%%%%%%%%%%%%%%%%%%%
\begin{tikzpicture}[scale=0.7] %
\draw[shift={(0,-0.5)}] (0,2) to[out=-135, in=90] (-1.25,0) to[out=-90,in=180] (0,-1)
to[out=0,in=-90] (1.25,0) to[out=90,in=-45] (0,2);
\fill[shift={(0,-0.5)}][black!25] (0,0) circle (0.5) [line width=1.2pt];
\draw[shift={(0,-0.5)}] (0,0) circle (0.5) [line width=1.2pt];
\begin{scriptsize}
\fill[shift={(0,-0.5)}] [color=black] (0, 2) circle (2.8pt);
\end{scriptsize}
\draw (0,-2.5) node{(ii)};
\end{tikzpicture}
\caption{ (i) and (ii) are 2-gons. In (i), we add an extra point on the side not belonging to $\Gamma$. }
\label{fig:2-gon}
\end{center}
\end{figure}

\item[(3)] The positive direction of $\partial\Surf$ is defined as follows:
the interior of $\Surf$ is on the left while walking along the boundary.
\end{enumerate}}
\end{definition}

\begin{remark} \rm
\begin{itemize}
\item[ ]
\item[(1)] In some cases, we do not know whether a point in the surface is a marked point or an extra point,
we use ``$\pmb{\times}$'' to express it.
\item[(2)] We say that arcs $\gamma_1, \cdots, \gamma_n$ with a common endpoint $p$
{\it surround $p$ in the clockwise $($respectively, counterclockwise$)$ order} if $\gamma_{i+1}$
is right (respectively, left) to $\gamma_{i}$ ($1\le i\le n-1$) at the point $p$ (see \Pic \ref{fig:surround}).

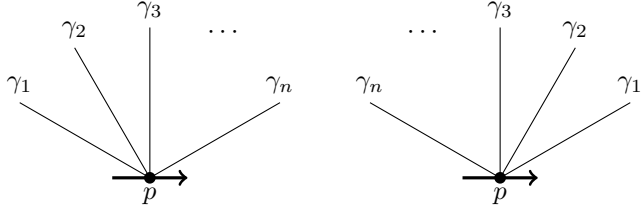
\begin{figure}[htbp]
\begin{center}
\begin{tikzpicture}%[xscale=0.8]
\draw [line width =1.2pt][->] (-0.5,0) -- (0.5,0);
\fill (0,0) circle (2.2pt) node[below]{$p$};
\draw (0,0) -- (-1.73, 1) node[above]{$\gamma_1$};
\draw (0,0) -- (-1, 1.73) node[above]{$\gamma_2$};
\draw (0,0) -- (0, 2) node[above]{$\gamma_3$};
\draw ( 1, 1.73) node[above]{$\cdots$};
\draw (0,0) -- ( 1.73, 1) node[above]{$\gamma_n$};
\end{tikzpicture}
\ \ \ \
\begin{tikzpicture}%[xscale=0.8]
\draw [line width =1.2pt][->] (-0.5,0) -- (0.5,0);
\fill (0,0) circle (2.2pt) node[below]{$p$};
\draw (0,0) -- (-1.73, 1) node[above]{$\gamma_n$};
\draw (-1, 1.73) node[above]{$\cdots$};
\draw (0,0) -- (0, 2) node[above]{$\gamma_3$};
\draw (0,0) -- ( 1, 1.73) node[above]{$\gamma_2$};
\draw (0,0) -- ( 1.73, 1) node[above]{$\gamma_1$};
\end{tikzpicture}
\caption{$\gamma_1, \cdots, \gamma_n$ with common endpoint $p$ are surround $p$
in the clockwise $($respectively, counterclockwise$)$ order. }
\label{fig:surround}
\end{center}
\end{figure}

\item[(3)] In \cite{APS2019}, Amiot, Plamondon and Schroll defined the $\rbullet$-points and $\rbullet$-dissection for marked ribbon surfaces.
The {\it  $\rbullet$-point} is either a point belonging to $\partial\Surf$ which lies between two adjacent marked points or a
boundary component of $\Surf$ with no marked point.
The {\it  $\rbullet$-dissection} $\Delta_{\rbullet}$ of $\Surf$, whose elements are called {\it $\rbullet$-arcs},
is dual of $\bullet$-dissection $\Delta=\Delta_{\bullet}$, that is,
each $\rbullet$-arc is a curve whose endpoints are $\rbullet$-points,
any two $\rbullet$-arcs have no intersection point in $\Surf\backslash\partial\Surf$,
and for every $\bullet$-arc (respectively, $\rbullet$-arc), there exists a unique $\rbullet$-arc (respectively, $\bullet$-arc)
such that they have only one intersection point in $\Surf\backslash\partial\Surf$.
Indeed, $\Delta_{\rbullet}$ is uniquely determined by $\Delta_{\bullet}$, and it is said to be
{\it lamination} of $\Surf$ by Opper, Plamondon and Schroll in \cite{OPS2018}, although their definitions seem different.
In this paper, we do not need the definitions of $\rbullet$-point and $\rbullet$-dissection,
but we need to define the extra points which are special $\rbullet$-points.
They are used in the description of some indecomposable modules.
\end{itemize}
\end{remark}

\begin{definition}\rm ($k$-algebras of marked ribbon surfaces) \cite{APS2019, OPS2018} \label{def. alg. of m. r. s.}
Let $\Surf=(\Surf,\M, \Gamma)$ be a marked ribbon surface.
We associate a quiver $Q_{\Surf} = (Q_0, Q_1)$ and relation $I_{\Surf} = \langle R \rangle$ to $\Surf$ as follows:
\begin{enumerate}
\item[(1)] The vertexes in $Q_0$ correspond to the $\bullet$-arcs in $\Gamma$ and this corresponding
is one-to-one, denoted by $\mathfrak{v}: Q_0\to \Gamma$, that is, each $\bullet$-arc can be seen as a vertex of $Q_\Surf$;
\item[(2)] For every elementary polygon $\Delta_i$ with two sides $u,v \in \Gamma$,
if $v$ follows $u$ in the counterclockwise order, then there is an arrow $u \to v$ in $Q_1$;
\item[(3)] $R$ is the set of all compositions $\alpha\beta$, where $\alpha: u\to v$ and $\beta: v\to w$ satisfies that
$u$, $v$ and $w$ are sides of the same elementary polygon.
\end{enumerate}
\noindent Then {\it the $k$-algebra $A_{\Surf}$ of the marked ribbon surfaces $\Surf$}
is defined by $A_{\Surf}:=kQ_{\Surf}/I_{\Surf}$.
\end{definition}

\begin{definition}\rm (Gentle algebras) \cite{AAG2008,APS2019,AS1987,OPS2018}
A $k$-algebra $A=kQ/I$ is said to be {\it  special biserial} if the following conditions are satisfied.
\begin{enumerate}
\item[(1)] For each vertex $v\in Q_0$, $\sharp\{\alpha\in Q_1 \mid s(\alpha)=v \}\le 2$
and $\sharp\{\alpha\in Q_1\mid t(\alpha)=v \}\le 2$;
\item[(2)] For each arrow $\beta\in Q_1$, $\sharp\{\alpha\in Q_1\mid s(\beta)=t(\alpha), \alpha\beta \notin I\}\le 1$
and $\sharp\{\alpha\in Q_1\mid s(\alpha)=t(\beta), \beta\alpha \notin I\}\le 1$;
\item[(3)] For each arrow $\beta\in Q_1$, there exists a bound $n\in\mathbb{N}^+$ such that
\begin{itemize}
\item For each path $p=\beta_1\cdots\beta_{n-1}$ where $t(p)=s(\beta)$, the path $p\beta$ contains a subpath in $I$;
\item For each path $p'=\beta_1'\cdots\beta_{n-1}'$ where $t(\beta)=s(p')$, the path $\beta p'$ contains a subpath in $I$.
\end{itemize}
\end{enumerate}
Moreover, $A=kQ/I$ is said to be {\it  gentle} if it is a special biserial algebra such that the following hold:
\begin{enumerate}
\item[(4)] $Q$ is a bound quiver;
\item[(5)] All relations in $I$ are paths of length $2$;
\item[(6)] For each arrow $\beta\in Q_1$, $\sharp\{\alpha\in Q_1\mid s(\beta)=t(\alpha), \alpha\beta \in I\}\le 1$
and $\sharp\{\alpha\in Q_1\mid s(\alpha)=t(\beta), \beta\alpha \in I\}\le 1$.
\end{enumerate}
\end{definition}

Notice that a special biserial algebra is finite dimensional if and only if $Q$ is finite, thus a gentle algebra
is always finite dimensional. Moreover, Opper, Plamondon and Schroll proved that $A_{\Surf} := kQ_{\Surf}/I_{\Surf}$
is gentle, and that there is a bijection between isomorphism classes of gentle algebras and homotopy
classes of marked ribbon surfaces (see \cite[Proposition 1.21]{OPS2018}).
It should be pointed out that there are different definitions of gentle algebras.
For example, the definition of gentle algebras in \cite{AAG2008,AS1987} does not require the fourth condition,
and so there are gentle algebras whose quivers are not bound;
and that of gentle algebras in \cite{APS2019,OPS2018} does not require the third condition,
and so there are gentle algebras whose quivers has at least one oriented cycles with no relations.

\subsection{Permissible curves and permissible closed curves}
Each curve $c$ in a marked ribbon surface $\Surf=(\Surf, \M, \Gamma)$ can be defined
as a function $c: [0,1]\to \Surf$ where $c(0)$ and $c(1)$ are its endpoints.
We say that $c(0)$ and $c(1)$ are its {\it starting point} and {\it ending point}, respectively.
In this paper, we only consider such curves whose points lie in the interior of $\Surf$ except endpoints,
that is, $\{c(x)\mid 0<x<1\}\subseteq \Surf\backslash\partial\Surf$, and whose endpoints are elements belong to $\M\cup E$.
Thus each non-closed curve can be regarded as a function $c: (0,1) \to \Surf\backslash\partial\Surf$ where $c(0^+)$ and $c(1^-)$ are its endpoints
and each closed curve can be regarded as a function $c:[0,1]\to \Surf\backslash\partial\Surf$ where $c(0)=c(1)\in \Surf\backslash\partial\Surf$.
In this paper, for each curve $c$ with the endpoints $c(0^+)$ and $c(1^-)$ on the marked ribbon surface,
we always suppose that the number of intersections between $c$ and $\Gamma$ is minimal up to homotopy.

\begin{definition} \rm (Permissible curves and permissible closed curves) \cite[Definition 3.1]{BS2018} \label{BS2018, Def 3.1}
Let $\Surf=(\Surf,\M,\Gamma)$ be a marked ribbon surface.
\begin{enumerate}
\item[(1)] A curve $c:(0,1)\to \Surf\backslash\partial\Surf$ (or a closed curve $c:[0,1]\to \Surf\backslash\partial\Surf$)
in $\Surf$ is said to {\it consecutively cross} $u, v \in \Gamma$ if
$p_1 = c \cap u$ and $p_2 = c \cap v$ belong to $\Surf\backslash \partial \Surf$,
and the segment of $c$ between the points $p_1$ and $p_2$ does not cross any other arc in $\Gamma$.

Then each non-closed curve $c$ consecutively crossing $\gamma_1,\cdots, \gamma_r \in \Gamma$ can be written as $c=c_1c_2\cdots c_rc_{r+1}$, where
each $c_i$ ($2\le i\le r$) is a segment between $\gamma_{i-1}$ and $\gamma_i$ does not cross any other arc in $\Gamma$,
and $c_1$ (respectively, $c_{r+1}$) is a segment between $c(0^+)$ and $\gamma_1$ (respectively, $\gamma_r$ and $c(1^-)$)
does not cross any other arc in $\Gamma$.
The segment $c_i$ ($2\le i\le r-1$) is called {\it middle}, and the segments $c_1$ and $c_r$ are called {\it end}.
Similarly, each closed curve $c$ consecutively crossing $\gamma_1,\cdots, \gamma_r \in \Gamma$ can be written as $c=c_1c_2\cdots c_r$, where
each $c_i$ $(2\le i\le r)$ is a segment between $\gamma_{i-1}$ and $\gamma_i$ does not cross any other arc in $\Gamma$,
and $c_1$ is a segment between $\gamma_r$ and $\gamma_1$.

\item[(2)] Let $B$ be a boundary component with no marked point of $\Surf$, called an
{\it unmarked boundary component} of $\Surf$, and let $c:(0,1)\to \Surf\backslash\partial \Surf$ be a curve in $\Surf$.
Write $c=c_1c'c_2$ where $c_1$ (respectively, $c_2$) is the end segment between $c(0^+)$ (respectively, $c(1^-)$)
and the first (respectively, last) crossing point in $\Gamma$.
We define the {\it winding number} of $c_i$ $(i=1,2)$ around $B$ as the minimum number of times travels
$\bar{c}$ around $B$ in either direction, where $\bar{c}$ lies in the homotopy class of $c_i$.

\item[(3)] A curve $c$ is called {\it permissible} if the following conditions are satisfied.
\begin{enumerate}
\item The winding number of $c$ around an unmarked boundary component is either 0 or 1;
\item If $c$ consecutively crosses two arcs $u$ and $v$ in $\Gamma$,
then $u$ and $v$ have a common endpoint $p$, denoted by $p(u,v,c)$, lying in $\M$, and locally we have a triangle with $p$ a vertex
(as Picture I in \Pic \ref{exp. of perm. c.}, the curves $c_1$, $c$, $c'$ and $c''$ are permissible).
\end{enumerate}
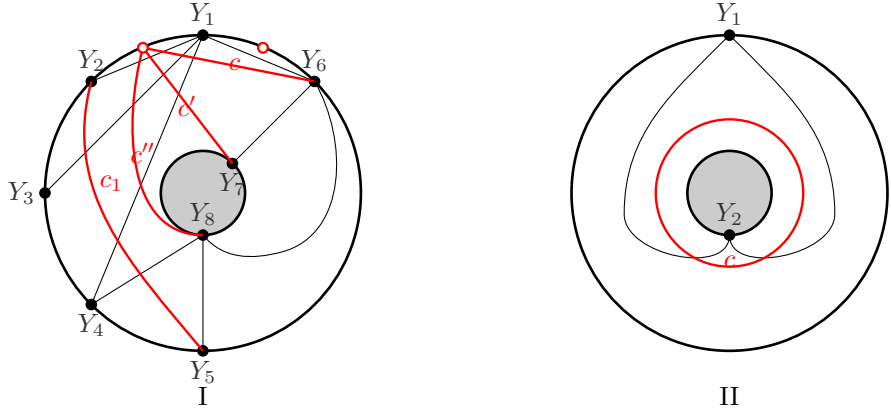
\begin{figure}[htbp]
\begin{center}
\begin{tikzpicture} [scale=1.4] % \tiny
\filldraw[color=black!20] (0,0) circle (0.4);
 \draw[line width=1pt] (0,0) circle (1.5);
 \draw[line width=1pt] (0,0) circle (0.4);
 %%%%%%%%%% half-edges and points %%%%%%%%%%
 \draw ( 0.00,-0.40)--( 0.00,-1.50);     \filldraw ( 0.00,-0.40) circle (0.05);
                                         \filldraw ( 0.00,-1.50) circle (0.05);
 \draw ( 0.00,-0.40)--(-1.06,-1.06);     \filldraw (-1.06,-1.06) circle (0.05);
 \draw (-1.06,-1.06)--( 0.00, 1.50);     \filldraw ( 0.00, 1.50) circle (0.05);
 \draw ( 0.00, 1.50)--(-1.06, 1.06);     \filldraw (-1.06, 1.06) circle (0.05);
 \draw ( 0.00, 1.50)--(-1.50, 0.00);     \filldraw (-1.50, 0.00) circle (0.05);
 \draw ( 0.00, 1.50)--( 1.06, 1.06);     \filldraw ( 1.06, 1.06) circle (0.05);
 \draw ( 1.06, 1.06)--( 0.28, 0.28);     \filldraw ( 0.28, 0.28) circle (0.05);
 \draw ( 0.00,-0.40) to[out=-50,in=195] ( 0.75,-0.56) to[out=15,in=300] ( 1.06, 1.06);
 \filldraw[color=red] ( 0.57, 1.38) circle (0.05);
 \filldraw[color=red] (-0.57, 1.38) circle (0.05);
 \draw[thick,color=black!80] ( 0.00, 1.50) node[above]{$Y_1$};
 \draw[thick,color=black!80] (-1.06, 1.06) node[above]{$Y_2$};
 \draw[thick,color=black!80] (-1.50, 0.00) node[ left]{$Y_3$};
 \draw[thick,color=black!80] (-1.06,-1.06) node[below]{$Y_4$};
 \draw[thick,color=black!80] ( 0.00,-1.50) node[below]{$Y_5$};
 \draw[thick,color=black!80] ( 1.06, 1.06) node[above]{$Y_6$};
 \draw[thick,color=black!80] ( 0.28, 0.28) node[below]{$Y_7$};
 \draw[thick,color=black!80] ( 0.00,-0.40) node[above]{$Y_8$};
%%%%%%%%%%%%%%%%%%%%%%%%%%%%%%%%%%%%%%%%%%%%
 \draw[color=red, line width=0.9pt] ( -1.06, 1.06) to[out=-105,in=130] ( 0.00,-1.50);
 \draw ( -0.87, 0.09) node[color=red]{$c_1$};
 \draw[color=red, line width=0.9pt] ( -0.57, 1.38) to ( 1.06, 1.06);
 \draw ( 0.30, 1.22) node[color=red]{$c$};
 \draw[color=red, line width=0.9pt] ( -0.57, 1.38) to ( 0.28, 0.28);
 \draw ( -0.15, 0.83) node[color=red]{$c'$};
 \draw[color=red, line width=0.9pt] ( -0.57, 1.38) to[out=-105,in=-180] ( 0.00,-0.40);
 \draw ( -0.57, 0.38) node[color=red]{$c''$};
% an example of permissible curve %
%%%%%%%%%%%%%%%%%%%%%%%%%%%%%%%%%%%%%%%%%%%%
\filldraw[color=black!20] (+5,0) circle (0.4);
 \draw[line width=1pt] (0+5,0) circle (1.5);
 \draw[line width=1pt] (0+5,0) circle (0.4);
 \draw ( 0.00+5, 1.50) to[out=-135, in=90] (-1.00+5,-0.20) to[out=-90, in=-90] ( 0.00+5, -0.40);
 \draw ( 0.00+5, 1.50) to[out=-45,  in=90] ( 1.00+5,-0.20) to[out=-90, in=-90] ( 0.00+5, -0.40);
 \filldraw ( 0.00+5, 1.50) circle (0.05); \draw[thick,color=black!80] ( 0.00+5, 1.50) node[above]{$Y_1$};
 \filldraw ( 0.00+5,-0.40) circle (0.05); \draw[thick,color=black!80] ( 0.00+5,-0.40) node[above]{$Y_2$};
 \draw[color=red, line width=0.9pt] (0+5,0) circle (0.7); \draw[thick,color=red] ( 0.00+5,-0.50) node[below]{$c$};
%%%%%%%%%%%%%%%%%%%%%%%%%%%%%%%%%%%%%%
 \filldraw[color=red]   ( 0.57, 1.38) circle (0.05); \filldraw[color=white] ( 0.57, 1.38) circle (0.03);
 \filldraw[color=red]   (-0.57, 1.38) circle (0.05); \filldraw[color=white] (-0.57, 1.38) circle (0.03);
%%%%%%%%%%%%%%%%%%%%%%%%%%%%%%%%%%%%%%
\draw ( 0.00,-1.75) node[below]{$\mathrm{I}$};
\draw ( 0.00+5,-1.75) node[below]{$\mathrm{II}$};
\end{tikzpicture}
\caption{Some examples of permissible (closed) curves in marked ribbon surfaces}
\label{exp. of perm. c.}
\end{center}
\end{figure}
\item[(4)] A {\it permissible closed curve} is a closed curve $c$ satisfying condition (3)(b)
(as shown Picture II in \Pic \ref{exp. of perm. c.}, the algebra of the marked ribbon surface
is a 2-Kronecker algebra, and the curve $c$ is a permissible closed curve).
\end{enumerate}
\end{definition}

\begin{definition}\rm (Equivalence class of permissible curves) \cite{BS2018} \label{def. p.c. eq. cls.}
Let $\Surf=(\Surf,\M,\Gamma)$ be a marked ribbon surface.
Two permissible curves $c, c' : (0,1) \to \Surf\backslash \partial \Surf $ in $\Surf$ are called {\it equivalent},
denoted by $c \simeq c'$, if one of the following conditions is satisfied.
\begin{enumerate}
\item[(1)] $c'$ is the {\it reverse curve} $c^-$ of $c$, that is, $c(x) = c'(1-x)$ for any $x\in (0,1)$.
\item[(2)] There is a sequence of consecutive arcs $(u_i : (0,1) \to \Surf\backslash \partial \Surf)_{1\le i\le n}$
in the arc system $\Gamma$ such that
\begin{enumerate}
\item All $u_i$ $(1\le i\le n)$ are sides of the same elementary polygon $\Delta_j$.
\item $c(0^+)$ (respectively, $c(1^-)$) is a marked point $\bullet$ or an extra point $\rbullet$,
$c(1^-) = u_1(0^+)$ (respectively, $c(0^+) = u_1(0^+)$), $u_{i-1}(1^-)=u_i(0^+)$ for any $1\le i\le n$.
\item $c$ is homotopic to the concatenation of $c'$ and $(u_i)_{1\le i\le n}$.
\item $c'$ starts at an endpoint of $u_n$, and the first crossing with $\Gamma$ of $c$ and $c'$ are with the same side of $\Delta_j$.
\end{enumerate}
\item[(3)] $c(0^+) = c'(0^+)$ and $c(1^-) = c'(1^-)$ (respectively, $c(0^+) = c'(1^-)$ and $c(1^-) = c'(0^+)$) are marked points or extra points,
arcs consecutively crossed by $c$ and $c'$ ($c^-$ and $c'$) are same, and each segment of $c$ cut by two arcs is homotopic to that of $c'$.
\item[(4)] $c$ (respectively, $c'$) is either a curve satisfying $c\cap\Gamma=\varnothing$ (respectively, $c'\cap\Gamma=\varnothing$) or a arc;
in this case, we call $c, c'$ are {\it trivial}. Thus each arc can be regarded as a trivial permissible curve.
\end{enumerate}
\end{definition}

\begin{example}\rm
In \Pic \ref{exp. of perm. c.} Picture I, the curves $c$, $c'$ and $c''$ are permissible and
$c \simeq c' \simeq c''$, but $c \not\simeq c_1$.
\end{example}

\section{\bf Finitely generated module categories}
%The descriptions of indecomposable modules and Auslander-Reiten quiver
%in marked ribbon surface of gentle algebra are given in \cite{APS2019},
%the Auslander-Reiten translate $\tau$ can be described by the rotation of permissible curve.

\subsection{The geometric models of string and band modules}
Let $A=kQ/I$ be a gentle algebra.
For each arrow $\alpha \in Q_1$, the {\it formal inverse} of $\alpha$, denoted by $\alpha^{-}$,
is an arrow with $s(\alpha^{-})=t(\alpha)$ and $t(\alpha^{-1})=s(\alpha)$. Obviously, $(\alpha^{-})^{-}=\alpha$.
We use $Q_1^{-}$ to denote the set of all formal inverses of arrows in $Q_1$.
Also we can define that the formal inverse of any path $p=\alpha_1\cdots\alpha_n$ is $p^{-}=\alpha_n^{-}\cdots\alpha_1^{-}$.
In this subsection, we recall the definitions of string modules and band modules over a gentle algebra $A$ \cite{BR87},
and review the description of indecomposable modules in marked ribbon surfaces \cite{BS2018}.

A {\it string} $s$ is a reduced walk in the quiver $Q$ with no relations,
that is, $s=a_1\cdots a_n$ with $a_i\in Q_1\cup Q_1^{-}$ has neither subwalks of the
form $aa^{-}$ and $a^{-}a$ nor subwalks of the form $ab$ such that $ab\in I$ or $b^{-}a^{-}\in I$;
particularly, $s=1_v$ is a path of length $0$ corresponding to a vertex $v\in Q_0$, it is called a {\it simple string}
(note that $1_v^{-}=1_v$ for each $v\in Q_0$). Trivially, we define the {\it trivial string} $s=0$.

Each string $s=a_1\cdots a_n$ in $A$ corresponds to a module $M(s)\in \bmod A$ as follows:
\begin{itemize}
\item replacing each vertex of $s$ by a copy of the field $k$, that is,
for each $v\in Q_0$, $\dim_k M(s) 1_v = 1$ if $1_v$ is a trivial subwalk of $s$ or zero if otherwise; and
\item the action of an arrow $\alpha\in Q_1$ on $M(s)$ is the identity morphism if $\alpha$
is an arrow of $s$ or zero if otherwise.
\end{itemize}
Obviously, $s=0$ yields $M(s)=0$; $s$ being simple yields that $M(s)$ is simple;
and $M(s)\cong M(s')$ if and only if $s'=s^{-}$ or $s'=s$.

A {\it band} $b=a_1\cdots a_n$ is a cyclic string (that is, $t(a_n)=s(a_1)$) such that each power $b^n$
is a string but $b$ is not a proper power of any string.

Each band $b=a_1\cdots a_n$ in $A$ corresponds to a family of modules $M(b,\varphi)\in \bmod A$ as follows:
\begin{itemize}
\item $\varphi$ is an indecomposable $k$-linear automorphism of $k^n$ with $n\geq 0$
(that is, $\varphi \in \mathrm{Aut}(k^n)$ is a Jordan block);
\item replacing each vertex of $s$ by a copy of the vector space $k^n$; and
\item the action of an arrow $\alpha\in Q_1$ on $M(b,\varphi)$ is the identity morphism
if $\alpha=a_i$ for $1\le i\le n-1$ or $\varphi$ if $\alpha=a_n$ (thus $M(b,\varphi)\alpha = 0$
if $\alpha$ is not an arrow of $b$).
\end{itemize}

\begin{remark}\rm
By \cite{BR87}, we know that each indecomposable $A$-module is either a string module or a band module.
Baur and Coelho Sim\~oes \cite{BS2018} showed that there is a bijection between the equivalence classes of
non-trivial permissible curves in $\Surf_A$ and non-zero strings of $A$,
and there is a bijection between the homotopy classes of permissible closed curves $c$ in $\Surf_A$
with $|I_{\Gamma_A}(c)| \ge 2$ and the powers of bands of $A$, where $|I_{\Gamma_A}(c)|$ is
the number of intersection points of $c$ and $\Gamma_A$ \cite{BS2018}. To be precise:
\begin{itemize}
\item If $[c]$ is an equivalence class of non-trivial permissible curve,
then $c=c_1 \cdots c_r$, where each $c_i$ is a segment cut by two arcs $\gamma_{i-1}$
and $\gamma_{i}$ (we suppose that $c$ consecutively crosses $\gamma_1, \cdots, \gamma_{r-1}$),
and $c$ corresponds to such string $s=s_1 \cdots s_r$ that $s_i$ is one of:
\begin{itemize}
\item the arrow $\gamma_{i-1} \to \gamma_{i}$, if $c$ counterclockwise crosses
$\gamma_{i-1}$ and $\gamma_{i}$ around their common endpoint $p(\gamma_{i-1}, \gamma_i, c)$;
\item the formal inversion of the arrow $\gamma_{i} \to \gamma_{i-1}$, if $c$ clockwise crosses
$\gamma_{i-1}$ and $\gamma_{i}$ around their common endpoint $p(\gamma_{i-1}, \gamma_i, c)$.

\noindent (Note that each arc in $\Gamma$ can be regarded as a vertex of quiver $Q$ by
Definition \ref{def. alg. of m. r. s.}).
\end{itemize}
Thus $c$ can be corresponded to the string module of $s$.
\item If $c$ is a homotopy class of permissible closed curves, then $c$ with an indecomposable
invertible linear transformation $\varphi:k^n\to k^n$ corresponds to a band which induces a band module $M(c, \varphi)$.
Band modules lie in homogeneous tubes of the Auslander-Reiten quiver, thus the Auslander-Reiten translate
$\tau$ acts on them as the identity morphism.
\end{itemize}
\end{remark}

\begin{remark} \label{rmk. in BS} \rm \cite[Theorem 3.8 and Proposition 3.9]{BS2018}
Let $[c]$ (respectively, $\mathrm{htp}(c)$) be the equivalence class of permissible curve
(respectively, the homotopy class of permissible closed curve) $c$. Define the following three sets:
\begin{itemize}
\item $\Eqcls(\Surf_A)= \{[c]\mid c \text{ is a permissible curve}\}$;
\item $\Htpcls(\Surf_A)= \{(\mathrm{htp}(c), \varphi)\mid c \text{ is a permissible closed curve with }
                  |I_{\Gamma_A}(c)| \ge 2 \ \text{and}\ \varphi\ \text{a Jordan block}\}$;
\item $\mathfrak{P}(\Surf_A)= \Eqcls(\Surf_A) \cup \Htpcls(\Surf_A)$.
\end{itemize}
Then we have a bijection
\[M: \mathfrak{P}(\Surf_A) \to  \{\text{string modules}\}\cup \{\text{band modules}\} =
\mathrm{ind}(\bmod A), x\mapsto M(x).\]
Obviously, each trivial permissible curve corresponds to zero by $M$.
\end{remark}

\begin{definition} \rm (Formal direct sums) \label{def. d. s.}
Let $\Surf_A$ be the marked ribbon surface of a gentle algebra $A$.
A {\it formal direct sum} of $c_1, \cdots, c_n$, denoted by $\bigoplus_{i=1}^n c^i$,
%is the pair $(c_1,\cdots, c_n)$ which
naturally corresponds to the direct sum of modules $\bigoplus_{i=1}^n M(c^i)$
by the corresponding $M$ in Remark \ref{rmk. in BS}.
\end{definition}

\subsection{The geometric models of some short exact sequences}
In this subsection, we depict certain short exact sequences by marked ribbon surfaces for gentle algebras,
which are used to describe the minimal projective and injective resolutions of simple modules in Section \ref{Section gl.dim},
and describe the minimal projective resolutions of injective modules in Section \ref{Section Gproj}.
First at all, we fix some terms and notations.

We always assume that each non-trivial permissible curve
$c: (0,1) \to \Surf\backslash\partial\Surf$ in $\Surf=(\Surf,\M,\Gamma)$ consecutively crossing arcs
$\gamma_1, \cdots, \gamma_{r-1}$ satisfies that $\gamma_1$ is the first arc crossed by $c$,
that is, there is a sequence $0=x_0<x_1<x_2\cdots <x_r = 1$ such that
$c = c_1\cup c_2 \cup \cdots \cup c_r$ (we write $c=c_1c_2\cdots c_r$ for simplicity),
where each $c_i=\{c(x)\mid x_{i-1}\le x\le x_i\}$ is a segment of $c$ satisfying
$c_i\cap \gamma_i=c(x_i) = c_{i+1}\cap \gamma_i $ if $2\le i\le r-1$;
and $c_1=\{c(x)\mid 0< x\le x_1\}$ and $c_r=\{c(x)\mid x_{r-1}\le x< 1\}$ are the end segments of
$c$ satisfying $c_1\cap\gamma_1 = c(x_1) = c_2\cap\gamma_1$ and
$c_r\cap\gamma_{r-1}=c(x_{r-1})=c_{r-1}\cap\gamma_{r-1}$, respectively.

Let $c=c_1c_2\cdots c_r$ be a permissible curve consecutively crossing $\gamma_1, \cdots, \gamma_{r-1}$.
Then it induces a new permissible curve $c'=c^{\mathfrak{T}}_{i-1} c_{i}\cdots c_{j} c^{\mathfrak{T}}_{j+1}$
for any $2\le i \le j \le r-1$ such that
\begin{enumerate}
\item[(1)] $c^{\mathfrak{T}}_{i-1}$ and $c_{i-1}$ lie in the inner of the same elementary polygon of $\Surf$;
\item[(2)] one endpoint of $c^{\mathfrak{T}}_{i-1}$ is $c_{i-1}\cap\gamma_{i-1}$ and the other is an endpoint of $\gamma_{i-2}$;
\item[(3)] the dual of the condition (1), that is, $c^\mathfrak{T}_{j+1}$ and $c_{j+1}$ lie in the inner of the same elementary polygon;
\item[(4)] the dual of the condition (2), that is, one endpoint of $c^{\mathfrak{T}}_{j+1}$ is
$c_{j+1}\cap \gamma_j$ and the other is an endpoint of $\gamma_{j+1}$.
\end{enumerate}
Note that it is a trivial case for $c_1^{\mathfrak{T}}c_2\cdots c_{r-1}c_r^{\mathfrak{T}}=c_1c_2\cdots c_{r-1}c_r$,
that is, $c_1^{\mathfrak{T}}=c_1$, and $c_r^{\mathfrak{T}} =c_r$.
We provide an example in the \Pic\ref{exp. permissible curve}, in which
$c = \cdots c_{i-1}c_{i}c_{i+1}\cdots c_{j-1}c_{j}c_{j+1}$ and
$c' = c_{i}^{\mathfrak{T}}c_{i+1}\cdots c_{j-1}c_{j}^{\mathfrak{T}}$ are the red and blue permissible curves respectively.
In this example, $c$ corresponds to the string
\[ \gamma_1 \longrightarrow \cdots -\!\!\!-\!\!\!-
\gamma_{i-1} \longrightarrow \gamma_i \longrightarrow \cdots \longrightarrow \gamma_j \longrightarrow \gamma_{j+1}
-\!\!\!-\!\!\!- \cdots \longleftarrow \gamma_{r-1},\]
and $c'$ corresponds to the string
\[ \gamma_i \longrightarrow \cdots \longrightarrow \gamma_j .\]

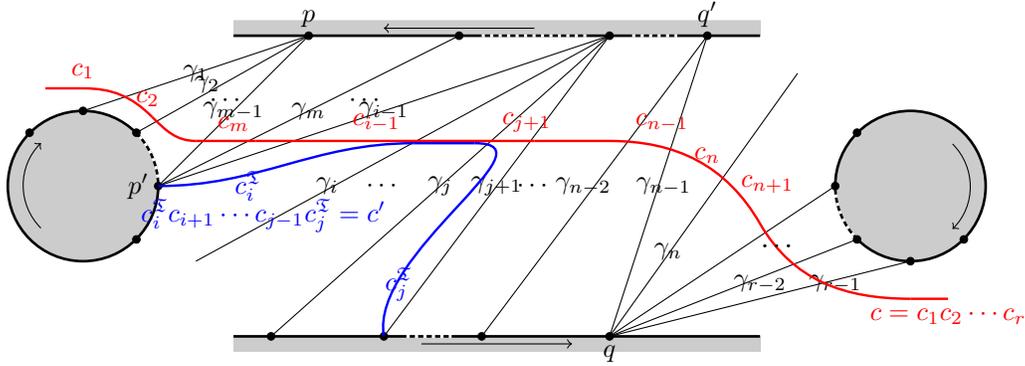
\begin{figure}[htbp]
\begin{center}
{% \tiny
\begin{tikzpicture}% [scale=0.8]
% boundary
\filldraw[color=black!20] (0,0) circle (1.0);
\draw[line width =1pt] (0,0) circle (1.0); \draw[->] (-0.56, -0.56) arc (225:135:0.8);
\filldraw[color=black!20] (11,0) circle (1.0);
\draw[line width =1pt] (11,0) circle (1.0); \draw[->] (0.56+11, 0.56) arc (45:-45:0.8);
\filldraw[color=black!20] (2.0, 2.0) to (9.0, 2.0) to (9.0, 2.2) to (2.0, 2.2);
\filldraw[color=black!20] (2.0,-2.0) to (9.0,-2.0) to (9.0,-2.2) to (2.0,-2.2);
\draw[line width =1pt] (2.0, 2.0) to (9.0, 2.0); \draw[->] (6.0, 2.1) to (4.0, 2.1);
\draw[line width =1pt] (2.0,-2.0) to (9.0,-2.0); \draw[->] (4.5,-2.1) to (6.5,-2.1);
\draw[line width =1.2pt, dotted, color=white] (1,0) arc (0:45:1);
\draw[line width =1.2pt, dotted, color=white] (10,0) arc (180:225:1);
\draw[line width =1.2pt, dotted, color=white] (5.3, 2.0) to (6.7, 2.0);
\draw[line width =1.2pt, dotted, color=white] (7.3, 2.0) to (7.9, 2.0);
\draw[line width =1.2pt, dotted, color=white] (4.3,-2.0) to (4.9,-2.0);
%%%%%%%%%%%%%%%%%%%%%%%%%%%%
%% marked points and arcs %%
%%%%%%%%%%%%%%%%%%%%%%%%%%%%
% p, p', q and q'
\draw (3.00, 2.00) node[above]{$p$};
\draw (1.00, 0.00) node[left]{$p'$};
\draw (8.30, 2.00) node[above]{$q'$};
\draw (7.00,-2.00) node[below]{$q$};
% left marked points
\filldraw ( 1.00, 0.00) circle (0.05); % p'
\filldraw ( 0.71, 0.71) circle (0.05);
\filldraw ( 0.00, 1.00) circle (0.05);
\filldraw (-0.71, 0.71) circle (0.05);
%\filldraw (-1.00, 0.00) circle (0.05);
%\filldraw (-0.71,-0.71) circle (0.05);
%\filldraw ( 0.00,-1.00) circle (0.05);
\filldraw ( 0.71,-0.71) circle (0.05);
% right marked points
%\filldraw ( 1.00+11, 0.00) circle (0.05);
%\filldraw ( 0.71+11, 0.71) circle (0.05);
%\filldraw ( 0.00+11, 1.00) circle (0.05);
\filldraw (-0.71+11, 0.71) circle (0.05);
\filldraw (-1.00+11, 0.00) circle (0.05);
\filldraw (-0.71+11,-0.71) circle (0.05);
\filldraw ( 0.00+11,-1.00) circle (0.05);
\filldraw ( 0.71+11,-0.71) circle (0.05);
% above marked points
\filldraw (3.00, 2.00) circle (0.05); % p
\filldraw (5.00, 2.00) circle (0.05);
\filldraw (7.00, 2.00) circle (0.05);
\filldraw (8.30, 2.00) circle (0.05); % q'
% bellow marked points
\filldraw (2.50,-2.00) circle (0.05);
\filldraw (4.00,-2.00) circle (0.05);
\filldraw (5.30,-2.00) circle (0.05);
\filldraw (7.00,-2.00) circle (0.05); % q
% arcs form left to above
\draw (1.00, 0.00) to (3.00, 2.00); \draw (1.92, 1.14) node{$\cdots$}; % p' to p
                                    \draw (2.00, 1.00) node{$\gamma_{m-1}$};
\draw (0.71, 0.71) to (3.00, 2.00); \draw (1.65, 1.35) node{$\gamma_2$};
\draw (0.00, 1.00) to (3.00, 2.00); \draw (1.50, 1.50) node{$\gamma_1$};
\draw (1.00, 0.00) to (5.00, 2.00); \draw (3.00, 1.00) node{$\gamma_m$};
\draw (1.00, 0.00) to (7.00, 2.00); \draw (4.00, 1.00) node{$\gamma_{i-1}$};
                                    \draw (3.77, 1.14) node{$\cdots$};
% arcs form above to below
\draw (7.00, 2.00) to (1.50,-1.00); \draw (3.25, 0.00) node{$\gamma_{i}$};
\draw (7.00, 2.00) to (2.50,-2.00); \draw (4.75, 0.00) node{$\gamma_{j}$};
                                    \draw (4.00, 0.00) node{$\cdots$};
\draw (7.00, 2.00) to (4.00,-2.00); \draw (5.50, 0.00) node{$\gamma_{j+1}$};
\draw (8.30, 2.00) to (5.30,-2.00); \draw (6.65, 0.00) node{$\gamma_{n-2}$};
                                    \draw (6.00, 0.00) node{$\cdots$};
\draw (8.30, 2.00) to (7.00,-2.00); \draw (7.72, 0.00) node{$\gamma_{n-1}$}; % q' to q
\draw (7.00,-2.00) to (9.50, 1.50); \draw (9.25,-0.80) node{$\cdots$}; \draw (7.78,-0.87) node{$\gamma_{n}$};
\draw (7.00,-2.00) to (-1.00+11, 0.00);
\draw (7.00,-2.00) to (-0.71+11,-0.71);  \draw (9.00,-1.30) node{$\gamma_{r-2}$};
\draw (7.00,-2.00) to ( 0.00+11,-1.00);  \draw (10.00,-1.30) node{$\gamma_{r-1}$};
% permissible curve (red)
\draw[red, line width=0.9pt] (-0.5,1.3) to (0,1.3);
\draw[red, line width=0.9pt] (0,1.3) to[out=0, in=135] (0.919239, 0.919239);
\draw[red, line width=0.9pt] (0.919239, 0.919239) to[out=-45, in=180] (1.5, 0.6);
\draw[red, line width=0.9pt] (1.5, 0.6) to (2.3, 0.6);
\draw[red, line width=0.9pt] (2.3, 0.6) to (2.8, 0.6);
\draw[red, line width=0.9pt] (2.8, 0.6) to (6.0, 0.6);
\draw[red, line width=0.9pt] (6.0, 0.6) to (7.0, 0.6);
\draw[red, line width=0.9pt] (7.0, 0.6) to[out=0, in=120] (9.0,-0.5) to[out=-60, in=180] (11,-1.5);
\draw (9.25,-0.80) node{$\cdots$}; %%% repeat
\draw[red, line width=0.9pt] (11,-1.5) to (11.5, -1.5); \draw[red] (11.5, -1.5) node[below]{$c=c_1c_2\cdots c_r$};
\draw[red] (0, 1.3) node[above]{$c_1$};
\draw[red] (0.86, 1.16) node{$c_2$};
\draw[red] (2.00, 0.60) node[above]{$c_m$};
\draw[red] (3.90, 0.60) node[above]{$c_{i-1}$};
\draw[red] (5.90, 0.60) node[above]{$c_{j+1}$};
\draw[red] (7.70, 0.60) node[above]{$c_{n-1}$};
\draw[red] (8.30, 0.41) node{$c_{n}$};
\draw[red] (9.10, 0.02) node{$c_{n+1}$};
% permissible curve (blue)
\draw[blue, line width=0.9pt] (1.00, 0.00) to[out=0, in=180] (4.40, 0.57);
\draw[blue, line width=0.9pt] (2.20, 0.00) node{$c_i^{\mathfrak{T}}$};
\draw[blue, line width=0.9pt] (4.40, 0.57) to (5.21, 0.57);
\draw[blue, line width=0.9pt] (5.21, 0.57) to[out=0, in=100] (4.00,-2.00);
\draw[blue] (4.20,-1.30) node{$c_{j}^{\mathfrak{T}}$};
\draw[blue] (2.40,-0.05) node[below]{$c_i^{\mathfrak{T}}c_{i+1}\cdots c_{j-1}c_{j}^{\mathfrak{T}}=c'$};
\end{tikzpicture}
} % tiny end
  \caption{Two examples of permissible curves corresponding to strings.}
  \label{exp. permissible curve}
  \end{center}
\end{figure}

A string module which corresponds to the permissible curves $c$ (respectively, its equivalence class $[c]$)
is denoted by $M(c)$ (respectively, $M([c])$);
a band module with a Jordan block $\varphi \in \mathrm{Aut}(k^n)$ which corresponds to the permissible closed curve $c$
(respectively, its homotopic class $\mathrm{htp}(c)$) is denoted by $M(c, \varphi)$ (respectively, $M(\mathrm{htp}(c),\varphi)$).

For the convenience of describing rotations, we suppose that the positive direction of
the boundary $\partial\Surf$ of the surface $\Surf$ is the following walking direction:
walking along $\partial\Surf$, the inner of $\Surf$ is on the left.

\begin{definition} \rm (Rotations of Type 1) \label{def-3.4}
Let $c^1=c_1\cdots c_r$ be a permissible curve consecutively crossing $\gamma_1, \cdots, \gamma_{r-1}$,
and let $c^2$ be equivalent to $c'c_{i+1}\cdots c_r$ (in this case $c^1(1^-) = c^2(1^-)$),
where $c'$ is either $c_i^{\mathfrak{T}}$ or a curve with an endpoint being an extra point.
We say that the permissible curve $c^3$, denoted by $\mathrm{prot}_{c^1}(c^2)$,
is obtained by the {\it positive rotating $c^2$} with respect to $c^1$ if it is given as follows:
\begin{enumerate}
\item[Step 1.] move the endpoint of $c^2$ which is the common endpoint of $c^2$ and $c^1$ to the other of $c^1$;
\item[Step 2.] following the positive direction of the boundary,
move the other endpoint of $c^2$ to the vertex $p$ of $\Delta_j$, where
\begin{enumerate}
  \item[-] $\Delta_j$ is the elementary polygon such that $c_i$ lies in the inner of $\Delta_j$
    (in this case, $\gamma_{i-1}, \gamma_{i} \in \mathfrak{S}(\Delta_j)$),
  \item[-] and $p$ is a vertex such that $c^3 = \mathrm{prot}_{c^1}(c^2)$ crosses $\gamma_{i-1}$.
\end{enumerate}
\end{enumerate}
Dually, we can define $\mathrm{nrot}_{c^1}(c^2)$ obtained by the {\it negative rotating} $c^2$ with respect to $c^1$ (see \Pic \ref{fig:def-3.4}).
\end{definition}

\begin{figure}[htbp]
\begin{center}
\begin{tikzpicture}%[xscale=0.8]
\draw[black!25] (2+0.07,-1-0.07) -- (3+0.07,0-0.07) [line width =5pt];
\fill (0,2) circle (2.2pt);
\draw (0,2) -- (-2,0);
\draw (0,2) -- (2,-1); \fill (2,-1) circle (2.2pt); \draw(1.2, 0.2) node[below]{$\gamma_{i-1}$};
\draw (0,2) -- (3,0);  \fill (3,0) circle (2.2pt); \draw(2, 0.5) node{$\gamma_i$};
\draw[dotted] (2,-1) -- (3,0) [line width =1.2pt];
% c^1
\draw[line width =1.2pt] (-2.5,1) node{$\pmb{\times}$};
\draw[line width =1.2pt] ( 3.0,1) node{$\pmb{\times}$};
\draw[red][line width =1.2pt][->] (-2.5,1) -- (0,1) node[above]{$c^1$};
\draw[red][line width =1.2pt] (0,1) -- (3,1);
% c^2
\draw[blue][line width =1.2pt] (3,1) -- (2,-1);
\draw[blue] (2,-0.5) node{$c^2$};
% c^3
\draw[orange][line width =1.2pt] (-2.5,1) -- (3,0) node[right]{\color{black}$p$};
\draw[orange](0.25,0.5) node[below]{$c^3$};
\end{tikzpicture}
\ \ \ \
\begin{tikzpicture}%[scale=0.7]
\draw[black!25] (1,-1.2-0.1) -- (-1,-1-0.1) [line width = 7pt];
\draw[black!25] (1.5,-1.25-0.1) -- (3,-1.25-0.1) [line width = 7pt];
\fill (0,2) circle (2.2pt);
\draw (0,2) -- (-2,0);
\draw (0,2) -- (-1,-1); \fill (-1,-1) circle (2.2pt); \draw(-0.6,0) node[left]{$\gamma_{i-1}$};
\draw (0,2) -- (3,0);  \fill (3,0) circle (2.2pt); \draw(2, 0.5) node{$\gamma_i$};
\draw (3,0) to[out=190,in=50] (1,-1.2); \fill (1,-1.2) circle (2.2pt);
\draw (1,-1.2) -- (-1,-1) [line width = 1.2pt] [dotted];
\draw (1.5,-1.25) to[out=45,in=135] (3,-1.25);
\draw (1.5,-1.25) -- (3,-1.25) [line width = 1.2pt];
\fill[black] (1.5,-1.25) circle (2.2pt); \fill[black] (3,-1.25) circle (2.2pt);
\fill[red] (2.25,-1.25) circle (2.8pt);  \fill[white] (2.25,-1.25) circle (1.8pt);
% \draw (1,-1.2) to[out=35,in=125] (3,-1.25) [dotted];
\draw (2.25,-0.5) node{$\vdots$};
% c^1
\draw[line width =1.2pt] (-2.5,1) node{$\pmb{\times}$};
\draw[line width =1.2pt] ( 3.0,1) node{$\pmb{\times}$};
\draw[red][line width =1.2pt][->] (-2.5,1) -- (0.5,1) node[above]{$c^1$};
\draw[red][line width =1.2pt] (0.5,1) -- (3,1);
% c^2
\draw[blue][line width =1.2pt] (3,1) node[below]{$c^2$} to[out=235,in=80] (2.25,-1.25);
% c^3
\draw[orange][line width =1.2pt] (-2.5,1) -- (3,0) node[right]{\color{black}$p$};
\draw[orange](0.25,0.5) node[below]{$c^3$};
\end{tikzpicture}
\caption{The point ``$\pmb{\times}$'' is either a marked point or an extra point. }
\label{fig:def-3.4}
\end{center}
\end{figure}

\begin{definition} \rm (Rotations of Type 2) \label{def-3.5}
Let $c^1=c_1\cdots c_r$ be a permissible curve consecutively crossing $\gamma_1, \cdots, \gamma_{r-1}$,
and $c^2$ equivalent to either $c_{m}^{\mathfrak{T}}c_{m+1}\cdots c_nc_{n+1}^{\mathfrak{T}}$ ($2\le m\le n+1 \le r-2$)
such that $c^1$ and $c^2$ have an intersection, say $x$, in $\Surf\backslash\partial\Surf$.
Then $c^2$ is divided into two parts $c^2_{\I}$ and $c^2_{\II}$ where $c^2_{\I}(1^-)=x=c^2_{\II}(0^+)$.
We say that the permissible curve $c^3_{\I}$ (respectively, $c^3_{\II}$), denoted by
$\mathrm{prot}_{c^1}(c^2_{\I})$ (respectively, $\mathrm{prot}_{c^1}(c^2_{\II})$), is obtained by
the {\it positive rotating} $c^2_{\I}$ (respectively, $c^2_{\II}$) with respect to $c^1$ if it is given as follows:
\begin{enumerate}
\item[Step 1.] move the endpoint $x$ of $c^2_{\I}$ (respectively, $c^2_{\II}$) to the end point $c^1(0^+)$ (respectively, $c^1(1^-)$);
\item[Step 2.] following the positive direction of boundary, move the endpoint $c^2_{\I}(0^+)$ of $c^2_{\I}$
(respectively, the endpoint $c^2_{\II}(1^-)$ of $c^2_{\II}$)
to the next vertex of $\Delta_j$, where $\Delta_j$ is an elementary polygon such that $c_m$
(respectively, $c_{n+1}$) lies in the inner of $\Delta_j$
(in this case, we have $\gamma_{m-1}, \gamma_{m} \in \mathfrak{S}(\Delta_j)$ (resp., $\gamma_{n}, \gamma_{n+1} \in \mathfrak{S}(\Delta_j)$)).
\end{enumerate}
Dually, we can define $\mathrm{nrot}_{c^1}(c^2_{\I})$ and $\mathrm{nrot}_{c^1}(c^2_{\II})$ obtained by
the {\it negative rotating $c^2_{\I}$ and $c^2_{\II}$} with respect to $c^1$, respectively (see \Pic \ref{fig:def-3.5}).
\end{definition}

\begin{figure}[htbp]
\begin{center}
\begin{tikzpicture}[yscale=1.25]
\draw (-3,-1) -- (-1, 1) -- (-1,-1);
\fill (-3,-1) circle (2.2pt); \fill (-1, 1) circle (2.2pt); \fill (-1,-1) circle (2.2pt);
\draw ( 3, 1) -- ( 1,-1) -- ( 1, 1);
\fill ( 3, 1) circle (2.2pt); \fill ( 1,-1) circle (2.2pt); \fill ( 1, 1) circle (2.2pt);
% c^1
\draw[line width =1.2pt] (-3.5,0) node{$\pmb{\times}$};
\draw[line width =1.2pt] ( 3.5,0) node{$\pmb{\times}$};
\draw[red][line width =1.2pt][->] (-3.5,0) -- (2.5,0) node[below]{$c^1$};
\draw[red][line width =1.2pt] (2.5,0) -- (3.5,0);
% c^2
\draw[blue][line width =1.2pt] (-3,-1) to[out=40,in=220] ( 3, 1);
\draw (0,0) circle (2.2pt) node[below]{$x$};
\draw[blue] (-2.5,-1) node{$c^2_{\mathrm{I}}$};
\draw[blue] ( 2.5, 1) node{$c^2_{\mathrm{II}}$};
% c^3
\draw[orange][line width =1.2pt] (-3.5,0) to[out=-10,in=90] (-1,-1);
\draw[orange][line width =1.2pt] ( 3.5,0) to[out=170,in=-90] ( 1, 1);
\draw[orange] (-1,-1) node[left]{$c^3_{\mathrm{I}}$};
\draw[orange] ( 1, 1) node[right]{$c^3_{\mathrm{II}}$};
\draw (-1,0.6) node[right]{$\gamma_{m}$};
\draw (-1.4,0.6) node[left]{$\gamma_{m-1}$};
\draw (1,-0.6) node[left]{$\gamma_{n}$};
\draw (1.4,-0.6) node[right]{$\gamma_{n+1}$};
\end{tikzpicture}
\caption{The point ``$\pmb{\times}$'' is either a marked point or an extra point. }
\label{fig:def-3.5}
\end{center}
\end{figure}

\begin{proposition} \label{prop. s. e. s.}
Let $A=kQ/I$ be a gentle algebra and $\Surf_A = (\Surf_A, \M_A, \Gamma_A)$ its marked ribbon surface.
If there are three permissible curves $c$, $c'$, $c''$ such that the following conditions are satisfied
{\rm(}see \Pic \ref{Pic. of prop. s. e. s.}{\rm)}:
\begin{enumerate}
\item[{\rm(1)}] $c'$ and $c$ have a common endpoint $m_s= c(0^+)$,
        $c''$ and $c$ have another common endpoint $m_t= c(1^-)$.
\item[{\rm(2)}] $c$ consecutively crosses arcs $\gamma'$ and $\gamma''$ such that $p(\gamma', \gamma'', c)$ is on the right of $c$;
        $c'$ crosses $\gamma'$ such that $c'$ and $\gamma''$ have a common endpoint which is not $m_t$; and
        $c''$ crosses $\gamma''$ such that $c''$ and $\gamma'$ have a common endpoint which is not $m_s$.
\item[{\rm(3)}] $c''\simeq \mathrm{prot}_c(c')$ $($equivalently, $c' \simeq \mathrm{nrot}_c(c''))$
    {\rm(}Definition \ref{def-3.4}{\rm)};.
\end{enumerate}
Then we have a short exact sequence
\[0 \longrightarrow M(c') \longrightarrow M(c) \longrightarrow M(c'') \longrightarrow 0.\]
\end{proposition}

\begin{figure}[htbp]
\begin{center}
\begin{tikzpicture}[scale=0.7]
\filldraw [color=black!20] (-2,1)-- (-2,-1)--(-2.1,-1)--(-2.1,1);
\filldraw [color=black!20] (2,1)-- (2,-1)--(2.1,-1)--(2.1,1);
\filldraw [color=black!20] (-1,2)--(1,2)--(1,2.1)--(-1,2.1);
\filldraw [color=black!20] (-1,-2)--(1,-2)--(1,-2.1)--(-1,-2.1);
\draw [line width=1.2pt] (-2,1)-- (-2,-1);
\draw [line width=1.2pt] (2,1)-- (2,-1);
\draw [line width=1.2pt] (-1,2)-- (1,2);
\draw [line width=1.2pt] (-1,-2)-- (1,-2);
\draw [line width=12pt] [white] (-0.2,2) -- (0.2,2) ;
\draw (0,-2)-- (-0.5,2);
\draw (0,-2)-- (0.5,2);
\begin{scriptsize}
\fill [color=black] (0,-2) circle (2.2pt);
\fill [color=black] (-0.5,2) circle (2.2pt);
\fill [color=black] (0.5,2) circle (2.2pt);
\fill [color=black] (-2,-0.5) circle (2.2pt);
\fill [color=black] (2,0.5) circle (2.2pt);
\end{scriptsize}
\draw[red] (-2,-0.5) to[out=0, in=180] (2,0.5); \draw[->][red] (1.79,0.5)--(1.8,0.5); \draw[red] (-1,-0.5) node {$c$};
\draw[blue] (-2,-0.5) to (0.5,2); \draw[blue] (-1,0.7) node {$c'$};
\draw[orange] (2,0.5) to (-0.5,2); \draw[orange] (1,1.3) node {$c''$};
\draw (-0.4,0.4) node{$\gamma'$}; \draw (0.4,0.4) node{$\gamma''$};
\draw (-2,-0.5) node[left]{$m_s$}; \draw (2,0.5) node[right]{$m_t$};
\end{tikzpicture}
\ \
\begin{tikzpicture}[scale=0.7]
\filldraw [color=black!20] (-2,1)-- (-2,-1)--(-2.1,-1)--(-2.1,1);
\filldraw [color=black!20] (2,1)-- (2,-1)--(2.1,-1)--(2.1,1);
\filldraw [color=black!20] (-1,2)--(1,2)--(1,2.1)--(-1,2.1);
\filldraw [color=black!20] (-1,-2)--(1,-2)--(1,-2.1)--(-1,-2.1);
\draw [line width=1.2pt] (-2,1)-- (-2,-1);
\draw [line width=1.2pt] (2,1)-- (2,-1);
\draw [line width=1.2pt] (-1,2)-- (1,2);
\draw [line width=1.2pt] (-1,-2)-- (1,-2);
\draw [line width=12pt] [white] (-0.2,-2) -- (0.2,-2);
\draw (0,2)-- (-0.5,-2);
\draw (0,2)-- (0.5,-2);
\begin{scriptsize}
\fill [color=black] (0, 2) circle (2.2pt);
\fill [color=black] (-0.5,-2) circle (2.2pt);
\fill [color=black] (0.5,-2) circle (2.2pt);
\fill [color=black] (-2,-0.5) circle (2.2pt);
\fill [color=black] (2,0.5) circle (2.2pt);
\end{scriptsize}
\draw[red] (-2,-0.5) to[out=0, in=180] (2,0.5); \draw[<-][red] (1.79,0.5)--(1.8,0.5); \draw[red] (-1,-0.5) node {$c$};
\draw[orange] (-2,-0.5) to (0.5,-2); \draw[orange] (-1,-1.3) node {$c''$};
\draw[blue] (2,0.5) to (-0.5,-2); \draw[blue] (1,-0.3) node {$c'$};
\draw (-0.4,0.4) node{$\gamma''$}; \draw (0.4,0.4) node{$\gamma'$};
\draw (-2,-0.5) node[left]{$m_t$}; \draw (2,0.5) node[right]{$m_s$};
\end{tikzpicture}
\caption{The short exact sequence is described in marked ribbon surface.}
\label{Pic. of prop. s. e. s.}
\end{center}
\end{figure}
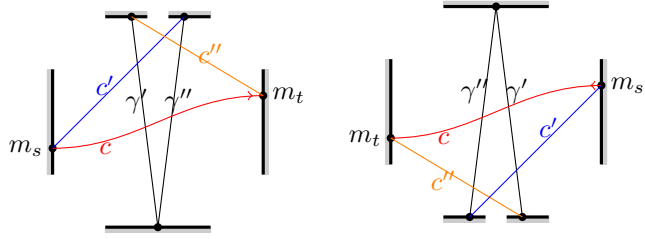

\begin{proof}
Notice that $c$ corresponds to such a string which is of the form
$ s_c:=\ \  \cdots -\!\!\!- \gamma' \longleftarrow \gamma'' -\!\!\!- \cdots. $
Then $c'$ and $c''$ correspond to
$s_{c'}:=\ \ \cdots -\!\!\!- \gamma' \ \text{ and } \ \ s_{c''}:=\ \ \gamma'' -\!\!\!- \cdots, \ \text{ respectively. } $
Furthermore, $M(c)$, $M(c')$ and $M(c'')$ are string modules corresponding to $s_c$, $s_{c'}$ and $s_{c''}$, respectively,
and hence
$$0 \longrightarrow M(c') \longrightarrow M(c) \longrightarrow M(c'') \longrightarrow 0$$ is exact.
\end{proof}

\begin{lemma} \label{lemm. submod.}
Let $A$ be a gentle algebra and $\Surf_A=(\Surf_A, \M_A,\Gamma_A)$ its marked ribbon surface.
Let $c=c_1\cdots c_r:(0,1)\to \Surf_A\backslash\partial\Surf_A$ be a permissible curve consecutively
crossing $\gamma_1, \cdots \gamma_{r-1}$ such that the endpoints $p$ and $q$
of $\gamma_i$ $(1 < i < r-1)$ are on the left and right of $c$,
respectively $($see \Pic \ref{Pic. of lemm. submod.}$)$. Then
for any permissible curve $c'$ consecutively crossing $\gamma_m, \cdots, \gamma_n$ $(1\le m < i < n \le r-1)$
such that
\begin{itemize}
  \item $\gamma_m, \cdots, \gamma_i$ have a common endpoint
    $p = p(\gamma_{m}, \gamma_{m+1}, c) = \cdots = p(\gamma_{i-1}, \gamma_i, c)$, and
  \item $\gamma_i, \cdots, \gamma_n$ have a common endpoint
    $q = p(\gamma_{i}, \gamma_{i+1}, c) = \cdots = p(\gamma_{n-1}, \gamma_n, c)$,
\end{itemize}
we have $M(c')\le M(c)$.
\end{lemma}

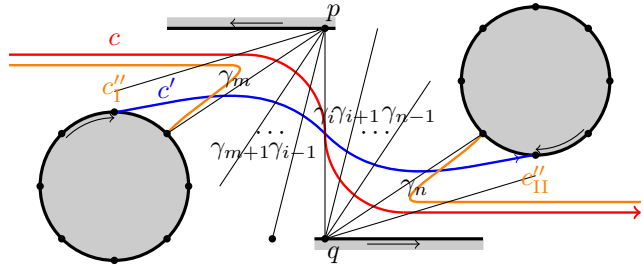
\begin{figure}[htbp]
\begin{center}
\begin{tikzpicture}[scale=0.7]
\filldraw[color=black!20] (-3,2) to (1,2) to (1,2.2) to (-3,2.2);
\filldraw[color=black!20] (-1,-2) to (3,-2) to (3,-2.2) to (-1,-2.2);
\filldraw[color=black!20] (-4,-1) circle (1.41cm);
\filldraw[color=black!20] (4,1) circle (1.41cm);
% permissible curves
\draw [line width=1.2pt] (-4,-1) circle (1.41cm);
\draw [line width=1.2pt] (4,1) circle (1.41cm);
\draw [line width=1.2pt] (-3,2)-- (1,2);
\draw [line width=1.2pt] (-1,-2)-- (3,-2);
\draw [blue, line width = 0.9pt]
      (-4,0.4) to[out=10, in=135] (0,0) to[out=-45, in=-170] (4,-0.4);
\draw[->] [blue](-4,0.4) to[out=10, in=135] (0,0) to[out=-45, in=-170] (3.7,-0.46);
\draw [shift={(-1.5,0)}, red, line width = 0.9pt]
      plot[domain=0:1.57,variable=\t]({1*1.5*cos(\t r)+0*1.5*sin(\t r)},{0*1.5*cos(\t r)+1*1.5*sin(\t r)});
\draw [shift={(1.5,0)}, red, line width = 0.9pt]
      plot[domain=3.14:4.71,variable=\t]({1*1.5*cos(\t r)+0*1.5*sin(\t r)},{0*1.5*cos(\t r)+1*1.5*sin(\t r)});
\draw [red, line width = 0.9pt] (-1.5,1.5)-- (-6,1.5);
\draw[->] [red, line width = 0.9pt] (1.5,-1.5)-- (6,-1.5);
\draw [orange, line width = 0.9pt] (-6,1.3) to (-1.8,1.3) to[out=  0,in= 50] (-3,0);
\draw [orange, line width = 0.9pt] (6,-1.3) to (1.8,-1.3) to[out=180,in=230] (3,0);
\draw [white][line width = 12pt] ( 0.2, 2.1) -- ( 1.6, 2.1);
\draw [white][line width = 12pt] (-1.6,-2.1) -- (-0.2,-2.1);
% oriented
\draw[<-] (-1.8, 2.1)-- (-0.8, 2.1);
\draw[->] ( 0.8,-2.1)-- ( 1.8,-2.1);
\draw[<-] (-4, 0.3) arc ( 90: 135:1.3);
\draw[<-] (4, -0.3) arc (-90:-45 :1.3);
% arc
\draw (0,2)-- (0,-2);
\draw (0,2)-- (-1,-2);
\draw (0,2)-- (-2,-1);
\draw (0,2)-- (-3,0);
\draw (0,-2)-- (1,2);
\draw (0,-2)-- (2,1);
\draw (0,-2)-- (3,0);
\draw (0,2)-- (-4,0.8);
\draw (0,-2)-- (4,-0.8);
\begin{scriptsize}
%\fill [color=black] (-3,2) circle (2pt);
%\fill [color=black] (3,-2) circle (2pt);
\fill [color=black] (0,2) circle (2pt);
\fill [color=black] (0,-2) circle (2pt);
\fill [color=black] (-1,-2) circle (2pt);
%\fill [color=black] (-2,-1) circle (2pt);
\fill [color=black] (-3,0) circle (2pt);
%\fill [color=black] (1,2) circle (2pt);
%\fill [color=black] (2,1) circle (2pt);
\fill [color=black] (3,0) circle (2pt);
\fill [color=black] (-4,0.41) circle (2pt);
\fill [color=black] (4,-0.41) circle (2pt);
\fill [color=black] (-5,0) circle (2pt);
\fill [color=black] (-5.41,-1) circle (2pt);
\fill [color=black] (5,0) circle (2pt);
\fill [color=black] (5.41,1) circle (2pt);
\fill [color=black] (-2.59,-1) circle (2pt);
\fill [color=black] (-5,-2) circle (2pt);
\fill [color=black] (-4,-2.41) circle (2pt);
\fill [color=black] (-3,-2) circle (2pt);
\fill [color=black] (2.59,1) circle (2pt);
\fill [color=black] (3,2) circle (2pt);
\fill [color=black] (5,2) circle (2pt);
\fill [color=black] (4,2.41) circle (2pt);
\end{scriptsize}
\draw[color=black] (0.16, 2.3) node{$p$};
\draw[color=black] (0.16,-2.3) node{$q$};
\draw (   0, 0) node[above]{$\gamma_i$};
\draw (-0.6, 0) node[below]{$\gamma_{i-1}$};
\draw (+0.6, 0) node[above]{$\gamma_{i+1}$};
\draw (-1  , 0) node{$\cdots$};
\draw ( 1.0, 0) node{$\cdots$};
\draw (-1.6, 0) node[below]{$\gamma_{m+1}$};
\draw ( 1.6, 0) node[above]{$\gamma_{n-1}$};
\draw (-1.7, 0.7) node[above]{$\gamma_{m}$};
\draw ( 1.7,-0.7) node[below]{$\gamma_{n}$};
\draw[red]    (-4.0, 1.5) node[above]{$c$};
\draw[blue]   (-3.0, 0.5) node[above]{$c'$};
\draw[orange] (-4.0, 1.3) node[below]{$c_{\I}''$};
\draw[orange] ( 4.0,-1.3) node[above]{$c_{\II}''$};
\end{tikzpicture}
\caption{$c$ and $c'$ given in Lemma \ref{lemm. submod.} corresponding to two module $M(c)$ and $M(c')$,
we have $M(c')\le M(c)$.
The formal direct sum $c_{\I}'' \oplus c_{\II}''$ corresponds to the module $M(c_{\I}''\oplus c_{\II}'')$
which isomorphic to the quotient $M(c)/M(c')$.}
\label{Pic. of lemm. submod.}
\end{center}
\end{figure}

\begin{proof}
The strings $s_c$ and $s_{c'}$ corresponding to $c$ and $c'$ are
\begin{align}
 \cdots \longrightarrow \gamma_m \longrightarrow \gamma_{m+1} \longrightarrow \cdots \longrightarrow \gamma_i
 \longleftarrow \cdots \longleftarrow \gamma_{n-1}\longleftarrow \gamma_n \longleftarrow  \cdots\ \text{ and } \nonumber \\
 \gamma_m \longrightarrow \gamma_{m+1} \longrightarrow \cdots \longrightarrow \gamma_i
 \longleftarrow \cdots \longleftarrow \gamma_{n-1} \longleftarrow \gamma_n \text{, respectively. }\nonumber
\end{align}
Then the string module corresponding to $s_{c'}$ is a submodule of the module corresponding to $s_c$,
and thus $M(c')\le M(c)$.
\end{proof}

\begin{proposition}\label{prop. submod. of string mod.3}
Keeping the notations in Lemma \ref{lemm. submod.}, we have that
$c' = c_m^{\mathfrak{T}}c_{m+1}\cdots c_nc_{n+1}^{\mathfrak{T}}: (0,1) \to \Surf_A\backslash\partial\Surf_A$
crosses $\gamma_i$ and that $c$ and $c'$ have an intersection.
Assume that
\begin{itemize}
\item $c_{\I}''$ is a permissible curve which is given as follows:
\begin{itemize}
\item $c'$ is divided into two parts by $c$, where denote the part with endpoints $c'(0^+)$ and $c' \cap \gamma_i$ by $c_{\I}'$;
\item $c_{\I}''\simeq\mathrm{prot}_{c}(c'_{\I})$ $($Definition \ref{def-3.5}$)$;
\end{itemize}
\item and $c_{\II}''$ is a permissible curve which is given as follows:
\begin{itemize}
\item $c'$ is divided into two parts by $c$, where denote the part with endpoints $c' \cap \gamma_i$ and $c'(1^-)$ by $c_{\II}'$;
\item $c_{\II}''\simeq\mathrm{prot}_{c}(c'_{\II})$ $($Definition \ref{def-3.5}$)$.
\end{itemize}
\end{itemize}
Then the quotient module $M(c)/M(c')$ is isomorphic to $M(c_{\I}''\oplus c_{\II}'')$
{\rm(}see \Pic \ref{fig:Prop-3.8 and 3.11} $(1)${\rm)}.
\end{proposition}

\begin{figure}[htbp]
\begin{center}
%%%%%%%%%%%%%%%%%%%%%%%%%%%%%%%%%%%%%%%%%%%%%%%%%%%%%%%%%%%%%%
%%%%%%%%%%%%%%%%%%%%%%%%% prop 3.8 %%%%%%%%%%%%%%%%%%%%%%%%%%%
%%%%%%%%%%%%%%%%%%%%%%%%%%%%%%%%%%%%%%%%%%%%%%%%%%%%%%%%%%%%%%
\ \
\begin{tikzpicture}[scale=0.7]
\draw [line width=1.2pt] (-2,1)-- (-2,-1);
\draw [line width=1.2pt] (2,1)-- (2,-1);
\draw [line width=1.2pt] (-1,2)-- (1,2);
\draw [line width=1.2pt] (-1,-2)-- (1,-2);
\draw [line width=12pt] [white] (-2,-0.2) -- (-2,0.2);
\draw [line width=12pt] [white] ( 2,-0.2) -- ( 2,0.2);
\draw [line width=12pt] [white] ( 0.2, 2) -- ( 0.6, 2);
\draw [line width=12pt] [white] (-0.2,-2) -- (-0.6,-2);
\draw (0,-2)-- (0, 2);
\draw (0, 2)-- (-0.8,-2.0); \draw (0,-2)-- ( 0.8, 2.0);
\draw (0, 2)-- (-1.5,-1.5); \draw (0,-2)-- ( 1.5, 1.5);
\draw (0, 2)-- (-2.0,-0.8); \draw (0,-2)-- ( 2.0, 0.8);
\draw (0, 2)-- (-2.0, 0.4); \draw (0,-2)-- ( 2.0,-0.4);
\begin{scriptsize}
\fill [color=black] ( 0.0,-2.0) circle (2.2pt); \fill [color=black] ( 0.0, 2.0) circle (2.2pt);
\fill [color=black] ( 0.8, 2.0) circle (2.2pt); \fill [color=black] (-0.8,-2.0) circle (2.2pt);
\fill [color=black] (-2.0,-0.8) circle (2.2pt); \fill [color=black] ( 2.0, 0.8) circle (2.2pt);
\end{scriptsize}
\draw[red] (-2,1.5) to[out=0, in=180] ( 2,-1.5); \draw[red] (-1, 1.1) node[above] {$c$};
\draw[blue] (-2, 0.4) to[out=0, in=180] ( 2,-0.4); \draw[blue] (-1,0.4) node {$c'$};
\draw[orange] (-2.0,-0.8) to[out=70, in=0] (-2, 1.45); \draw[orange] (-1.8, 1.4) node[below] {$c_{\I}''$};
\draw[orange] ( 2.0, 0.8) to[out=-110, in=180] ( 2,-1.45); \draw[orange] ( 1.8,-1.4) node[above] {$c_{\II}''$};
\fill [color=black] (-2, 0.4) circle (2.2pt); \fill [color=black] ( 2,-0.4) circle (2.2pt);
% number
\draw  (0,-2.2) node[below]{(1)};
\end{tikzpicture}
\ \ \ \
%%%%%%%%%%%%%%%%%%%%%%%%%%%%%%%%%%%%%%%%%%%%%%%%%%%%%%%%%%%%%%
%%%%%%%%%%%%%%%%%%%%%%%%% pro 3.11 %%%%%%%%%%%%%%%%%%%%%%%%%%%
%%%%%%%%%%%%%%%%%%%%%%%%%%%%%%%%%%%%%%%%%%%%%%%%%%%%%%%%%%%%%%
\ \ \ \
\begin{tikzpicture}[scale=0.7]
\draw [line width=1.2pt] (-2,1)-- (-2,-1);
\draw [line width=1.2pt] (2,1)-- (2,-1);
\draw [line width=1.2pt] (-1,2)-- (1,2);
\draw [line width=1.2pt] (-1,-2)-- (1,-2);
\draw [line width=12pt] [white] (-2,-0.2) -- (-2,0.2);
\draw [line width=12pt] [white] ( 2,-0.2) -- ( 2,0.2);
\draw [line width=12pt] [white] (-0.2, 2) -- (-0.6, 2);
\draw [line width=12pt] [white] ( 0.2,-2) -- ( 0.6,-2);
\draw (0,-2)-- (0, 2);
\draw (0,-2)-- (-0.8, 2.0); \draw (0, 2)-- ( 0.8,-2.0);
\draw (0,-2)-- (-1.5, 1.5); \draw (0, 2)-- ( 1.5,-1.5);
\draw (0,-2)-- (-2.0,-0.8); \draw (0, 2)-- ( 2.0, 0.8);
\draw (0,-2)-- (-2.0, 0.4); \draw (0, 2)-- ( 2.0,-0.4);
\begin{scriptsize}
\fill [color=black] ( 0.0,-2.0) circle (2.2pt); \fill [color=black] ( 0.0, 2.0) circle (2.2pt);
\fill [color=black] ( 0.8,-2.0) circle (2.2pt); \fill [color=black] (-0.8, 2.0) circle (2.2pt);
\fill [color=black] (-2.0,-0.8) circle (2.2pt); \fill [color=black] ( 2.0, 0.8) circle (2.2pt);
\end{scriptsize}
\draw[red] (-2,1.5) to[out=0, in=180] ( 2,-1.5); \draw[red] (-1, 1.1) node[above] {$c$};
\draw[orange] (-2, 0.4) to[out=0, in=180] ( 2,-0.4); \draw[orange] (-1,0.4) node {$c''$};
\draw[blue] (-2.0,-0.8) to[out=70, in=0] (-2, 1.45); \draw[blue] (-1.8, 1.4) node[below] {$c_{\I}'$};
\draw[blue] ( 2.0, 0.8) to[out=-110, in=180] ( 2,-1.45); \draw[blue] ( 1.8,-1.4) node[above] {$c_{\II}'$};
\fill [color=black] (-2, 0.4) circle (2.2pt); \fill [color=black] ( 2,-0.4) circle (2.2pt);
% number
\draw  (0,-2.2) node[below]{(2)};
\end{tikzpicture}
\caption{The descriptions of short exact sequences given in Propositions \ref{prop. submod. of string mod.3} and \ref{prop. quotientmod.}.}
\label{fig:Prop-3.8 and 3.11}
\end{center}
\end{figure}
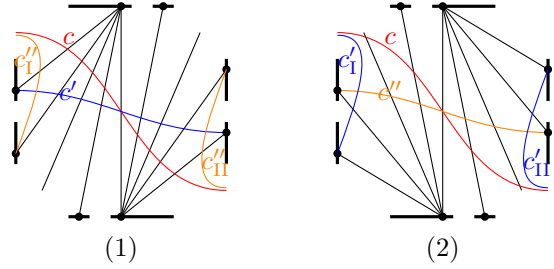

\begin{proof}
Similar to the proof of Lemma \ref{lemm. submod.}, we have that $M(c)/M(c')$ corresponds to two permissible curves
$$\tilde{c}_{\I}''=c_1\cdots c_{m-1}c_{m}^{\mathfrak{T}}\ {\rm and}\ \tilde{c}_{\II}''=c_{n+1}^{\mathfrak{T}}c_{n+2}\cdots c_r.$$
We can see that $\tilde{c}_{\I}''$ consecutively crosses arcs $\gamma_1, \cdots, \gamma_{m-1}$ and $\tilde{c}_{\II}''$
consecutively crosses arcs $\gamma_{n+1}, \cdots, \gamma_{r-1}$ ($1\le m\le i\le n\le r-1$) (see \Pic \ref{Pic. of lemm. submod.}).
Thus $\tilde{c}_{\I}''$ and $\tilde{c}_{\II}''$ lie in the equivalence class $[c_{\I}'']$ and $[c_{\II}'']$, respectively,
and therefore $M(c)/M(c')\cong M(\tilde{c}_{\I}''\oplus \tilde{c}_{\II}'') \cong M(c_{\I}''\oplus c_{\II}'')$.
\end{proof}

\begin{remark}\rm
If $\gamma_1=\gamma_m=\gamma_i$ (respectively, $\gamma_i=\gamma_n=\gamma_r$), then
$c' = c_1^{\mathfrak{T}}c_2\cdots c_nc_{n+1}^{\mathfrak{T}} \simeq c_1c_2\cdots c_nc_{n+1}^{\mathfrak{T}}$
(respectively, $c' = c_m^{\mathfrak{T}}c_{m+1}\cdots c_{r-1}c_r^{\mathfrak{T}} \simeq c_m^{\mathfrak{T}}c_{m+1}\cdots c_{r-1}c_r$),
and thus $c_{\I}''$ (respectively, $c_{\II}''$) is zero.
\end{remark}

The following two results are dual to Lemma \ref{lemm. submod.} and Proposition \ref{prop. submod. of string mod.3}, respectively.

\begin{lemma} \label{lemm. quotientmod.}
Let $A$ be a gentle algebra and $\Surf_A=(\Surf_A, \M_A,\Gamma_A)$ its marked ribbon surface.
Let $c=c_1\cdots c_r:(0,1)\to \Surf_A\backslash\partial\Surf_A$ be a permissible curve consecutively crossing
$\gamma_1, \cdots \gamma_{r-1}$ such that the endpoints $p$ and $q$ of $\gamma_i$ $(1\le i\le r-1)$
are on the right and left of $c$, respectively.
Then for any permissible curve $c''$ consecutively crossing $\gamma_m, \cdots, \gamma_n$
$(1\le m < i < n \le r-1)$ such that
\begin{itemize}
  \item $\gamma_m, \cdots, \gamma_i$ have a common endpoint
    $p = p(\gamma_{m}, \gamma_{m+1}, c) = \cdots = p(\gamma_{i-1}, \gamma_i, c)$, and
  \item $\gamma_i, \cdots, \gamma_n$ have a common endpoint
    $q = p(\gamma_{i}, \gamma_{i+1}, c) = \cdots = p(\gamma_{n-1}, \gamma_n, c)$,
\end{itemize}
we have that $M(c'')$ is a quotient module of $M(c)$.
\end{lemma}

\begin{proposition}\label{prop. quotientmod.}
Keeping the notations in Lemma \ref{lemm. quotientmod.}, we have that
$c'' = c_m^{\mathfrak{T}}c_{m+1}\cdots c_nc_{n+1}^{\mathfrak{T}}: (0,1) \to \Surf_A\backslash\partial\Surf_A$ crosses $\gamma_i$
and that $c$ and $c''$ have an intersection. Assume that
\begin{itemize}
\item $c_{\I}'$ is a permissible curve which is given as follows:
\begin{itemize}
\item $c''$ is divided into two parts, where denote the part with endpoints $c''(0^+)$ and $c'' \cap \gamma_i$ by $c_{\I}''$;
% \item move the point of $c_{\I}''$ at $c'' \cap \gamma_i$ to the endpoint $c(0^+)$;
% \item follows the positive direction of boundary, move the endpoint $c''(0^+)$ of $c''$ to the previous vertex of elementary polygon whose arc set contains $\gamma_{m-1}$ and $\gamma_{m}$.
\item $c_{\I}'\simeq \mathrm{nrot}_{c}(c_{\I}'')$ $($Definition \ref{def-3.5}$)$;
\end{itemize}
\item and $c_{\II}'$ is a permissible curve which is given as follows:
\begin{itemize}
\item $c''$ is divided into two parts, where denote the part with endpoints $c'' \cap \gamma_i$ and $c''(1^-)$ by $c_{\II}''$;
\item $c_{\II}'\simeq \mathrm{nrot}_{c}(c_{\II}'')$ $($Definition \ref{def-3.5}$)$.
\end{itemize}
\end{itemize}
Then $M(c'')$ is isomorphic to $M(c)/M(c_{\I}'\oplus c_{\II}')$ {\rm(}see \Pic \ref{fig:Prop-3.8 and 3.11} $(2)${\rm)}.
\end{proposition}

By Propositions \ref{prop. s. e. s.}, \ref{prop. submod. of string mod.3} and \ref{prop. quotientmod.},
we get the following theorem.

\begin{theorem} \label{thm. s. e. s.}
Let $A=kQ/I$ be a gentle algebra and $\Surf_A = (\Surf_A, \M_A, \Gamma_A)$ its marked ribbon surface.
\begin{enumerate}
\item[{\rm(1)}] If $c$, $c'$ and $c''$ are permissible curves satisfying the conditions in Proposition \ref{prop. s. e. s.}, then there is
%(i.e., the positional relationship of $c$, $c'$ and $c''$ is given by \Pic \ref{Pic. of prop. s. e. s.}),
a short exact sequence
$$0 \longrightarrow M(c') \longrightarrow M(c) \longrightarrow M(c'') \longrightarrow 0.$$

\item[{\rm(2)}] If $c$ and $c'$ are permissible curves satisfy the conditions in Lemma \ref{lemm. submod.},
and if $c_{\I}''$ and $c_{\II}''$ are permissible curves satisfying the conditions in Proposition \ref{prop. submod. of string mod.3},
%(i.e.,  the positional relationship of $c$, $c'$, $c_{\I}''$ and $c_{\II}''$ is given by \Pic \ref{Pic. of lemm. submod.}),
then there is a short exact sequence
$$0 \longrightarrow M(c') \longrightarrow M(c) \longrightarrow M(c_{\I}''\oplus c_{\II}'') \longrightarrow 0.$$

\item[{\rm(3)}] {\rm (The dual of (2))} If $c$ and $c''$ are permissible curves satisfying the conditions in Lemma \ref{lemm. quotientmod.},
and if $c_{\I}'$ and $c_{\II}'$ are permissible curves satisfying the conditions in Proposition \ref{prop. quotientmod.},
then there is a short exact sequence
$$0 \longrightarrow M(c_{\I}'\oplus c_{\II}') \longrightarrow M(c) \longrightarrow M(c'') \longrightarrow 0.$$
\end{enumerate}
\end{theorem}

\section{\bf Special modules}

\subsection{Simple modules and top (respectively, socle) of indecomposable modules}
A permissible curve in $\Surf_A$ corresponding to a simple $A$-module is a permissible curve crossing a unique arc in $\Gamma_A$,
we call it {\it simple}. The top (respectively, socle) of an $A$-module is semi-simple,
thus we can use a family of simple permissible curves to describe the top (respectively, socle) of any $A$-module.
Moreover, for a permissible curve $c$ consecutively crossing arcs $\gamma$ and $\gamma'$,
recall that $p(\gamma, \gamma', c)$ is the marked point lying in $\gamma \cap \gamma'$ such that
it is the vertex of the triangle decided by $\gamma$, $\gamma'$ and the segment given by $\gamma$ and $\gamma'$ cutting $c$.
We will use this notation frequently in the sequel.

\begin{lemma}\label{lemm. tops of string mods.}
Let $A=kQ/I$ be a gentle algebra and $\Surf_A=(\Surf_A,\M_A,\Gamma_A)$ its marked ribbon surface.
Let $c=c_1c_2\cdots c_r$ be a non-trivial permissible $($non-closed$)$ curve in $\Surf_A$ consecutively crossing arcs
$\gamma_1, \gamma_2, \cdots, \gamma_{r-1}$ $($note: $c_i$ is the segment between $\gamma_{i-1}$ and $\gamma_i$, $c_1$
is the segment between $\partial\Surf_A$ and $\gamma_1$, and $c_r$ is the segment between $\gamma_{r-1}$ and $\partial\Surf_A$$)$.
If there exists some $i$ with $1\le i\le r$ such that one of the following conditions is satisfied:
\begin{itemize}
\item Case 1 {\rm(}respectively, Case 1$'${\rm)}: $p(\gamma_{i-1},\gamma_{i},c)$ is on the right
{\rm(}respectively, left{\rm)} of $c$, and $p(\gamma_{i},\gamma_{i+1},c)$ is on the left {\rm(}respectively, right{\rm) of $c$} $(2\le i \le r-2)$;
\item Case 2 {\rm(}respectively, Case 2$'${\rm)}: $c_i=c_1$ is an end segment such that $p(\gamma_1,\gamma_2,c)$
is on the left {\rm(}respectively, right{\rm)} of $c$;
\item Case 3 {\rm(}respectively, Case 3$'${\rm)}: $c_i=c_r$ is an end segment such that $p(\gamma_{r-2}, \gamma_{r-1}, c)$
is on the right {\rm(}respectively, left{\rm)} of $c$,
\end{itemize}
\noindent then $M(c_{\simp}^{\gamma_i}) \le_{\oplus} {\rm top} M(c)$ $($respectively, $M(c_{\simp}^{\gamma_i}) \le_{\oplus} \soc M(c))$,
where $c_{\simp}^{\gamma_i}$ is the simple permissible curve crossing $\gamma_i$
and $M:\mathfrak{P}(\Surf_A) \to \ind \bmod A$ is the bijection given in Remark \ref{rmk. in BS}.
\end{lemma}

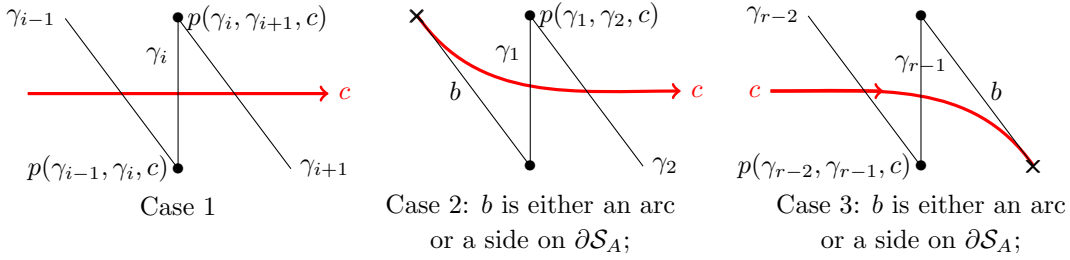
\begin{figure}[htbp]
\begin{center}
\begin{tikzpicture}%[xscale=0.8]
\draw[red][line width=1.2pt][->] (-2,0) -- (2,0) node[right]{$c$};
\draw[black] (0,1) node{$\bullet$} node[right]{$p(\gamma_i, \gamma_{i+1}, c)$};
\draw[black] (0,-1) node{$\bullet$} node[left]{$p(\gamma_{i-1}, \gamma_i, c)$};
\draw[black] (-1.5,1) node[left]{$\gamma_{i-1}$} -- (0,-1) -- (0,1) -- (1.5,-1) node[right]{$\gamma_{i+1}$};
\draw[black] (0,0.5) node[left]{$\gamma_i$};
\draw (0,-1.5) node{Case 1};
\draw[opacity=0] (0,-2) node{Sentences};
\end{tikzpicture}
\
\begin{tikzpicture}%[xscale=0.8]
\draw[red][line width=1.2pt][->] (-1.5,1) to[out=-55,in=180] (2,0) node[right]{$c$};
\draw[black] (0,1) node{$\bullet$} node[right]{$p(\gamma_1, \gamma_2, c)$};
\draw[black] (0,-1) node{$\bullet$};
\draw[black] (-1.5,1) node{$\pmb{\times}$} -- (0,-1) -- (0,1) -- (1.5,-1) node[right]{$\gamma_2$};
\draw[black] (-1,0) node{$b$};
\draw[black] (0,0.5) node[left]{$\gamma_1$};
\draw (0,-1.5) node{Case 2: $b$ is either an arc};
\draw (0,-2) node{or a side on $\partial\mathcal{S}_A$;};
\end{tikzpicture}
\
\begin{tikzpicture}%[xscale=0.8]
\draw[red][line width=1.2pt] (-2,0)node[left]{$c$} to[out=0,in=125] (1.5,-1);
\draw[red][line width=1.2pt][->] (-2,0) -- (-0.5,0);
\draw[black] (0,1) node{$\bullet$};
\draw[black] (0,-1) node{$\bullet$} node[left]{$p(\gamma_{r-2}, \gamma_{r-1}, c)$};
\draw[black] (-1.5,1) node[left]{$\gamma_{r-2}$} -- (0,-1) -- (0,1) -- (1.5,-1) node{$\pmb{\times}$};
\draw[black] ( 1,0) node{$b$};
\draw[black] (0,0.35) node{$\gamma_{r-1}$};
\draw (0,-1.5) node{Case 3: $b$ is either an arc};
\draw (0,-2) node{or a side on $\partial\mathcal{S}_A$;};
\end{tikzpicture}
\caption{$M(c_{\mathrm{simp}}^{\gamma_i})$ is a direct summand of $\mathrm{top}M(c)$
(the point ``$\pmb{\times}$'' is either a marked point or an extra point),
where $i=1$ in Case 2 and $i=r-1$ in Case 3.
Cases $1'$, $2'$ and $3'$ are dual. }
\label{fig:lemma-4.1}
\end{center}
\end{figure}

\begin{proof}
We only prove the case for top, and the case for socle is dual.
For any string module $M$, the source, denoted by $i$, of the string $s$ corresponding to $M$
is one of the following three types (we denote by $\gamma_i$ the arc corresponding to $i$):
\[\text{Type 1:}\ \ \cdots  \longleftarrow  i \longrightarrow \cdots  \ ; \text{Type 2:} \ \ i \longrightarrow \cdots\ ;\
\text{and}\ \text{Type 3:}\ \ \cdots \longleftarrow i\ .\]
Then the module $S(i)$, the simple module corresponding to the vertex $i$, is a direct summand of ${\rm top} M$,
that is, $S(i) \le_{\oplus}{\rm top} M$.
On the other hand, the permissible curve $c=c_1\cdots c_r$ corresponding to $s$ satisfies one of the conditions in Cases 1--3
accordingly (in Cases 2 and 3, we have $i=1$ and $i=r-1$, respectively).
Thus $S(i)$ corresponds to the permissible curve $c_{\simp}^{\gamma_i}$ and
$M(c_{\simp}^{\gamma_i}) \le_{\oplus}{\rm top}  M \cong {\rm top} M(c)$.
\end{proof}

\begin{lemma}\label{lemm. tops of band mods.}
Let $A=kQ/I$ be a gentle algebra and $\Surf_A=(\Surf_A,\M_A,\Gamma_A)$ its marked ribbon surface.
Let $c=c_1c_2\cdots c_r$ be a non-trivial permissible closed curve in $\Surf_A$  consecutively crossing arcs
$\gamma_1, \gamma_2, \cdots, \gamma_{r-1}, \gamma_r=\gamma_0$
{\rm(}$c_i$ is the segment between $\gamma_{\overline{i-1}}$ and $\gamma_{\overline{i}}$, where $\overline{i}$ is $i$ modulo $r${\rm)},
and let $\varphi\in\mathrm{Aut}(k^m)$ with $m\geq 1$ be a Jordan block.
If there exists some $i$ with $1\le i\le r$ such that the following condition is satisfied:
\begin{itemize}
\item Case RL {\rm(}respectively, Case LR{\rm)}: $p(\gamma_{\overline{i-1}}, \gamma_{\overline{i}}, c)$
is on the right {\rm(}respectively, left{\rm)} of $c$ and
  $p(\gamma_{\overline{i}},\gamma_{\overline{i+1}},c)$ is on the left {\rm(}respectively, right{\rm)} of $c$,
\end{itemize}
\noindent then $M(c_{\simp}^{\gamma_i})^{\oplus m} \le_{\oplus} \mathrm{top}(M(c,\varphi))$
$($respectively, $M(c_{\simp}^{\gamma_i})^{\oplus m} \le_{\oplus} \mathrm{soc}(M(c,\varphi)))$.
\end{lemma}

\begin{proof}
For any band module $M$, a source, denoted by $i$, of the band corresponding to $M$ must be
Type 1 in the proof of Lemma \ref{lemm. tops of string mods.}.
Note that $M(c,\varphi) e_i \cong k^m$ is $k$-linearly isomorphic to $M(c_{\simp}^{\gamma_i})^{\oplus m}$,
where $e_i$ is the trivial path corresponding to $i\in Q_0$.
Now, similar to the proof of Lemma \ref{lemm. tops of string mods.}, we have
$M(c_{\simp}^{\gamma_i})^{\oplus m} \le_{\oplus} \mathrm{top} M(c,\varphi)$.
The case for socle is dual.
\end{proof}

As a consequence of Lemmas \ref{lemm. tops of string mods.} and \ref{lemm. tops of band mods.},
we get the following proposition.

\begin{proposition}\label{prop. tops of mods.}
\begin{enumerate}
\item[]
\item[{\rm (1)}] Under the assumptions in Lemma \ref{lemm. tops of string mods.}, let \begin{align}
\mathfrak{I} & =\{1\le i\le r-1\mid i \text{ satisfies one of Cases 1, 2 and 3
in Lemma \ref{lemm. tops of string mods.} } \}, \nonumber \\
(respectively, \mathfrak{I} & =\{1\le i\le r-1\mid i \text{ satisfies one of Cases 1$'$, 2$'$ and 3$'$
in Lemma \ref{lemm. tops of string mods.} } \}). \nonumber
\end{align}
Then $\mathrm{top}M(c) \cong \bigoplus\limits_{i\in \mathfrak{I}} M(c_{\simp}^{\gamma_i})$
(respectively, $\mathrm{soc}M(c) \cong \bigoplus\limits_{i\in \mathfrak{I}} M(c_{\simp}^{\gamma_i}))$.

\item[{\rm (2)}] Under the assumptions in Lemma \ref{lemm. tops of band mods.},
let \begin{align}
\mathfrak{I} &= \{1\le i\le r-1 \mid i \text{ satisfies Case RL  in Lemma \ref{lemm. tops of band mods.} } \}, \nonumber \\
(respectively, \mathfrak{I} &= \{1\le i\le r-1 \mid i \text{ satisfies Case LR in Lemma \ref{lemm. tops of band mods.} } \}).
\nonumber \end{align}
Then $\mathrm{top}M(c,\varphi) \cong \bigoplus\limits_{i\in \mathfrak{I}} M(c_{\simp}^{\gamma_i})^{\oplus m}$
{\rm(}respectively, $\mathrm{soc}M(c,\varphi) \cong \bigoplus\limits_{i\in \mathfrak{I}} M(c_{\simp}^{\gamma_i})^{\oplus m}${\rm)}.
\end{enumerate}
\end{proposition}

\begin{proof}
(1) Let $M\in\bmod A$ be string and
\[\mathfrak{I}'=\{1\le i\le r-1 \mid i \text{ is a source which is one of Types 1, 2 and 3 in the proof of
Lemma \ref{lemm. tops of string mods.}} \}.\]
Then $\mathrm{top}M \cong \mathrm{top}M(c) \cong \bigoplus\limits_{i\in \mathfrak{I}'}S(i) \cong
M(\bigoplus\limits_{i\in \mathfrak{I}}c_{\simp}^{\gamma_i}) \cong \bigoplus\limits_{i\in \mathfrak{I}} M(c_{\simp}^{\gamma_i})$
($\bigoplus\limits_{i\in \mathfrak{I}}c_{\simp}^{\gamma_i}$ is the formal direct sum of $c_{\simp}^{\gamma_i}$,
see Definition \ref{def. d. s.}).

(2) It is similar to (1).
\end{proof}

\subsection{Projective modules and projective covers}

Let $A=kQ/I$ be a gentle algebra and $\Surf_A=(\Surf_A,\M_A,\Gamma_A)$ its marked ribbon surface of $A$.
If $s=\alpha_1\alpha_2\cdots\alpha_n$ ($\alpha_i\in Q_0\cup Q_0^-$) is a string corresponding to a projective module,
then $s$ is one of the following cases:
\begin{itemize}
\item $s$ is {\it left-maximal inverse}, that is, $s$ is inverse and $\alpha^- s\in I$ for any $\alpha\in Q_0$ with $t(\alpha_1^-)=s(\alpha)$;
\item $s$ is {\it right-maximal direct}, that is, $s$ is direct and $s\alpha\in I$ for any $\alpha\in Q_0$ with $t(\alpha_n)=s(\alpha)$;
\item there exists a unique $i$ with $1\leq i\leq n-1$
such that $\alpha_1\cdots\alpha_{i}$ is left-maximal inverse and $\alpha_{i+1}\cdots\alpha_n$ is right-maximal direct.
\end{itemize}
Therefore, a permissible curve $c:(0,1)\to \Surf_A\backslash\partial \Surf_A$ (which consecutively crosses arcs $\gamma_1, \cdots, \gamma_{r-1}$)
corresponds to an indecomposable projective module by $M:\mathfrak{P}(\Surf_A) \to \mathrm{ind}(\bmod A)$ defined in Remark \ref{rmk. in BS}
if and only if $c$ is a permissible curve such that the following conditions are satisfied:
\begin{itemize}
\item (\textbf{PQC}) (Projective Quotient Condition) There exits $i$ with $1\leq i\leq r-1$ such that
\begin{itemize}
\item (\textbf{PQC 1}) %$\gamma_1, \gamma_2, \cdots,\gamma_i$ have a common endpoint $m\in\M_A$
    $m:=p(\gamma_1, \gamma_2, c) = \cdots = p(\gamma_{i-1},\gamma_i, c)$ which is on the right of $c$; and
\item (\textbf{PQC 2}) %$\gamma_i, \gamma_{i+1}, \cdots, \gamma_{r-1}$ have a common endpoint $m'\in\M_A$
$m':=p(\gamma_i, \gamma_{i+1}, c) = \cdots = p(\gamma_{r-2},\gamma_{r-1}, c)$ which is on the left of $c$;
\end{itemize}
\item (\textbf{MC}) (Maximal Conditions) There are no arcs $\gamma$ and $\gamma'$ such that
\begin{itemize}
\item (\textbf{LMC}) $\gamma, \gamma_1, \gamma_2, \cdots, \gamma_i$ have a common endpoint $m$ and they surround $m$ in the clockwise order; and
\item (\textbf{RMC}) $\gamma_i,\gamma_{i+1}, \cdots, \gamma_{n}, \gamma'$ have a common endpoint $m'$ and
they surround $m'$ in the counterclockwise order.
\end{itemize}
\end{itemize}

\noindent \Pic\ref{proj. per. curve} provides an example of a permissible curve corresponding to a projective module.
\begin{figure}[htbp]
\begin{center}
\begin{tikzpicture}[scale=0.7]
\draw (-1,2)-- (0,0);
\draw (-2,2)-- (0,0);
\draw (0,0)-- (0,3);
\draw (0,3)-- (1,1);
\draw (0,3)-- (2,1);
\draw (0,0)-- (-3,2);
\draw (0,3)-- (3,1);
\filldraw[color=black!20] (-1,-1)to [out=90, in=180] (0,0) to [out=0,in=90] (1,-1);
\filldraw[color=black!20] (-1,4)to [out=-90, in=180] (0,3) to [out=0,in=-90] (1,4);
\draw[line width =1pt] (-1,-1)to [out=90, in=180] (0,0) to [out=0,in=90] (1,-1);
\draw[line width =1pt] (-1,4)to [out=-90, in=180] (0,3) to [out=0,in=-90] (1,4);
\draw[color=red] (-0.71,-0.29) to[out=90, in=180] (0,1.5) to[out=0,in=-90] (0.71,3.29);
\begin{scriptsize}
\fill [color=black] (0,0) circle (2.2pt);
\fill [color=black] (0,3) circle (2.2pt);
\fill [color=black] (-1,-1) circle (2.2pt);
\fill [color=black] (1,-1) circle (2.2pt);
\fill [color=black] (-1,4) circle (2.2pt);
\fill [color=black] (1,4) circle (2.2pt);
\fill [color=black] (-0.71,3.29) circle (2.2pt);
\fill [color=black] (0.71,3.29) circle (2.2pt);
\fill [color=black] (-0.71,-0.29) circle (2.2pt);
\fill [color=black] (0.71,-0.29) circle (2.2pt);
\end{scriptsize}
\draw (-0.71, 0.2) node[left]{$\color{red}{c}$};
\draw (-1, 1.5) node{$\cdots$};
\draw ( 1, 1.5) node{$\cdots$};
\end{tikzpicture}
\caption{Each point $\bullet$ is either a marked point or an extra point.
The red curve $c$ is a permissible curve satisfying (\textbf{PQC}) and (\textbf{MC}),
thus the string module corresponding to $c$ is a projective module. }
\label{proj. per. curve}
\end{center}
\end{figure}
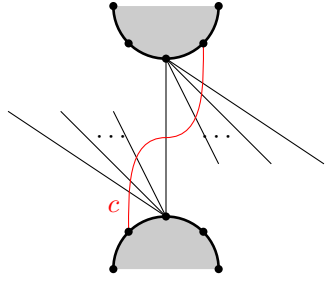

\begin{definition} \label{def. proj. per. curve} \rm (Projective permissible curves)
Let $A=kQ/I$ be a gentle algebra and $\Surf_A=(\Surf_A, \M_A, \Gamma_A)$ its marked ribbon surface.
For an arc $\gamma\in \Gamma_A$, a permissible curve $c$ consecutively crossing arcs
$\gamma_1, \cdots, \gamma_{r-1}$ is called a {\it projective permissible curve} of $\gamma$, denoted by $c_{\proj}^{\gamma}$,
if $c$ is decided by the pair $(c_1, c_2)$ of two curves $c_1:(0,1]\to \Surf_A\backslash\partial\Surf_A$ and $c_2: [0,1)\to \Surf_A\backslash\partial\Surf_A$, where
\begin{itemize}
\item there is a unique $i$ with $1\leq i\leq r-1$ such that $\gamma=\gamma_i$;
\item $c(0^+)=c_1(0^+)$, $c(1^-)=c_2(1^-)$, $c_1(1)=c_2(0)=\gamma\cap c$;
\item $c$ is obtained by connecting $c_1$ and $c_2$, that is,  $c=c_1\cup c_2$; and
\item $c_1$ and $c_2$ consecutively cross arcs $\gamma_1, \cdots, \gamma_i$ and $\gamma_i, \cdots, \gamma_{r-1}$
such that (\textbf{PQC 1}, \textbf{LMC}) and (\textbf{PQC 2}, \textbf{RMC}) are satisfied, respectively.
\end{itemize}
\end{definition}

We use $\mathfrak{P}_{\proj}(\Surf_A)$ to denote the set of all equivalence classes of projective permissible curves.
Obviously, we have $\mathfrak{P}_{\proj}(\Surf_A)\subseteq\Eqcls(\Surf_A)$.

\begin{lemma}{\rm(Projective modules)} \label{Lemm. proj}
Let $A=kQ/I$ be a gentle algebra and $\Surf_A=(\Surf_A,\M_A,\Gamma_A)$ its marked ribbon surface.
Then there exists a one-to-one mapping between the equivalence classes of projective permissible curves
and the isomorphic classes of indecomposable projective modules which is given by
\begin{align}
M|_{\mathfrak{P}_{\proj}(\Surf_A)}: \mathfrak{P}_{\proj}(\Surf_A) \to \mathrm{ind}(\proj A), c \mapsto M(c). \nonumber
\end{align}
\end{lemma}

\begin{proof}
A permissible $c$ lies in $\mathfrak{P}_{\proj}(\Surf_A)$ if and only if it corresponds to a string $s$
which is the connection of a left-maximal (trivial) inverse string $s_1$ and a right-maximal (trivial) direct string $s_2$,
that is, $s=\alpha_1\alpha_2\cdots\alpha_n$ is of the form
$$\xymatrix@C=0.5cm{\cdot\ar@{<-}[r]^{\alpha_1^-} &\cdot\ar@{<-}[r]^{\alpha_2^-} &\cdots
\ar@{<-}[r]^{\alpha_i^-}&\cdot\ar[r]^{\alpha_{i+1}}&\cdots\ar[r]^{\alpha_{n-1}}&\cdot\ar[r]^{\alpha_n}&\cdot}.$$
Thus the string module corresponding to $s$ is projective.
(If $\alpha_1\cdots\alpha_i$ (respectively, $\alpha_{i+1}\cdots\alpha_n$) is trivial, then the module corresponding
to $s$ is $P(v)\cong e_v A$, where $v\in Q_0$ is the source of $\alpha_{i+1}$ (respectively, $\alpha_{i}$)
and $e_v$ is the primitive idempotent corresponding to the vertex $v$).
\end{proof}

Now we consider the projective cover of a module by marked ribbon surfaces of gentle algebras.

\begin{theorem} \label{thm. proj. cover} {\rm(Projective covers of indecomposable modules)}
Let $A=kQ/I$ be a gentle algebra and $\Surf_A=(\Surf_A, \M_A, \Gamma_A)$ its marked ribbon surface.
Let $M = M(c)$ {\rm(}respectively, $M = M(c, \varphi)${\rm)} be a string module (respectively, band module)
which corresponds to the permissible curve $c$ {\rm(}respectively, the permissible closed curve $c$
with a Jordan block over a $k$-linear space $V${\rm)}. Then
\[ P(M) \cong \bigoplus\limits_{i\in \mathfrak{I}}M(c_{\proj}^{\gamma_i}) \ \ \ \
(respectively, P(M) \cong \bigoplus\limits_{i\in \mathfrak{I}}M(c_{\proj}^{\gamma_i})^{\oplus \dim_k V}), \]
where $\{\gamma_i\mid i\in \mathfrak{I}\}$ and $\mathfrak{I}$ are the same as that in
Lemma \ref{lemm. tops of string mods.} and Proposition \ref{prop. tops of mods.}$(2)$
{\rm(}respectively, Lemma \ref{lemm. tops of band mods.} and Proposition \ref{prop. tops of mods.}$(2)${\rm)}, respectively.
\end{theorem}

\begin{proof}
For each indecomposable module $M$, its projective cover $P(M)$ is isomorphic to the projective cover
$P(\mathrm{top}M)$ of its top. Thus we only need to consider the projective cover of each simple module.
Note that, for each simple permissible curve $c_{\simp}^{\gamma}$, we have $P(M(c_{\simp}^{\gamma})) = M(c_{\proj}^{\gamma})$.
Now suppose $\mathrm{top}M \cong \bigoplus\limits_{i\in \mathfrak{I}} M(c_{\simp}^{\gamma_i})$
(when $M$ is string, each $\gamma_i$ correspond the vertex of $Q$ such that one of three cases in Lemma \ref{lemm. tops of string mods.} holds;
when $M$ is band, each $\gamma_i$ corresponds to the vertex of $Q$ such that the condition in Lemma \ref{lemm. tops of band mods.} holds).
If $M$ is a string module, then
\begin{align}
P(M) & \cong P(\mathrm{top}M) \cong P(\bigoplus\limits_{i\in \mathfrak{I}} M(c_{\simp}^{\gamma_i}))
\cong \bigoplus\limits_{i\in \mathfrak{I}} P(M(c_{\simp}^{\gamma_i}))
\cong \bigoplus\limits_{i\in \mathfrak{I}} M(c_{\proj}^{\gamma_i}). \nonumber
\end{align}
If $M$ is a band module, then
\begin{align}
P(M) &\cong P(\mathrm{top}M) \cong P(\bigoplus\limits_{i\in \mathfrak{I}} M(c_{\simp}^{\gamma_i})^{\oplus\dim_k V})
\cong \bigoplus\limits_{i\in \mathfrak{I}} P(M(c_{\simp}^{\gamma_i})^{\oplus\dim_k V})
\cong \bigoplus\limits_{i\in \mathfrak{I}} M(c_{\proj}^{\gamma_i})^{\oplus\dim_k V}. \nonumber
\end{align}
\end{proof}

\subsection{Injective modules and injective envelopes}
If a permissible curve $c:(0,1)\to \Surf_A\backslash\partial\Surf_A$ consecutively crosses arcs $\gamma_1, \cdots, \gamma_{r-1}$
and corresponds to an injective module by $M: \Eqcls(\Surf_A) \cup \Htpcls(\Surf_A) \to \mathrm{ind}(\bmod A)$, then
there is a unique $i$ with $1\leq i\leq n-1$ such that the following conditions are satisfied:
\begin{itemize}
\item (\textbf{ISC}) (Injective Submodule Condition) There exits $i$ with $1\leq i\leq r-1$ such that
\begin{itemize}
\item (\textbf{ISC 1}) $m:=p(\gamma_1, \gamma_2, c) = \cdots = p(\gamma_{i-1},\gamma_i, c)$ which is on the left of $c$; and
\item (\textbf{ISC 2}) $m':=p(\gamma_{i}, \gamma_{i+1}, c) = \cdots = p(\gamma_{r-2},\gamma_{r-1}, c)$ which is on the right of $c$;
\end{itemize}
\item (\textbf{MC}$'$) There are no arcs $\gamma$ and $\gamma'$ such that
\begin{itemize}
\item (\textbf{LMC}$'$) $\gamma, \gamma_1, \gamma_2, \cdots, \gamma_i$ have a common endpoint $m$ and they surround $m$
in the counterclockwise order; and
\item (\textbf{RMC}$'$) $\gamma_i,\gamma_{i+1}, \cdots, \gamma_{n}, \gamma'$ have a common endpoint $m'$
and they surround $m'$ in the clockwise order.
\end{itemize}
\end{itemize}

We define injective permissible curves and $\mathfrak{P}_{\inj}(\Surf_A)$ as follows.

\begin{definition} \rm (Injective permissible curves) \label{def. inj. per. curve}
Let $A=kQ/I$ be a gentle algebra and $\Surf_A=(\Surf_A, \M_A, \Gamma_A)$ its marked ribbon surface.
For an arc $\gamma\in \Gamma_A$, a permissible curve $c$ consecutively crossing arcs $\gamma_1, \cdots, \gamma_{r-1}$
is called an {\it injective permissible curve} of $\gamma$, denoted by $c_{\inj}^{\gamma}$,
if $c$ is decided by the pair $(c_1,c_2)$ of two curves $c_1:(0,1]\to \Surf_A\backslash\partial\Surf_A$
and $c_2: [0,1)\to \Surf_A\backslash\partial\Surf_A$, where
\begin{itemize}
\item there is a unique $i$ with $1\leq i\leq r-1$ such that $\gamma=\gamma_i$;
\item $c(0^+)=c_1(0^+)$, $c(1^-)=c_2(1^-)$, $c_1(1)=c_2(0)=\gamma\cap c$;
\item $c$ is obtained by connecting $c_1$ and $c_2$, i.e.,  $c=c_1\cup c_2$; and
\item $c_1$ and $c_2$ consecutively crosses arcs $\gamma_1, \cdots, \gamma_i$ and $\gamma_i, \cdots, \gamma_{r-1}$
such that (\textbf{ISC 1}, \textbf{LMC}$'$) and (\textbf{ISC 2}, \textbf{RMC}$'$) are satisfied, respectively.
\end{itemize}
\end{definition}

We use $\mathfrak{P}_{\inj}(\Surf_A)$ to denote the set of all equivalence classes of injective permissible curves.
Obviously, we have $\mathfrak{P}_{\inj}(\Surf_A)\subseteq\Eqcls(\Surf_A)$. The following results are dual to
Lemma \ref{Lemm. proj} and Theorem \ref{thm. proj. cover}.

\begin{lemma}{\rm(Injective modules)} \label{lemm. inj}
Let $A=kQ/I$ be a gentle algebra and $\Surf_A=(\Surf_A,\M_A,\Gamma_A)$ its marked ribbon surface.
Then there exists a one-to-one mapping between the equivalence classes of injective permissible
curves and the isomorphic classes of indecomposable injective modules which is given by
\begin{align}
\mathfrak{P}|_{\inj}(\Surf_A): \mathfrak{P}_{\inj}(\Surf_A) \to \mathrm{ind}(\inj A), c \mapsto M(c). \nonumber
\end{align}
\end{lemma}

%Furthermore, we have the dual of Theorem \ref{thm. proj. cover}.

\begin{theorem}\label{thm. inj. envelope} {\rm(Injective envelopes of indecomposable modules)}
Let $A=kQ/I$ be a gentle algebra and $\Surf_A=(\Surf_A, \M_A, \Gamma_A)$ its marked ribbon surface.
Let $M = M(c)$ {\rm(}respectively, $M = M(c, \varphi)${\rm)} be a string module {\rm(}respectively, band module{\rm)}
which corresponds to the permissible curve $c$ {\rm(}respectively, the permissible closed curve
$c$ with a Jordan block over a $k$-linear space $V${\rm)}. Then
\[ E(M) \cong \bigoplus\limits_{i\in \mathfrak{I}}M(c_{\inj}^{\gamma_i}) \ \ \ \
\big(\text{respectively}, E(M) \cong \bigoplus\limits_{i\in \mathfrak{I}}M(c_{\inj}^{\gamma_i})^{\oplus \dim_k V}\big), \]
where $\{\gamma_i\mid i\in \mathfrak{I}\}$ and $\mathfrak{I}$ are the same as that in Lemma \ref{lemm. tops of string mods.}
and Proposition \ref{prop. tops of mods.}$(2)$
{\rm(}respectively, Lemma \ref{lemm. tops of band mods.} and Proposition \ref{prop. tops of mods.}$(2)${\rm)}, respectively.
\end{theorem}

\section{\bf Global dimension} \label{Section gl.dim}

\subsection{The descriptions of projective and injective resolutions of simple modules in geometric models}
In this subsection, we consider the minimal projective and injective resolutions of simple modules through marked ribbon surfaces.

\begin{definition} \rm (P-condition and I-condition) \label{def. P-cond.}
Let $A=kQ/I$ be a gentle algebra and $\Surf_A=(\Surf_A,\M_A,\Gamma_A)$ its marked ribbon surface.
Let $\gamma$ be an arc in $\Gamma_A$ with two endpoints $p$, $q$.
We say that $\gamma$ satisfies the {\it {\rm P}-condition at $p$} (respectively, {\it {\rm I}-condition at $p$})
if the following conditions are satisfied:
\begin{itemize}
  \item there are arcs $\gamma_1,  \cdots, \gamma_r$ such that $p$ is the common endpoint of $\gamma$ and $\gamma_i$
    (for any $1\le i\le r$), and

  \item $\gamma, \gamma_1,  \cdots, \gamma_r$ (respectively, $\gamma_1,  \cdots, \gamma_r, \gamma$) surround $p$ in the counterclockwise order.
\end{itemize}
Furthermore, we say that $\gamma$ satisfies the {\it {\rm P}-condition} (respectively, {\it {\rm I}-condition})
if the P-conditions (respectively, I-conditions) at $p$ and $q$ are satisfied.
\end{definition}

\begin{remark} \rm
If an arc $\gamma$ with endpoints $p$ and $q$ satisfies the P-condition at $p$ or $q$, then $M(c_{\proj}^{\gamma})$ is not simple.
\end{remark}

We provide an example in \Pic \ref{P-cond.}. In this figure,  $\gamma$ satisfies the P-condition.

\begin{figure}[htbp]
\begin{center}
\begin{tikzpicture} [scale=0.9]
\draw [line width=1.2pt] (-5,1)-- (5,1);
\draw [line width=1.2pt] (-5,-2)-- (5,-2);
\draw (0,1)-- (0,-2);
\draw [shift={(1,1)}] plot[domain=-3.14:0,variable=\t]({1*1*cos(\t r)+0*1*sin(\t r)},{0*1*cos(\t r)+1*1*sin(\t r)});
\draw [shift={(2,1)}] plot[domain=-3.14:0,variable=\t]({1*2*cos(\t r)+0*2*sin(\t r)},{0*2*cos(\t r)+1*2*sin(\t r)});
\draw [shift={(-1,-2)}] plot[domain=0:3.14,variable=\t]({1*1*cos(\t r)+0*1*sin(\t r)},{0*1*cos(\t r)+1*1*sin(\t r)});
\draw [shift={(-2,-2)}] plot[domain=0:3.14,variable=\t]({1*2*cos(\t r)+0*2*sin(\t r)},{0*2*cos(\t r)+1*2*sin(\t r)});
\draw [shift={(0.5,1)}] plot[domain=-3.14:0,variable=\t]({1*0.5*cos(\t r)+0*0.5*sin(\t r)},{0*0.5*cos(\t r)+1*0.5*sin(\t r)});
\draw [shift={(-0.5,-2)}] plot[domain=0:3.14,variable=\t]({1*0.5*cos(\t r)+0*0.5*sin(\t r)},{0*0.5*cos(\t r)+1*0.5*sin(\t r)});
\begin{scriptsize}
\fill [color=black] (0,1) circle (2.2pt);
\fill [color=black] (0,-2) circle (2.2pt);
\fill [color=black] (1,1) circle (2.2pt);
\fill [color=black] (2,1) circle (2.2pt);
\fill [color=black] (3,1) circle (2.2pt);
\fill [color=black] (4,1) circle (2.2pt);
\fill [color=black] (-1,1) circle (2.2pt);
\fill [color=black] (-2,1) circle (2.2pt);
\fill [color=black] (-3,1) circle (2.2pt);
\fill [color=black] (-4,1) circle (2.2pt);
\fill [color=black] (1,-2) circle (2.2pt);
\fill [color=black] (2,-2) circle (2.2pt);
\fill [color=black] (3,-2) circle (2.2pt);
\fill [color=black] (4,-2) circle (2.2pt);
\fill [color=black] (-1,-2) circle (2.2pt);
\fill [color=black] (-2,-2) circle (2.2pt);
\fill [color=black] (-3,-2) circle (2.2pt);
\fill [color=black] (-4,-2) circle (2.2pt);
\end{scriptsize}
% curve
\draw (0,-0.5) node{$\gamma$};
\draw (0, 1) node[above]{$p$};
\draw (0,-2) node[below]{$q$};
\draw (2,-1) node[below]{$\gamma_1$};
\draw (1, 0) node[below]{$\gamma_2$};
\draw (1.5,1.5-1) node[below]{$\cdots$};
\draw (0.5,1.7-1) node[below]{$\gamma_r$};
\draw (-2, 0) node[above]{$\gamma_1'$};
\draw (-1, -1) node[above]{$\gamma_2'$};
\draw (-1.5,-1.5) node[above]{$\cdots$};
\draw (-0.5,-1.7) node[above]{$\gamma_s'$};
\end{tikzpicture}
\caption{The arc $\gamma$ satisfying the P-condition. }
\label{P-cond.}
\end{center}
\end{figure}
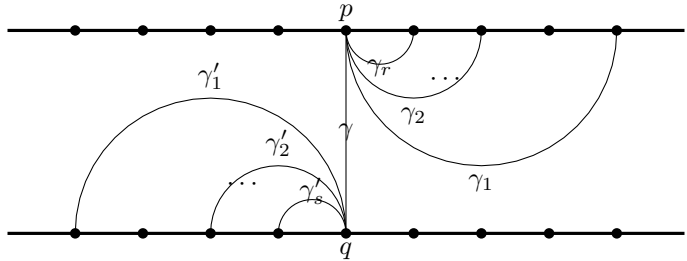

\begin{lemma}\label{lemm. proj. cover of simple mod. 1}
Under the conditions in Definition \ref{def. P-cond.},
let $c=c_{\simp}^\gamma$ be a simple permissible curve, where $\gamma$ is an arc in $\Gamma_A$ with endpoints $p$ and $q$
such that the P-condition is satisfied at $p$ {\rm(}respectively, $q${\rm)} but not at $q$ {\rm(}respectively, $p${\rm)}.
Then the kernel of the projective cover $p(M(c_{\simp}^\gamma)): M(c_{\proj}^\gamma) \twoheadrightarrow M(c_{\simp}^\gamma)$
is the module $M(c')$, where $c'$ is the permissible curve consecutively crossing $\gamma_1, \cdots, \gamma_r$
$($respectively, $\gamma_1', \cdots, \gamma_s')$.
\end{lemma}

\begin{proof}
We only prove the case for $\gamma$ satisfying the {\rm P}-condition at $p$ but not at $q$, and the argument
for the other case is similar.
As shown in \Pic \ref{Proof of Lemm-5.3}, let $\gamma=\tgamma_1$, and $\gamma_j=\tilde{\gamma}_{j+1}$ ($1\le j \le r$).
Then we may suppose $c = c_{\proj}^\gamma = c_{\proj}^{\tilde{\gamma}_1} \simeq c_1c_2\cdots c_{r+2}$, where each $c_t$
($2\le t\le r+1$) is a segment between $\tilde{\gamma}_{t-1}$ and $\tilde{\gamma}_t$, and $c_1$, $c_{r+2}$ are end segments of $c$.
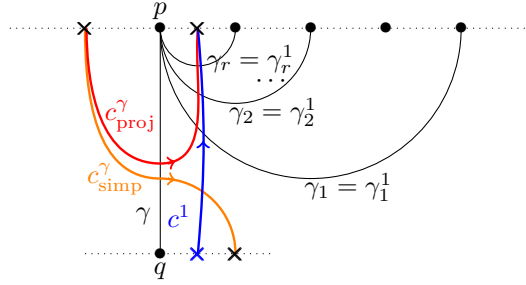
\begin{figure}[htbp]
\begin{center}
\begin{tikzpicture}
\draw[dotted] (-2,2)-- (5,2);
\draw[dotted] (-1,-1)-- (1.5,-1);
\draw (0,2)-- (0,-1);
\draw (0,2) arc ( 180: 360:0.5); \draw (0,2) arc ( 180: 360:1.0); \draw (0,2) arc ( 180: 360:2.0);
% curve simp
\draw[orange, line width = 0.9pt] (-1,2) to[out=-90,in=180] (0,0) to[out=0,in=90]  (1,-1);
\draw[->][orange, line width = 0.9pt] (0.1,0) --(0.2,-0.02);
\draw[orange] (-0.6,0) node{$c_\simp^{\gamma}$};
% curve proj
\draw[red, line width = 0.9pt] (-0.98,2) to[out=-90,in=180] (0,0.2) to[out=0,in=-95] (0.5,2);
\draw[->][red, line width = 0.9pt] (0.1, 0.2) --(0.2,0.2+0.03);
\draw[red] (-0.4,0.5) node[above]{$c_\proj^{\gamma}$};
% curve c^1
\draw[blue, line width = 0.9pt] (0.5,-1)  node{$\pmb{\times}$} to[out=85,in=-85]  (0.5,2);
\draw[->][blue, line width = 0.9pt] (0.57,0.45) -- (0.57,0.5);
\draw[blue, line width = 0.9pt] (0.25,-0.5) node{$c^1$};
% marked/extra points
\draw [color=black] (0, 2) node{$\bullet$};
\draw [color=black] (1, 2) node{$\bullet$};
\draw [color=black] (2, 2) node{$\bullet$};
\draw [color=black] (3, 2) node{$\bullet$};
\draw [color=black] (4, 2) node{$\bullet$};
\draw [color=black] (1,-1) node{$\pmb{\times}$};
\draw [color=black] (0,-1) node{$\bullet$};
\draw [color=black] (-1,2) node{$\pmb{\times}$};
\draw [color=black] (0.5,2) node{$\pmb{\times}$};
% curve
\draw (0,-0.5) node[left]{$\gamma$};
\draw (0, 2) node[above]{$p$};
\draw (0,-1) node[below]{$q$};
\draw (0.5+2, 0.2+0) node[below]{$\gamma_1=\gamma^1_1$};
\draw (0.5+1, 0.2+1) node[below]{$\gamma_2=\gamma^1_2$};
\draw (1.5,1.5) node[below]{$\cdots$};
\draw (1.2,0.2+1.7) node[below]{$\gamma_r=\gamma^1_r$};
\end{tikzpicture}
\caption{
The projective cover $p(M(c_{\simp}^{\gamma})): M(c_{\proj}^{\gamma}) \to M(c_{\simp}^{\gamma})$
of $M(c_{\simp}^{\gamma})$ and its kernel $\Ker\ p(M(c_{\simp}^{\gamma}))$ $\cong$ $M(c^1)$,
where $c^1 \simeq \mathrm{nrot}_{c_{\proj}^{\gamma}}(c_{\simp}^{\gamma})$  by Proposition \ref{prop. quotientmod.}.
}
\label{Proof of Lemm-5.3}
\end{center}
\end{figure}
Let $c''=c_{\simp}^\gamma = c_{\simp}^{\tilde{\gamma}_1}$ and let
$c'$ be the permissible curve consecutively crossing $\gamma_1, \cdots, \gamma_r$. Then $c'' \simeq c_1c_2^{\mathfrak{T}}$ and
$c'\simeq \mathrm{nrot}_c(c'')$ (by Definition \ref{def-3.4}) and
the positional relationship of $c$, $c'$ and $c''$ satisfies the conditions in Proposition \ref{prop. s. e. s.}
(that is, $c$ consecutively crosses $\tgamma_2$ and $\tgamma_1$, $c'$ crosses only $\tgamma_2$ and $c''$ crosses $\tgamma_1$,
where $\tgamma_2$ and $\tgamma_1$ are viewed as arcs $\gamma'$ and $\gamma''$ in Proposition \ref{prop. s. e. s.}, respectively).
Thus we get the following exact sequence
$$0 \longrightarrow M(c') \longrightarrow M(c) \longrightarrow M(c'') \longrightarrow 0, $$
by Theorem \ref{thm. s. e. s.}(1).
\end{proof}

\begin{lemma}\label{lemm. proj. cover of simple mod. 2}
Under the conditions in Definition \ref{def. P-cond.},
let $c=c_{\simp}^\gamma$ be a simple permissible curve, where $\gamma$ is an arc in $\Gamma_A$ with endpoints $p$ and $q$
such that the {\rm P}-condition is satisfied.
Then the kernel of the projective cover $p(M(c_{\simp}^\gamma)): M(c_{\proj}^\gamma) \twoheadrightarrow M(c_{\simp}^\gamma)$
is the module $M(c_{\I}'\oplus c_{\II}')$, where $c_{\I}'$ consecutively crosses $\gamma_1', \cdots, \gamma_s'$ and $c_2'$
consecutively crosses $\gamma_1, \cdots, \gamma_r$.
\end{lemma}

\begin{proof}
Let $\gamma_i'=\tilde{\gamma}_{s-i+1}$ ($1\le i\le s$), $\gamma=\tilde{\gamma}_{s+1}$, and
$\gamma_j=\tilde{\gamma}_{s+j+1}$ ($1\le j \le r$). Then we may suppose
$c = c_{\proj}^\gamma = c_{\proj}^{\tilde{\gamma}_{s+1}} \simeq c_1c_2\cdots c_{s+r+2}$ where each $c_t$
($2\le t\le s+r+1$) is a segment between $\tilde{\gamma}_{t-1}$ and $\tilde{\gamma}_t$, and $c_1$, $c_{s+r+2}$ are end segments of $c$.
Let $c''=c_{\simp}^\gamma$, and let $x$ be the intersection of $c$ and $c''$ (note that $\sharp(c\cap c'')=1$ in this case)
such that $c_{\I}'', c_{\II}''$ are the curves obtained by $x$ dividing $c''$.
%Then $\tilde{c}_{\I}'\simeq \mathrm{prot}_{c}(c_{\I}'')$ and $\tilde{c}_{\II}'\simeq \mathrm{prot}_{c}(c_{\II}'')$ by Definition \ref{def-3.5}, and
Then the positional relationship of $c, \mathrm{prot}_{c}(c_{\I}''), \mathrm{prot}_{c}(c_{\II}'')$
and $c''$ satisfies the conditions in Proposition \ref{prop. quotientmod.}.
We get the following exact sequence
\[0 \longrightarrow M(\mathrm{prot}_{c}(c_{\I}'')\oplus \mathrm{prot}_{c}(c_{\II}'') )\longrightarrow M(c)
\mathop{ -\!\!\!-\!\!\!-\!\!\! \longrightarrow}\limits^{p(M(c_{\simp}^{\gamma}))}  M(c'') \longrightarrow 0 \]
by Theorem \ref{thm. s. e. s.}(3).
Notice that $\mathrm{prot}_{c}(c_{\I}'') \simeq c_{\I}' $ and $\mathrm{prot}_{c}(c_{\II}'') \simeq c_{\II}' $, thus
$$\Ker p(M(c_{\simp}^{\gamma})) = M(\mathrm{prot}_{c}(c_{\I}'') \oplus \mathrm{prot}_{c}(c_{\II}'')) \cong M(c_{\I}' \oplus c_{\II}').$$
\end{proof}

\begin{lemma} \label{lemm. ker. of proj. cover}
Under the conditions Definition \ref{def. P-cond.} and Lemma \ref{lemm. proj. cover of simple mod. 1},
let $c^1$ be the permissible curve $c'$ corresponding to $\Ker p(M(c_{\simp}^\gamma))$,
where $p(M(c_{\simp}^\gamma))$ is given by Lemma \ref{lemm. proj. cover of simple mod. 1}.
Then $P(M(c^1))=P(M(c_{\proj}^{\gamma_1}))$ is indecomposable, and the kernel of
$P(M(c^1)) \twoheadrightarrow M(c^1)$ is $M(\mathrm{nrot}_{c_{\proj}^{\gamma_1}}(c^1))$.
To be precise, if $\gamma^2_1, \gamma^2_2, \cdots, \gamma^2_s$ are arcs with the endpoint of $\gamma_1$ which is not $p$
such that $\gamma, \gamma_1, \gamma^2_1$ are sides of the same elementary polygon of $\Surf_A$,
then $\mathrm{nrot}_{c_{\proj}^{\gamma_1}}(c^1) \simeq c^2$, where $c^2$ is the permissible curve consecutively
crosses $\gamma^2_1, \gamma^2_2, \cdots, \gamma^2_s$.
\end{lemma}

\begin{proof}
Let $\gamma_{r-i}=\tilde{\gamma}_{i+1}$ $(0\le i\le r-1)$, $\gamma^2_j = \tilde{\gamma}_{r+j}$ $(1\le j \le s)$.
Then $c'':=c^1=c_1''c_2''\cdots c_{r+1}''$, where each $c_i''$ ($2\le i \le r$) is a segment between
$\tilde{\gamma}_{i-1}$ and $\tilde{\gamma}_{i}$, and $c_1''$, $c_{r+1}''$ are its end segments;
and $c:=c_{\proj}^{\gamma_1}=c_{\proj}^{\tilde{\gamma}_r}=c_1c_2\cdots c_{r+s+1}$ where each $c_j$ ($ 2 \le j\le r+s$)
is a segment between $\tilde{\gamma}_{j-1}$ and $\tilde{\gamma}_{j}$, and $c_1$, $c_j$ are its end segments.
Thus $c'' = c_1''c_2''\cdots c_{r+1}'' \simeq c_1c_2\cdots c_{r+1}^{\mathfrak{T}}$.
We have that the permissible curve $c'=c_{r+1}^{\mathfrak{T}}c_{r+2}\cdots c_{r+s+1}$
which consecutively crosses $\tilde{\gamma}_{r+1} = \gamma^2_1, \tilde{\gamma}_{r+2}
= \gamma^2_2, \cdots, \tilde{\gamma}_{r+s} = \gamma^2_s$, and
the positional relationship of $c$, $c'$ and $c''$ satisfies the conditions in Proposition \ref{prop. s. e. s.}.
So, by Theorem \ref{thm. s. e. s.}(1), we get the following short exact sequence
$$0 \longrightarrow M(c') \longrightarrow M(c) \longrightarrow M(c'') \longrightarrow 0,$$
where $M(c'')=M(c^1)$, $M(c)=M(c_{\proj}^{\gamma_1})$, and $M(c')=M(c_{r+1}^{\mathfrak{T}}c_{r+2}\cdots c_{r+s+1})
\cong M(\mathrm{nrot}_{c_{\proj}^{\gamma_1}}(c^1))$ (Note that $c'\simeq \mathrm{nrot}_{c_{\proj}^{\gamma_1}}(c^1)$
since $\gamma$ and $\gamma_1$ are sides of the same elementary polygon).
\end{proof}

\begin{figure}[htbp]
\begin{center}
\begin{tikzpicture}
\draw [dotted] (-1,2)-- (8,2);
\draw [dotted] (-1,-1)-- (8,-1);
\draw (0,2)-- (0,-1);
\draw (0,2) arc ( 180: 360:0.5); \draw (0,2) arc ( 180: 360:1.0); \draw (0,2) arc ( 180: 360:2.0);
\draw (4,2) arc ( 180: 360:0.5); \draw (4,2) arc ( 180: 360:1.0); \draw (4,2) arc ( 180: 360:2.0);
% curve simp
\draw[gray, dotted, line width = 0.9pt] (-1,2) to[out=-90,in=180] (0,0) to[out=0,in=90]  (1,-1);
\draw[->][gray, line width = 0.9pt] (0.1,0) --(0.2,-0.02);
\draw[gray] (-0.6,0) node{$c_\simp^{\gamma}$};
% curve proj
\draw[gray, dotted, line width = 0.9pt] (-0.98,2) to[out=-90,in=180] (0,0.2) to[out=0,in=-95] (0.5,2);
\draw[->][gray, line width = 0.9pt] (0.1, 0.2) --(0.2,0.2+0.03);
\draw[gray] (-0.4,0.5) node[above]{$c_\proj^{\gamma}$};
% curve c^1
\draw[orange, line width = 0.9pt] (0,-1) (0.5,-1)  node{$\pmb{\times}$} to[out=85,in=-85]  (0.5,2);
\draw[->][orange, line width = 0.9pt] (0.57,0.45) -- (0.57,0.5);
\draw[orange, line width = 0.9pt] (0.7,-0.2) node[left]{$c^1$};
% curve proj(c^1)
\draw[red, line width = 0.9pt] (0.5, 2) arc  (180: 360:2.0);
\draw[->][red] (2.5,0) -- (2.51, 0);
\draw[red] (2.5,0) node[above]{$c_\proj^{\gamma^1_1}$};
% curve kernel
\draw[blue, line width = 0.9pt] (0, 2) arc  (180: 360:2.25);
\draw[->][blue] (2.25,-0.25) -- (2.26,-0.25);
\draw[blue] (2.25,-0.25) node[below]{$\mathrm{nrot}_{c_\proj^{\gamma^1_1}}(c^1)$};
% marked/extra point
\draw [color=black] (-1,2) node{$\pmb{\times}$};
\draw [color=black] (0, 2) node{$\bullet$};
\draw [color=black] (0,-1) node{$\bullet$};
\draw [color=black] (1, 2) node{$\bullet$};
\draw [color=black] (2, 2) node{$\bullet$};
\draw [color=black] (3, 2) node{$\bullet$};
\draw [color=black] (4, 2) node{$\bullet$};
\draw [color=black] (5, 2) node{$\bullet$};
\draw [color=black] (6, 2) node{$\bullet$};
\draw [color=black] (7, 2) node{$\bullet$};
\draw [color=black] (8, 2) node{$\bullet$};
\draw [color=black] (1,-1) node{$\pmb{\times}$};
\draw [color=black] (0.5, 2) node{$\pmb{\times}$};
\draw [color=black] (4.5, 2) node{$\pmb{\times}$};
% curve
\draw (0,-0.5) node[left]{$\gamma$};
\draw (0, 2) node[above]{$p$};
\draw (0,-1) node[below]{$q$};
%\draw (0.5+2, 0.2+0) node[below]{$\gamma_1=\gamma^1_1$};
%\draw (0.5+1, 0.2+1) node[below]{$\gamma_2=\gamma^1_2$};
\draw (1.5,1.5) node[below]{$\cdots$};
%\draw (1.25,0.2+1.7) node[below]{$\gamma_r=\gamma^1_r$};
\draw (4+2, 0) node[below]{$\gamma^2_1$};
\draw (4+1, 1) node[below]{$\gamma^2_2$};
\draw (4+1.5,1.5) node[below]{$\cdots$};
\draw (4+0.5,1.7) node[below]{$\gamma^2_s$};
\end{tikzpicture}
\caption{
The projective cover $p(M(c_{\simp}^{\gamma})): M(c_{\proj}^{\gamma_1^1}) \to M(c_{\simp}^{\gamma})$ of $M(c_{\simp}^{\gamma})$
and its kernel $\Ker\ p(M(c_{\simp}^{\gamma})) \cong M(c^2)$, where $c^2 \simeq \mathrm{nrot}_{c_{\proj}^{\gamma_1}}(c^1)$ by Proposition \ref{prop. quotientmod.}.
}
\label{ker of proj. cover}
\end{center}
\end{figure}
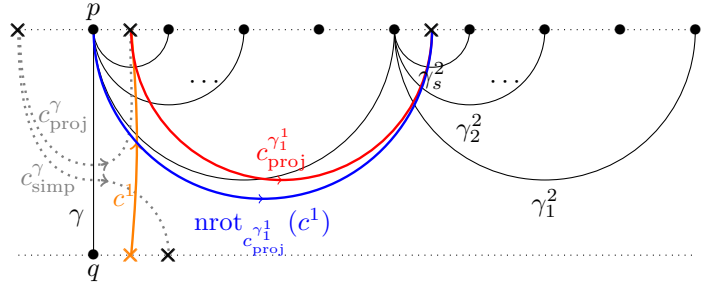

Keep the notations in Lemma \ref{lemm. ker. of proj. cover} and denote by $\gamma^1_i$
the permissible curve $\gamma_i$ for any $1\le i \le r$.
In \Pic \ref{ker of proj. cover}, $p(M(c_{\simp}^{\gamma})): M(c_{\proj}^{\gamma})
\twoheadrightarrow M(c_{\simp}^{\gamma})$ is the projective cover of $M(c_{\simp}^{\gamma})$ which is simple, and
the module $M(c^1)$ corresponding to the orange permissible curve $c^1$ is the kernel of $p(M(c_{\simp}^{\gamma}))$.
In addition, $c_{\proj}^{\gamma^1_1}$, the red permissible curve in this Figure, provides the projective cover of $M(c^1)$,
that is, $p(M(c^1)): M(c_{\proj}^{\gamma^1_1}) \twoheadrightarrow M(c^1)$. % Deleting {is the projective cover of $M(c^1)$.}
Then $\mathrm{nrot}_{c_{\proj}^{\gamma^1_1}}(c^1)$ is the kernel of $p(M(c^1))$.
Furthermore, let $\mathrm{nrot}_{c_\proj^{\gamma^1_1}}(c^1)=c^2$.
If there are arcs $\gamma^3_1, \cdots, \gamma^3_t$ having a common endpoint which is the endpoint of $\gamma^2_1$
but not the common endpoint of $\gamma^2_1, \cdots, \gamma^2_s$
such that $\gamma, \gamma^1_1, \gamma^2_1, \gamma^3_1$ are sides of the same elementary polygon of $\Surf_A$,
then the kernel of $M(c_{\proj}^{\gamma^2_1}) \twoheadrightarrow M(c^2)$ is isomorphic to $M(c^3)$,
where $c^3 \simeq \mathrm{nrot}_{c_{\proj}^{\gamma^2_1}}(c^2)$ is a permissible curve consecutively
crossing $\gamma^3_1, \cdots, \gamma^3_t$.
Repeating this process, we get the minimal projective resolution of $M(c_{\simp}^{\gamma})$ as follows:
\[\cdots \longrightarrow M(c_{\proj}^{\gamma^2_1}) \longrightarrow M(c_{\proj}^{\gamma^1_1})
\longrightarrow M(c_{\proj}^{\gamma}) \longrightarrow M(c_{\simp}^{\gamma}) \longrightarrow 0. \]
Thus we obtain a one-one mapping $M(c_{\proj}^{\gamma^t_1}) \mapsto \gamma^t_1$,
where $\gamma^0_1=\gamma$ and $\gamma^0_1, \gamma^1_1, \cdots$
are sides of the same elementary polygon of $\Surf_A$, see \Pic \ref{ker of proj. cover 2}.

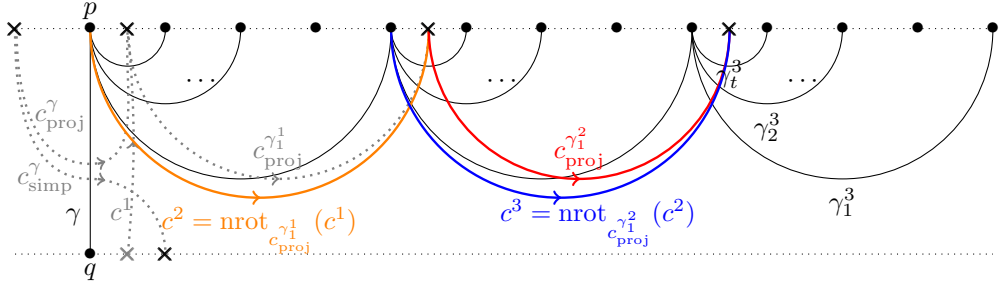
\begin{figure}[htbp]
\begin{center}
\begin{tikzpicture}
\draw [dotted] (-1,2)-- (12,2);
\draw [dotted] (-1,-1)-- (12,-1);
\draw (0,2)-- (0,-1);
\draw (0,2) arc ( 180: 360:0.5); \draw (0,2) arc ( 180: 360:1.0); \draw (0,2) arc ( 180: 360:2.0);
\draw (4,2) arc ( 180: 360:0.5); \draw (4,2) arc ( 180: 360:1.0); \draw (4,2) arc ( 180: 360:2.0);
\draw (8,2) arc ( 180: 360:0.5); \draw (8,2) arc ( 180: 360:1.0); \draw (8,2) arc ( 180: 360:2.0);
% curve simp
\draw[gray, dotted, line width = 0.9pt] (-1,2) to[out=-90,in=180] (0,0) to[out=0,in=90]  (1,-1);
\draw[->][gray, line width = 0.9pt] (0.1,0) --(0.2,-0.02);
\draw[gray] (-0.6,0) node{$c_\simp^{\gamma}$};
% curve proj
\draw[gray, dotted, line width = 0.9pt] (-0.98,2) to[out=-90,in=180] (0,0.2) to[out=0,in=-95] (0.5,2);
\draw[->][gray, line width = 0.9pt] (0.1, 0.2) --(0.2,0.2+0.03);
\draw[gray] (-0.4,0.5) node[above]{$c_\proj^{\gamma}$};
% curve c^1
\draw[gray, dotted, line width = 0.9pt] (0.5,-1)  node{$\pmb{\times}$} to[out=85,in=-85]  (0.5,2);
\draw[->][gray, dotted, line width = 0.9pt] (0.57,0.45) -- (0.57,0.5);
\draw[gray] (0.7,-0.4) node[left]{$c^1$};
% curve proj(c^1)
\draw[gray, dotted, line width = 0.9pt] (0.5, 2) arc  (180: 360:2.0);
\draw[->][gray, dotted, line width = 0.9pt] (2.5,0) -- (2.51, 0);
\draw[gray] (2.5,0) node[above]{$c_\proj^{\gamma^1_1}$};
% curve kernel
\draw[orange, line width = 0.9pt] (0, 2) arc  (180: 360:2.25);
\draw[->][orange][line width = 1pt] (2.25,-0.25) -- (2.26,-0.25);
\draw[orange] (2.25,-0.25) node[below]{$c^2=\mathrm{nrot}_{c_\proj^{\gamma^1_1}}(c^1)$};
% curve proj(c^2)
\draw[red, line width = 0.9pt] (4+0.5, 2) arc  (180: 360:2.0);
\draw[->][red][line width = 1pt] (4+2.5,0) -- (4+2.51, 0);
\draw[red] (4+2.5,0) node[above]{$c_\proj^{\gamma^2_1}$};
% curve kernel
\draw[blue, line width = 0.9pt] (4+0, 2) arc  (180: 360:2.25);
\draw[->][blue][line width = 1pt] (4+2.25,-0.25) -- (4+2.26,-0.25);
\draw[blue] (4+2.75,-0.15) node[below]{$c^3=\mathrm{nrot}_{c_\proj^{\gamma^2_1}}(c^2)$};
% marked/extra points
\draw [color=black] (-1,2) node{$\pmb{\times}$};
\draw [color=black] (0, 2) node{$\bullet$};
\draw [color=black] (0,-1) node{$\bullet$};
\draw [color=black] (1, 2) node{$\bullet$};
\draw [color=black] (2, 2) node{$\bullet$};
\draw [color=black] (3, 2) node{$\bullet$};
\draw [color=black] (4, 2) node{$\bullet$};
\draw [color=black] (5, 2) node{$\bullet$};
\draw [color=black] (6, 2) node{$\bullet$};
\draw [color=black] (7, 2) node{$\bullet$};
\draw [color=black] (8, 2) node{$\bullet$};
\draw [color=black] (9, 2) node{$\bullet$};
\draw [color=black] (10, 2) node{$\bullet$};
\draw [color=black] (11, 2) node{$\bullet$};
\draw [color=black] (12, 2) node{$\bullet$};
\draw [color=black] (1,-1) node{$\pmb{\times}$};
\draw [color=black] (0.5, 2) node{$\pmb{\times}$};
\draw [color=black] (4.5, 2) node{$\pmb{\times}$};
\draw [color=black] (8.5, 2) node{$\pmb{\times}$};
% curve
\draw (0,-0.5) node[left]{$\gamma$};
\draw (0, 2) node[above]{$p$};
\draw (0,-1) node[below]{$q$};
% \draw (2, 0) node[below]{$\gamma^1_1$};
% \draw (1, 1) node[below]{$\gamma^1_2$};
\draw (1.5,1.5) node[below]{$\cdots$};
% \draw (0.5,1.7) node[below]{$\gamma^1_r$};
% \draw (4+2, 0) node[below]{$\gamma^2_1$};
% \draw (4+1, 1) node[below]{$\gamma^2_2$};
\draw (4+1.5,1.5) node[below]{$\cdots$};
% \draw (4+0.5,1.7) node[below]{$\gamma^2_s$};
\draw (8+2, 0) node[below]{$\gamma^3_1$};
\draw (8+1, 1) node[below]{$\gamma^3_2$};
\draw (8+1.5,1.5) node[below]{$\cdots$};
\draw (8+0.5,1.7) node[below]{$\gamma^3_t$};
\end{tikzpicture}
\caption{
Keep the notations in Figure \ref{ker of proj. cover}.
Then the positional relationship of $c^2$, $c_{\proj}^{\gamma^2_1}$ and
$c^3$ satisfies the conditions in Proposition \ref{prop. quotientmod.}. }
\label{ker of proj. cover 2}
\end{center}
\end{figure}

In the case that $\gamma$ satisfies the {\rm P}-condition at $q$ but not at $p$, we get dually
the minimal projective resolution of $M(c_{\simp}^{\gamma})$ as follows:
\[\cdots \longrightarrow M(c_{\proj}^{\gamma'^2_1}) \longrightarrow M(c_{\proj}^{\gamma'^1_1})
\longrightarrow M(c_{\proj}^{\gamma}) \longrightarrow M(c_{\simp}^{\gamma}) \longrightarrow 0. \]
Moreover, if $\gamma$ satisfies the {\rm P}-condition, then the minimal projective resolution
of $M(c_{\simp}^{\gamma})$ is as follows:
\[\cdots
\longrightarrow M(c_{\proj}^{\gamma^2_1}) \oplus  M(c_{\proj}^{\gamma'^2_1})
\longrightarrow M(c_{\proj}^{\gamma^1_1})\oplus M(c_{\proj}^{\gamma'^1_1})
\longrightarrow M(c_{\proj}^{\gamma}) \longrightarrow M(c_{\simp}^{\gamma}) \longrightarrow 0. \]

\begin{proposition} \label{prop. pdM and idM}
Let $A=kQ/I$ be a gentle algebra and $\Surf_A=(\Surf_A,\M_A,\Gamma_A)$ its marked ribbon surface.
Let $\gamma^0$ be an arc $\gamma^0:(0,1)\to \Surf_A\backslash\partial\Surf_A$ whose endpoints are
$p:=\gamma^0(0^+)$ and $q:=\gamma^0(1^-)$. As shown in \Pic \ref{fig:prop-5.6}, assume that
\begin{itemize}
\item {\rm(\textbf{PD})} $($respectively, {\rm(\textbf{ID}))}
there are arcs $\gamma^1,\cdots,\gamma^n :(0,1)\to \Surf_A\backslash\partial\Surf_A $ such that following hold:
\begin{itemize}
  \item {\rm(PD1)}$=${\rm(ID1)}
    $\gamma^0, \gamma^1,\cdots, \gamma^n$ are sides of the same elementary polygon of $\Surf_A$;

  \item {\rm(PD2)}$=${\rm(ID2)}
    $\gamma^1(0^+) = q$ and $\gamma^{i-1}(1^-) = \gamma^{i}(0^+)$ for any $2\le i\le n$;

  \item {\rm(PD3)} {\rm(}respectively, {\rm(ID3))}
    in the clockwise $($respectively, counterclockwise$)$ order at the common endpoint of
    $\gamma^{i-1}$ and $\gamma^i$, $\gamma^i$ appears after $\gamma^{i-1}$ for any $1\le i\le n$;

  \item {\rm(PD4)}$=${\rm(ID4)}
    for any $\gamma \in \Gamma_A$, none of the arcs $\gamma^0, \gamma^1,\cdots,\gamma^n, \gamma$
    satisfies one of above three conditions;
\end{itemize}
and assume that
\item {\rm (\textbf{PD})$'$} $($respectively, {\rm (\textbf{ID})$'$)}
there are arcs $\gamma'^1,\cdots,\gamma'^m :(0,1)\to \Surf_A\backslash\partial\Surf_A $ such that following hold.
\begin{itemize}
  \item {\rm(PD1)$'$}$=${\rm(ID1)$'$}
    $\gamma'^0 := \gamma^0, \gamma'^1,\cdots, \gamma'^m$ are sides of the same elementary polygon of $\Surf_A$;

  \item {\rm(PD2)$'$}$=${\rm(ID2)$'$}
    $\gamma'^1(0^+)=p$ and $\gamma'^{j-1}(1^-) = \gamma'^{j}(0^+)$ for any $2\le j\le m$;

  \item {\rm(PD3)$'$} {\rm(}respectively, {\rm(ID3)$'$)}
    in the counterclockwise $($respectively, clockwise$)$ order at the common endpoint of
    $\gamma'^{j-1}$ and $\gamma'^j$, $\gamma'^j$ appears after $\gamma^{j-1}$ for any $1\le j\le m$;

  \item {\rm(PD4)$'$}$=${\rm(ID4)$'$}
    for any $\gamma' \in \Gamma_A$, none of the arcs $\gamma'^0, \gamma'^1,\cdots,\gamma'^m, \gamma'$
    satisfies one of above three conditions.
\end{itemize}
\end{itemize}
Then the projective dimension $($respectively, injective dimension$)$ of $M(c_{\simp}^{\gamma^0})$ is $\max\{n, m\}$.
\end{proposition}

\begin{figure}[htbp]
\begin{center}
\begin{tikzpicture}
\draw (-1.732,-1) node{$\bullet$} -- (-1.732,1);
\draw[rotate around={-60:(0,0)}] (-1.732,-1)node{$\bullet$} -- (-1.732,1);
\draw[rotate around={-120:(0,0)}] (-1.732,-1)node{$\bullet$} -- (-1.732,1) [dotted];
\draw[rotate around={-180:(0,0)}] (-1.732,-1)node{$\bullet$} -- (-1.732,1);
\draw[rotate around={-240:(0,0)}] [line width=1.2 pt] (-1.732,-1)node{$\bullet$} -- (-1.732,1)node{$\bullet$};
% direction
\draw [->] (-1.732,-0.05) -- (-1.732, 0.05);
\draw[rotate around={-60:(0,0)}] [->] (-1.732,-0.05) -- (-1.732, 0.05);
\draw[rotate around={-120:(0,0)}] [->] (-1.732,-0.05) -- (-1.732, 0.05);
\draw[rotate around={-180:(0,0)}] [->] (-1.732,-0.05) -- (-1.732, 0.05);
\draw (-2,0) node{$\gamma^0$};
\draw[rotate around={-60:(0,0)}] (-2,0) node{$\gamma^1$};
\draw[rotate around={-180:(0,0)}] (-2,0) node{$\gamma^n$};
\draw[rotate around={-240:(0,0)}] (-2,0) node{$\partial\Surf_A$};
\draw (-1.732,-1) node[left]{$p = \gamma^{0}(0^+)$};
\draw (-1.732, 1) node[left]{$q = \gamma^{0}(1^-)$};
\draw (0,-2.5) node {(1): $(\mathbf{PD})$};
\end{tikzpicture}
\ \ \ \
\begin{tikzpicture}
\draw ( 1.732,-1) node{$\bullet$} -- ( 1.732,1);
\draw[rotate around={-60:(0,0)}] ( 1.732,-1) node{$\bullet$} -- ( 1.732,1);
\draw[rotate around={-120:(0,0)}] ( 1.732,-1) node{$\bullet$} -- ( 1.732,1) [dotted];
\draw[rotate around={-180:(0,0)}] ( 1.732,-1) node{$\bullet$} -- ( 1.732,1);
\draw[rotate around={-240:(0,0)}] ( 1.732,-1) node{$\bullet$} -- ( 1.732,1) [line width = 1.2pt];
% direction
\draw [->] ( 1.732,-0.05) -- ( 1.732, 0.05);
\draw[rotate around={-60:(0,0)}] [->] ( 1.732,-0.05) -- ( 1.732, 0.05);
\draw[rotate around={-120:(0,0)}] [->] ( 1.732,-0.05) -- ( 1.732, 0.05);
\draw[rotate around={-180:(0,0)}] [->] ( 1.732,-0.05) -- ( 1.732, 0.05);
\draw (2,0) node[right]{$\gamma'^0 = \gamma^0$};
\draw[rotate around={-60:(0,0)}] (2,0) node{$\gamma'^1$};
\draw[rotate around={-180:(0,0)}] (2,0) node{$\gamma'^m$};
\draw[rotate around={-240:(0,0)}] (2,0) node{$\partial\Surf_A$};
\draw ( 1.732,-1) node[right]{$p = \gamma^{0}(0^+)$};
\draw ( 1.732, 1) node[right]{$q = \gamma^{0}(1^-)$};
\draw (0,-2.5) node {(2): $(\mathbf{PD})'$};
\end{tikzpicture}
\end{center}
\caption{$\gamma^0, \gamma^1, \cdots, \gamma^n$ are arcs satisfying $(\mathbf{PD})$,
and $\gamma'^0, \gamma'^1, \cdots, \gamma'^m$ are arcs satisfying $(\mathbf{PD})'$.
The figures of $(\mathbf{ID})$ and $(\mathbf{ID})'$ are dual. }
\label{fig:prop-5.6}
\end{figure}
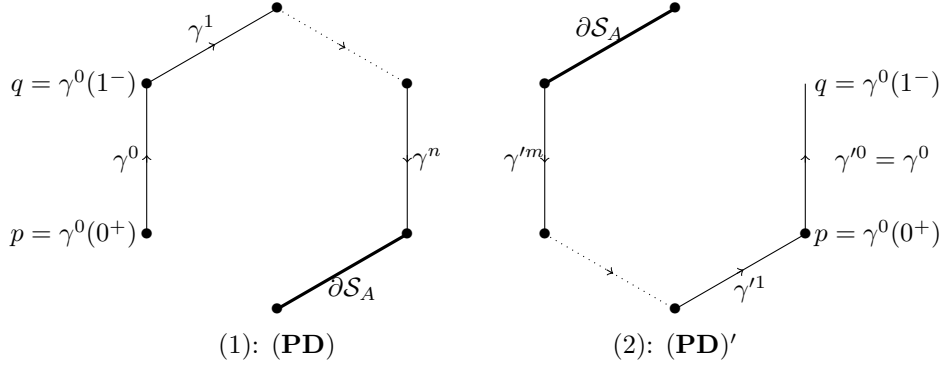

\begin{proof}
%We prove the case of projective resolution, the case of injective resolution is similar.
Notice that $M(c_{\proj}^{\gamma}) \mapsto \gamma$ for any $\gamma\in \{\gamma^i\mid 1\le i\le n\}
\cup \{\gamma^j\mid 1\le j\le m\} \cup \{\gamma^0\}$ by Lemmas \ref{lemm. proj. cover of simple mod. 1},
\ref{lemm. proj. cover of simple mod. 2} and \ref{lemm. ker. of proj. cover},
thus the minimal projective resolution of $M(c_{\simp}^{\gamma^0})$ is as follows:
\begin{align}
0 & \longrightarrow M(c_{\proj}^{\gamma^{\max\{n,m\}}}) \oplus M(c_{\proj}^{\gamma'^{\max\{n,m\}}})
\longrightarrow \cdots \nonumber\\
& \longrightarrow M(c_{\proj}^{\gamma^2}) \oplus M(c_{\proj}^{\gamma'^2})
\longrightarrow M(c_{\proj}^{\gamma^1}) \oplus M(c_{\proj}^{\gamma'^1})
\longrightarrow M(c_{\proj}^{\gamma^0}) \longrightarrow M(c_{\simp}^{\gamma^0}) \longrightarrow 0. \label{proj. res.}
\end{align}
If $m\ge n$, then $M(c_{\proj}^{\gamma^{i}})=0$ for any $i>n$, and so $\pdim M(c_{\simp}^{\gamma^0})= m$.
If $n\ge m$, then $\pdim M(c_{\simp}^{\gamma^0})= n$ similarly. Dually,
we can prove the case for injective dimension.
\end{proof}

The exact sequence (\ref{proj. res.}) is a description of the minimal projective resolution of a simple module in marked ribbon surfaces.
We will give two examples in Section \ref{Section appl.} to illustrate it, see Examples \ref{exp. 1} and \ref{exp. 2}.

\subsection{The descriptions of global dimension in geometric models}

\begin{definition} \rm (Consecutive arcs numbers) \label{def. arc no.}
Let $A=kQ/I$ be a gentle algebra and $\Surf_A = (\Surf_A, \M_A, \Gamma_A)$ its marked ribbon surface.
Recall from Definition \ref{def. m. r. s.} that $\Gamma_A$ divides $\Surf_A$ to some elementary polygons $\{\Delta_i\mid 1\le i \le d\}$
whose sides are arcs in $\Gamma_A$ except one side which is a boundary arc.
We define $\mathfrak{C}(\Delta_i)$, the {\it consecutive arcs numbers} of $\Delta_i$, as the number of sides belonging to $\Gamma_A$,
if the unmarked boundary component of $\Surf_A$ is not a side of $\Delta_i$;
otherwise, $\mathfrak{C}(\Delta_i) = \infty$.
\end{definition}

\begin{remark}\rm \label{rk. polygon-inf}
%If $\Delta_i$ has a side which is an unmarked boundary component $b$ of $\Surf_A$,
%then all sides in $\Gamma_A$ of $\Delta_i$ surround $b$, and we call it an {\it $\infty$-elementary polygon}.
An elementary polygon $\Delta_i$ is called an {\it $\infty$-elementary polygon} if
$\Delta_i$ has a side which is an unmarked boundary component $b$ of $\Surf_A$.
Obviously, $\mathfrak{C}(\Delta_i)=\infty$ if and only if $\Delta_i$ is $\infty$-elementary.
In this case, for any side belonging to $\Gamma_A$ of $\Delta_i$, say $\gamma^0$,
there are arcs $\gamma^0, \gamma^1, \cdots, \gamma^n$ satisfying the conditions (PD1), (PD2) and (PD3)
in Theorem \ref{prop. pdM and idM}. Let $\gamma^t = \gamma^{\bar{t}}$ for all $t>n$ where $\bar{t}$ equals to $t$ modulo $n+1$.
Then for any $t\geq 0$, $\gamma^0, \gamma^1, \cdots, \gamma^t$ are arcs such that (PD1), (PD2) and (PD3) hold.
Thus the minimal projective resolution of $M(c_{\simp}^{\gamma^0})$ is as follows:
\[\cdots \longrightarrow P^i \longrightarrow \cdots \longrightarrow P^2 \longrightarrow P^1
\longrightarrow M(c_{\proj}^{\gamma^0}) \longrightarrow M(c_{\simp}^{\gamma^0}) \longrightarrow 0, \]
where each $M(c_{\proj}^{\gamma^{\overline{i}}})$ is a quotient of $P^i$.  Therefore $\pdim M(c_{\simp}^{\gamma^0})=\infty$.
\end{remark}

We recall some notions from \cite[Section 2]{AAG2008}.

\begin{definition} \label{def. per. and forb.} \rm
Let $A=kQ/I$ be a gentle algebra.
A {\it non-trivial permitted} (respectively, {\it forbidden}) {\it path} is either a path of length $1$,
or a path $\mathit{\Pi}=\alpha_1\cdots\alpha_m$ ($m\ge 2$) such that $\alpha_i\alpha_{i+1}\notin I$
(respectively, $\alpha_i\alpha_{i+1}\in I$) for any $1\le i\le m-1$.
A {\it non-trivial permitted} (respectively, {\it forbidden}) {\it thread} is a maximal non-trivial permitted
(respectively, forbidden) path, that is, for each $\beta\in Q_1$ satisfying $t(\beta)=s(\alpha_1)$
we have $\beta\alpha_1\in I$ (respectively, $\beta\alpha_1\notin I$), and for each $\beta'\in Q_1$
satisfying $t(\alpha_m)=s(\beta')$ we have $\alpha_m\beta' \in I$ (respectively, $\alpha_m\beta' \notin I$).
A {\it trivial permitted} (respectively, {\it forbidden}) {\it thread} is a trivial path $v\in Q_0$
such that one of following holds:
(i) $v$ is a source point and $\sharp \{\beta \in Q_1\mid t(\beta)=v\}=0$;
(ii) $v$ is a sink point and $\sharp \{\beta'\in Q_1\mid s(\beta')=v\}=0$;
(iii) $\sharp \{\beta\in Q_1\mid t(\beta)=v\} = 1 = \sharp \{\beta'\in Q_1\mid s(\beta')=v\}$
and $\beta\beta'\notin I$ (respectively, $\beta\beta'\in I$).
\end{definition}

We use $\mathcal{P}_A$ (respectively, $\mathcal{F}_A$) to denote the set of all permitted (respectively, forbidden) threads.
By \cite[Section 1]{OPS2018}, we have a bijection $\mathfrak{p}: \mathcal{P}_A \to \M_A$
(respectively, $\mathfrak{f}: \mathcal{F}_A \to \{\Delta_i\mid 1\le i\le d\}$), where
$\mathfrak{p}(H)$ is the marked point $p$ such that every arc with endpoint $p\in\M_A$ corresponds to
a vertex of $H \in \mathcal{P}_A$ (respectively, $\mathfrak{f}(\mathit{\Pi})$ is the elementary polygons whose sides
belonging to $\Gamma_A$ correspond to the vertices of $\mathit{\Pi}\in \mathcal{F}_A$).

For any path $\wp$, we use $l(\wp)$ to denote the length of $\wp$. In particular,
for an {\it oriented cycle with full relation} of the quiver $(Q, I)$ of a gentle algebra $A$,
that is, an oriented cycle $\alpha_1\alpha_2\cdots\alpha_\ell$ ($t(\alpha_\ell)= s(\alpha_1)$) such that
$\alpha_{\overline{i}}\alpha_{\overline{i+1}}\in I$ where $\overline{i}$ is equal to $i$ modulo $\ell$,
it deduces a forbidden threads of infinite length $\cdots \alpha_{\ell}\alpha_1\alpha_2\cdots \alpha_{\ell} \alpha_1\alpha_2\cdots$.

\begin{theorem} \label{thm. gl.dim}
Let $A=kQ/I$ be gentle algebra and $\Surf_A = (\Surf_A, \M_A, \Gamma_A)$ its marked ribbon surface, and
let $\{\Delta_i\mid 1\le i\le d\}$ be the set of all elementary polygons of $\Surf_A$. Then
$$\mathrm{gl.dim}A = \max\limits_{1\le i\le d} \mathfrak{C}(\Delta_i)-1= \max\limits_{\mathit{\Pi}\in\mathcal{F}_A } l(\mathit{\Pi}).$$
\end{theorem}

\begin{proof}
If there exists an elementary polygons $\Delta_i$ which is of the form given in Remark \ref{rk. polygon-inf},
then $\mathfrak{C}(\Delta_i) = \infty$. In this case, for each side $\gamma$ of $\Delta_i$, we have
$\pdim M(c_{\simp}^{\gamma})=\infty$ by Remark \ref{rk. polygon-inf}, and the assertion follows.

Otherwise, for each elementary polygons $\Delta_i$ whose sides are denoted by $\gamma^0, \gamma^1, \cdots, \gamma^{m_i}$
such that (\textbf{PD}) holds, then $m_i+1 = \mathfrak{C}(\Delta_i)$.
For the arc $\gamma^t$ ($0\le t\le m_i-1$), there is an elementary polygon $\Delta'_i$ such that $\gamma^t$
is the unique common side of $\Delta_i$ and $\Delta_i'$.
We can find a sequence of arcs, say $\gamma'^0(=\gamma^t), \gamma'^1, \cdots, \gamma'^{n_{it}}$
such that (\textbf{PD}) holds, and thus the minimal projective resolution of $M(c_{\simp}^{\gamma^t})$ is as follows:
\[ \cdots \longrightarrow M(c_{\proj}^{\gamma^{t+2}})\oplus M(c_{\proj}^{\gamma'^2}) \longrightarrow
M(c_{\proj}^{\gamma^{t+1}})\oplus M(c_{\proj}^{\gamma'^1}) \longrightarrow
M(c_{\proj}^{\gamma^t}) \longrightarrow M(c_{\simp}^{\gamma^t}) \longrightarrow 0. \]
It stops from the $\max\{m_i-t, n_{it}\}$-th projective module.
Notice that $n_{it} \le \mathfrak{C}(\Delta_{i}')-1$, so $\pdim M(c_{\simp}^{\gamma^t})
= \max\{m_i, n_{it}\} \le \max \{\mathfrak{C}(\Delta_{i})-1, \mathfrak{C}(\Delta_{i}')-1\}$, and hence
\[\mathrm{gl.dim}A \le \max\limits_{1\le i\le d}\{\max \{\mathfrak{C}(\Delta_{i})-1, \mathfrak{C}(\Delta_{i}')-1\}\}
= \max\limits_{1\le i\le d} \mathfrak{C}(\Delta_{i})-1 .\]

%For the proof of $\mathrm{gl.dim}A \ge \max\limits_{1\le i\le d} \mathfrak{C}(\Delta_{i})-1$,
Let $\Delta_\ell$ be the elementary polygon such that $\max\limits_{1\le i\le d} \mathfrak{C}(\Delta_{i})
= \mathfrak{C}(\Delta_{\ell}) \mathop{=\!=\!=\!=\!=\!=}\limits^{\text{denoted by}} m_\ell$,
and let $\gamma^0, \gamma^1, \cdots, \gamma^{m_\ell}: (0,1)\to \Surf_A\backslash\partial\Surf_A$ be its sides such that
$\gamma_{i}(1^-)=p_i=\gamma_{i+1}(0^+)$ and $\gamma_{i+1}$ follows $\gamma_i$ around $p_i$ in the counterclockwise order ($0\le i \le m_\ell-1$).
Then $\gamma^0, \gamma^1, \cdots, \gamma^{m_\ell}$ satisfy the condition ($\textbf{PD}$), and thus
\[\gldim A \ge \pdim M(c_{\simp}^{\gamma^0}) = m_\ell
= \mathfrak{C}(\Delta_{\ell})-1 = \max\limits_{1\le j\le d} \mathfrak{C}(\Delta_{j})-1\]
by Proposition \ref{prop. pdM and idM}. It follows that $\mathrm{gl.dim}A = \max\limits_{1\le i\le d} \mathfrak{C}(\Delta_i)-1$.

Since the map $\mathfrak{f}$ above is bijective, the length $l(\mathit{\Pi})$ of the forbidden thread
$\mathit{\Pi}$ is equal to $\mathfrak{C}(\mathfrak{f}(\mathit{\Pi}))-1$ and
$\mathrm{gl.dim}A= \max\limits_{\mathit{\Pi}\in\mathcal{F}_A } l(\mathit{\Pi})$.
\end{proof}

%%%%%%%%%%%%%%%%%%%%%%%%%%%%%%%%%%%%%%%%
%%%%%%%%%%%%%%%%%%%%%%%%%%%%%%%%%%%%%%%%

\subsection{\textbf{AG-invariants and AG-equivalence}} \label{Section AG-invariants}
%\subsection{The definition of AG-invariants}
The {\it AG-invariant} of a gentle algebra $A=kQ/I$ introduced by Avella-Alaminos and Gei\ss\ in \cite{AAG2008}
is a function $\phi_A: \mathbb{N}^2\to \mathbb{N}$ with $\mathbb{N}$ the set of natural numbers
mapping $\phi[(m_t,n_t)]$ to the number of pairs $(m_t,n_t)$
in the sequence $(m_t,n_t)_{1\le t\le r}$, where $(m_i,n_i)_{1\le i\le r}$ is obtained as follows:
\begin{itemize}
\item[{Step 1.}]
Let $H_0=\alpha^0_1\cdots\alpha^0_{u_0}$ ($u_0\in\mathbb{N}$) be a (trivial) permitted thread of $A$.
\item[{Step 2.}]
If $H_0$ is non-trivial and there exists a non-trivial forbidden thread $F = \beta^0_1\cdots\beta^0_{v_0}$
($v_0\in\mathbb{N}$) such that $t(\beta^0_{v_0})=t(\alpha^0_{u_0})$ and $\alpha^0_{u_0}\ne\beta^0_{v_0}$,
consider $\mathit{\Pi}_0 = F$;
otherwise, consider the (trivial) forbidden thread $\mathit{\Pi}_0$ satisfying $t(\mathit{\Pi}_0)=t(H_0)$.
\item[{Step 3.}]
If $\mathit{\Pi}_0$ is non-trivial and there exists a non-trivial permissible thread
$P = \alpha^1_1\cdots\alpha^1_{u_1}$ ($u_1\in\mathbb{N}$) such that $s(\alpha^1_1) =  s(\beta^0_1)$ and $\alpha^1_1\ne \beta^0_1$,
consider $H_1 = P$;
otherwise, consider the (trivial) permitted thread $H_1$ satisfying $s(H_1)=s(\mathit{\Pi}_0)$.
\item[{Step 4.}]
If we obtain a permitted thread $H_i$ (respectively, forbidden thread $\mathit{\Pi}_i$),
then just as Step 2 (respectively, Step 3), determine the forbidden thread $\mathit{\Pi}_i$ (respectively, permitted thread $H_{i+1}$).
Repeat this step until $H_{m_1} = H_0$ first appears, we induce a pair $(m_1,n_1)$ where $n_1=\sum_{i=0}^{m_1-1} l(\mathit{\Pi}_i)$.
\item[{Step 5.}]
Until all permitted threads are considered in Steps 1--4, we obtain a sequence of pairs $(m_1, n_1)$, $\cdots$, $(m_{r'}, n_{r'})$
with $1\le r' \le r$.
\item[{Step 6.}]
For every oriented cycle of length $\ell$ having full relations,
we add a pair $(0, \ell)$ in the sequence obtained in Step 5, and obtain the sequence $(m_t,n_t)_{1\le t\le r}$
where $m_t=0$ if and only if $t>r'$. % (Note that in the case of $r'=r$, $Q$ with no oriented cycle with full relations.)
\end{itemize}

\noindent The bijection $\mathfrak{f}: \mathcal{F}_A \to \{\Delta_i\mid {1\le i\le d}\}$ provides a method for calculating
AG-invariants by marked ribbon surfaces. Indeed, each elementary polygon, say $\Delta^t_j \in\{\Delta_i\mid 1\le i\le d\}$,
with a unique single boundary arc on the boundary component $b^t$ of $\Surf_A$
corresponds to a forbidden thread $\mathfrak{f}^{-1}(\Delta^t_j)$, and $\sharp\mathfrak{S}(\Delta^t_j)$ is equal to:
\begin{itemize}
\item either $l(\mathfrak{f}^{-1}(\Delta^t_j))+1$, if $\Delta^t_j$ is not $\infty$-elementary;
\item or the number of vertices of $\Delta^t_j$, otherwise.
\end{itemize}
Thus $(m_t,n_t)_{1\le t\le r}$ can be obtained by
\begin{align}
m_t & = \text{ the number of marked points on the boundary component } b^t, \nonumber\\
n_t & = \sum_{j\in J_t}(\sharp\mathfrak{S}(\Delta^t_j)-1) = \sum_{j\in J_t}\sharp\mathfrak{S}(\Delta^t_j) - m_t, \nonumber
\end{align}
where $J_t \subseteq \{1,2,\cdots,d\}$ is a set of all elementary polygons
which have a side is on the boundary component $b^t$ (in this case, we have $\bigcup_{1\le t\le r} = \{1,2,\cdots, d\}$ and
$J_t \cap J_{t'} = \varnothing$ holds for all $1\le t\ne t'\le r$).
See \cite[Theorem 3.3]{LZ2021}.

\begin{remark} \rm
The AG-invariant of a gentle algebras is an invariant up to derived equivalence; that is, if two gentle algebras $A$ and $B$
are derived equivalent, then the number of cycles of $A$ and that of $B$ are identical and $\phi_A=\phi_B$.
Furthermore, if $A$ and $B$ are gentle one-cycle, that is, the quivers of $A$ and $B$ both have at most one cycle,
then $\phi_A=\phi_B$ yields that $A$ is derived equivalent to $B$ (\cite[Theorem A, Proposition B, Theorem C]{AAG2008}).
\end{remark}

%\subsection{AG-equivalence}
We say that two gentle algebras $A$ and $A'$ are {\it AG-equivalent} if $\phi_A = \phi_{A'}$.
By \cite[Proposition B]{AAG2008}, derived equivalence yields AG-equivalence.
We know that the finiteness of the global dimension is invariant under derived equivalence.
In the following, we show that it is also invariant under AG-equivalence. % (See Corollary \ref{coro. AG-equiv.}).
We need the following observation.

\begin{proposition} \label{prop. infty. gl.dim.}
For a gentle algebra $A=kQ/I$, we have that
$\mathrm{gl.dim} A =\infty$ if and only if there is $\ell\geq 1$ such that $\phi_A(0, \ell) \ne 0$.
\end{proposition}

\begin{proof}
For a gentle algebra $A$, if there is $\ell\geq 1$ such that $\phi_A(0,\ell) = x \ge 1$,
then its marked ribbon surface $\Surf_A=(\Surf_A, \M_A, \Gamma_A)$ has $x$ unmarked boundary component(s), say $b^1, b^2, \cdots, b^x$.
So $\Surf_A$ has $x$ $\infty$-elementary polygon(s) $\Delta^1, \Delta^2, \cdots, \Delta^x$.
% and each $\Delta^t$ ($1\le t \le x$) has $\ell$ sides $\gamma^{t1}, \gamma^{t2}, \cdots, \gamma^{t\ell} \in \Gamma_A$
% surrounding the unmarked boundary component $b^t$.
Thus $\mathrm{gl.dim} A =\infty$ by Theorem \ref{thm. gl.dim}.

Conversely, if $\phi_A(0,\ell)=0$ for any $\ell\geq 1$, then each boundary component $b^t$ of $\Surf_A$
has at least one marked point and each elementary polygon $\Delta^t_{j}$ ($j\in J$ with $J$ some finite set)
whose single boundary arc included in $b^t$ is not a unmarked boundary component of $\Surf_A$.
Then $\mathfrak{C}(\Delta^t_{j}) < \infty $, and so $\mathrm{gl.dim}A =
\max\limits_{1\le t\le r,j\in J}\mathfrak{C}(\Delta^t_{j}) -1 < \infty$ by Theorem \ref{thm. gl.dim}.
\end{proof}

As a consequence, we get the following corollary, which was proved
by Opper, Plamondon and Schroll by using the Koszul dual of gentle algebras.

\begin{corollary} \label{coro. infty gldim and oriented cycles} {\rm \cite[Section 1, 1.7]{OPS2018}}
Let $A=kQ/I$ be a gentle algebra. Then $\gldim A=\infty$ if and only if its quiver $(Q,I)$ has
at least one oriented cycle with full relations.
\end{corollary}

\begin{proof}
If $\gldim A=\infty$, then its marked ribbon surface has at least one $\infty$-elementary polygon $\Delta_j$ by Theorem \ref{thm. gl.dim}.
Suppose $\mathfrak{S}(\Delta_j)=\{\gamma_0, \cdots, \gamma_{\ell-1}\}$, where $\gamma_{\overline{i}}(1^-) = \gamma_{\overline{i+1}}(0^+)$
for any $i\geq 0$ and $\overline{i}$ is equals to $i$ modulo $\ell$.
Then $(Q,I)$ has an oriented cycle with full relations of length $\ell$ whose vertices are $\{\mathfrak{v}^{-1}(\gamma_i)\mid 0\le i\le \ell-1\}$.

Conversely, if $(Q,I)$ has at least one oriented cycle of length $\ell(\geq 1)$ with full relations,
then $\phi_A (0,\ell) >0$, and thus $\gldim A = \infty$ by Proposition \ref{prop. infty. gl.dim.}.
\end{proof}

The following result shows that the finiteness of the global dimension of gentle algebras is invariant under AG-equivalence.

\begin{corollary} \label{coro. AG-equiv.}
If two gentle algebras $A$ and $A'$ are AG-equivalent,
then $\mathrm{gl.dim}A < \infty$ if and only if $\mathrm{gl.dim}A' < \infty$.
\end{corollary}

\begin{proof}
By Proposition \ref{prop. infty. gl.dim.}, we have that
$\mathrm{gl.dim}A < \infty$ if and only if $\phi_A(0,\ell) = 0$ for any $\ell\geq 1$.
Suppose that $A$ and $A'$ are AG-equivalent and $\mathrm{gl.dim}A < \infty$. Then
$\phi_A=\phi_{A'}$ and $\phi_{A'}(0,\ell)=\phi_{A}(0,\ell) = 0$
for any $\ell\geq 1$. Thus $\mathrm{gl.dim}A' < \infty$ by Proposition \ref{prop. infty. gl.dim.}.
\end{proof}

\section{\textbf{Self-injective dimension}} \label{Section Gproj}

A module $G\in\bmod A$ is called {\it Gorenstein projective} if there exists an exact sequence of projective modules
\[\cdots \longrightarrow P^{-2}
\mathop{\longrightarrow}\limits^{d^{-2}} P^{-1} \mathop{\longrightarrow}\limits^{d^{-1}} P^{0}
\mathop{\longrightarrow}\limits^{d^{0}}  P^{1}  \mathop{\longrightarrow}\limits^{d^{1}}
P^{2} \longrightarrow \cdots \]
which remains still exact after applying the functor $\Hom_A(-,A)$, such that $G\cong \mathrm{Im}d^{-1}$ \cite{AB, EJ1995}.
Obviously, every projective module is Gorenstein projective.
For each module $M$, its {\it Gorenstein projective dimension} $\Gpdim M$, is defined as
\begin{align}
\inf\{n \mid
 & \text{ there exists an exact sequence }  0\to G_n \to \cdots \to G_1 \to G_0 \to M \to 0 \nonumber  \\
 & \text{in}\ \bmod A\ \text{ with all } G_i \text{ Gorenstein projective}\}. \nonumber
\end{align}
%and the global Gorenstein dimension $\Gdim A$ of $A$ is defined as
%\[\sup\{ \Gpdim M | M\in \bmod A \}. \]

\subsection{The descriptions of Gorenstein projective modules}
For a gentle algebra $A=kQ/I$, the following theorem shows that all indecomposable
Gorenstein projective modules can be determined by its quiver $(Q, I)$.

\begin{theorem} {\rm \cite{K2015}}  \label{thm. Gproj}
Let $A=kQ/I$ be a gentle algebra.
An $A$-module $G$ is Gorenstein projective if and only if $G$ is isomorphic to a projective module
or an $\alpha A$ where $\alpha\in Q_1$ is an arrow on some oriented cycle with full relations of $A$
$($note that $\alpha A$ is a direct summand of $\rad P(s(\alpha))\cong \rad (e_{s(\alpha)}A)$$)$.
\end{theorem}

In this subsection, we depict all indecomposable Gorenstein projective modules in marked ribbon surfaces.

\begin{lemma}\label{lemm. Gporj}
Let $A=kQ/I$ be a gentle algebra having an oriented cycle with full relation, and let
$\Surf_A=(\Surf_A,\M_A,\Gamma_A)$ be its marked ribbon surface % (which has at least one $\infty$-elementary polygon by),
and $\gamma\in \Gamma_A$ an arc.
\begin{enumerate}
\item[{\rm(1)}] If $\gamma$ is a side of an $\infty$-elementary polygon $\Delta_i$ of $\Surf_A$, then there is an arc $\gamma'$ such that
\begin{itemize}
\item $\gamma'$ and $\gamma$ are two adjacent sides of $\Delta_i$, the common endpoint of them is denoted by $p$; and
\item $\gamma'$ and $\gamma$ surround $p$ in the counterclockwise order.
\end{itemize}
\item[{\rm(2)}]  Moreover, suppose that $c_{\proj}^{\gamma'} = c_1\cdots c_r$ is the projective permissible curve of $\gamma'$,
where $c_j$ $(1\le j\le r)$ is the segment from $\gamma'$ to $\gamma$,
then the module $M(c_j^{\mathfrak{T}}c_{j+1}\cdots c_r)$ is a non-projective Gorenstein projective module
{\rm(}see \Pic \ref{fig:lemm-6.2}{\rm)}.
\end{enumerate}
\end{lemma}

\begin{figure}[htbp]
\begin{center}
\begin{tikzpicture}
\draw (0,2) node{$\bullet$} -- (-1.732,-1) node{$\bullet$} -- (1.732,-1) node{$\bullet$};
\draw (0,2) -- (1.732,-1) [dash pattern=on 2pt off 2pt];
\fill [black!25] (0.25,0.25) circle (0.25cm);
\draw (0.25,0.25) circle (0.25cm) [line width=1.2pt];
\draw (0,1) node{$\Delta_i$};
\draw (0,-1.25) node{$\gamma'$};
\draw[rotate around={-120:(0,0)}] (0,-1.25) node{$\gamma$};
\draw[shift={(1.732,-1)}] (0,0) -- (-1.73,-1.00);
\draw[shift={(1.732,-1)}] (0,0) -- (-1.00,-1.73);
\draw[shift={(1.732,-1)}] (0,0) -- (-0.00,-2.00);
\draw[shift={(-1.732,-1)}] (0,0) -- ( 0.00, 2.00);
\draw[shift={(-1.732,-1)}] (0,0) -- (-1.00, 1.73);
\draw[shift={(-1.732,-1)}] (0,0) -- (-1.73, 1.00);
\draw[red] [line width = 1.2pt] (-2.5,-1) to[out=45,in=180]
  (-0.75,-0.25) node[below]{$c_{\proj}^{\gamma'}$} to[out=0,in=200] (2.5,-1.75);
\draw[blue] [line width = 1.2pt] (-2.5,-1) to[out=45,in=180]
  (-0.6,0) node[above]{$c'$} to[out=0,in=135] (1.73,-1);
\draw (-2.5,-1) node{$\pmb{\times}$};
\draw (2.5,-1.75) node{$\pmb{\times}$};
\end{tikzpicture}
\end{center}
\caption{$c' \simeq c_j^{\mathfrak{T}}c_{j+1}\cdots c_r$ is the permissible curve corresponded by non-projective Gorenstein-projective module $\alpha A$, where $\alpha:\gamma' \to \gamma$ is the arrow on the oriented cycle with full relations corresponded by $\Delta_i$.}
\label{fig:lemm-6.2}
\end{figure}
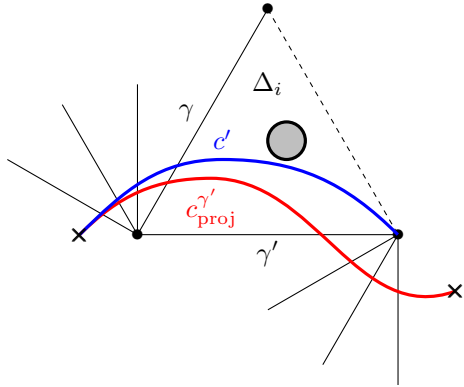

\begin{proof}
(1) Since the quiver $(Q,I)$ of $A$ has an oriented cycle with full relation,
we have $\gldim A =\infty$ by Corollary \ref{coro. infty gldim and oriented cycles},
and moreover each oriented cycle of length $\ell$ with full relation corresponds to a pair $(0, \ell)$ of the AG-invariant $\phi_A$
and an $\infty$-elementary $\ell$-polygon of marked ribbon surface $\Surf_A$.
Let $q$ be another endpoint of $\gamma$. Obviously, in the case for $p=q$, we have $\gamma=\gamma'$.
If $p\ne q$, then $\Delta_i$ corresponds to some oriented cycle with full relation
$0 \mathop{\to}\limits^{\alpha_1} 1  \mathop{\to}\limits^{\alpha_2} 2 \to \cdots \to \ell-1 \mathop{\to}\limits^{\alpha_{\ell}} 0$
such that the side $\gamma=\mathfrak{v}^{-1}(x)$ of $\Delta_i$ corresponds to some vertex
$x \in \{0,1,\cdots,\ell-1\}$ through the map $\mathfrak{v}: Q_0\to \Gamma_A$ which is defined in Definition \ref{def. alg. of m. r. s.}.
Then $\gamma' = \mathfrak{v}(\overline{x-1})$, where $\overline{x-1}$ is equal to $x-1$ modulo $\ell$.

(2) Following Definition \ref{def. alg. of m. r. s.}(1), for simplicity, we use $\gamma$ to denote the vertex $\mathfrak{v}^{-1}(\gamma) \in Q_0$.
We have that $c_{\proj}^{\gamma'}$ corresponds to the string
\[\cdots \mathop{\longleftarrow}\limits^{\alpha(c_{j-1})} \gamma' \mathop{\longrightarrow}\limits^{\alpha(c_j)}
\gamma \mathop{\longrightarrow}\limits^{\alpha(c_{j+1})} \gamma'' \mathop{\longrightarrow}\limits^{\alpha(c_{j+2})}  \cdots,\]
where each arrow $\alpha(c_j) \in Q_1$ is induced by the segment $c_j$ for any $1\le j\le r-1$
(see Definition \ref{def. alg. of m. r. s.}(2)).
Then $c_{j}^{\mathfrak{T}}c_{j+1}\cdots c_r$, denoted by $c$, corresponds to the string
\[\gamma \mathop{\longrightarrow}\limits^{\alpha(c_{j+1})} \gamma'' \mathop{\longrightarrow}\limits^{\alpha(c_{j+2})}  \cdots.\]
Thus $M(c)\cong \alpha(c_j) A \le_{\oplus} \mathrm{rad}M(c_{\proj}^{\gamma'})$ is Gorenstein projective by Theorem \ref{thm. Gproj}.

On the other hand, the side $\mathfrak{v}(\overline{x+1})$ of $\Delta_i$ satisfies that
\begin{itemize}
\item $\gamma$ and $\mathfrak{v}(\overline{x+1})$ are two adjacent sides of $\Delta_i$
(the common endpoint of them is denoted by $q$); and
\item $\gamma$ and $\mathfrak{v}(\overline{x+1})$ surround $q$ in the counterclockwise order.
\end{itemize}
Then $c \not\simeq c_{\proj}^{\gamma}$ and thus $M(c)$ is not projective by Lemma \ref{Lemm. proj}.
\end{proof}

\begin{proposition}{\rm (Gorenstein projective modules)} \label{prop. GProj}
Let $A=kQ/I$ be a gentle algebra and $\Surf_A=(\Surf_A,\M_A,\Gamma_A)$ its marked ribbon surface.
For a permissible curve $c$,
the indecomposable module $M(c)$ is Gorenstein projective if and only if one of the following holds:
\begin{itemize}
\item {\rm (GP1)} $c$ is a projective permissible curve;
\item {\rm (GP2)} $c \simeq c_1 \cdots c_r$ consecutively crosses arcs $\gamma_1, \cdots, \gamma_{r-1}$
which have a common endpoint
$$p = p(\gamma_1, \gamma_2, c) = \cdots = p(\gamma_{r-2}, \gamma_{r-1}, c)$$ such that:
  \begin{itemize}
    \item $\gamma_{i+1}$ appears after $\gamma_{i}$ $(1\le i\le r-2)$ in the counterclockwise order surrounding $p$;
    \item $\gamma_1$ is a side of some $\infty$-elementary polygon $\Delta_j$ with $1\le j\le d$;
    \item the end segment of $c$ between $c(0^+)$ and $\gamma_1$ lies in the inner of $\Delta_j$.
  \end{itemize}
\end{itemize}
\end{proposition}

\begin{proof} We first prove the sufficiency.
By Lemma \ref{Lemm. proj}, if $c$ is a projective permissible curve, then it is trivial that $M(c)$ is a Gorenstein projective module.

If $c$ satisfies (GP2), then there is a unique arc $\gamma'$ such that $\gamma'$ and $\gamma := \gamma_1$ satisfy
the conditions given in Lemma \ref{lemm. Gporj} (1).
Suppose that $c_{\proj}^{\gamma'} = c_1\cdots c_s$ consecutively crosses $\tgamma_1\cdots \tgamma_{s-1}$ ($s\ge r+1$).
Then there is a unique $h$ with $1\leq h \leq s-1$ such that
$\tgamma_{h-1}=\gamma'$, $\tgamma_h =\gamma_1$, $\tgamma_{h+1} =\gamma_2$, $\cdots$, $\tgamma_{s-1}=\gamma_{r-1}$.
To be more precise, $c_{\proj}^{\gamma'}$ corresponds to the following string:
\begin{align}\label{}
& \tgamma_1 \longleftarrow \cdots \longleftarrow \tgamma_{h-2} \longleftarrow \tgamma_{h-1}
\mathop{\longrightarrow}\limits^{\alpha} \tgamma_h \longrightarrow \tgamma_{h+1} \longrightarrow \cdots \longrightarrow \tgamma_{s-1} \nonumber \\
& \text{($\tgamma_{h-1}=\gamma'$ and $\tgamma_{h}=\gamma$ are adjacent sides of $\Delta_j$)}. \nonumber
\end{align}
We have that $c$ consecutively crosses $\tgamma_h$, $\tgamma_{h+1}$, $\cdots$, $\tgamma_{s-1}$,
where $\tgamma_{h}\in\mathfrak{S}(\Delta_j)$ and $\tgamma_{m} \notin \mathfrak{S}(\Delta_j)$ for any $h+1\le m\le s-1$.
Thus $M(c) \cong \alpha A $ is a Gorenstein projective module satisfying $M(c) \le_{\oplus} \mathrm{rad} M(c_{\proj}^{\tgamma_{h-1}}) $.
In this case, $M(c)$ is non-projective.

Next we prove the necessity.
By Theorem \ref{thm. Gproj}, if $M(c)$ is Gorenstein projective, then $M(c)$ is isomorphic to either
a projective module or an $\alpha A$ where $\alpha$ is an arrow on some oriented cycle with full relations.

If $M(c)$ is projective, then $c$ is a projective permissible curve.
Now suppose that $M(c)$ is isomorphic to $\alpha A$. Let $v:=s(\alpha)$,
It follows from Lemma \ref{Lemm. proj} that $\alpha A \cong \mathrm{rad} P(v) \cong \mathrm{rad} M(c_{\proj}^{\tgamma(v)})$
for some arc $\tgamma(v) \in \Gamma_A$,
where $c_{\proj}^{\tgamma(v)}=c_1\cdots c_s$ is a projective permissible curve consecutively crossing arcs $\tgamma_1, \cdots, \tgamma_{s-1}$.
Then $v$ corresponds to some $\tgamma_h$ with $2\leq h\leq s-2$ through $\mathfrak{v}: Q_0\to \Gamma_A$, that is, $\tgamma_h = \tgamma(v)$
(note that $h\ne 1$ and $h\ne s-1$ because $v$ is a vertex on the oriented cycle with full relations).
%For simplicity, we denote by $\tgamma$ the vertex $v=\mathfrak{v}^{-1}(\tgamma)\in Q_0$ for each $\tgamma \in \Gamma_A$.
Then $c_{\proj}^{\tgamma(v)} = c_{\proj}^{\tgamma_h}$ corresponds to the string
\[ \tgamma_1 \longleftarrow \cdots  \longleftarrow \tgamma_{h-1} \mathop{\longleftarrow}\limits^{\alpha_1} \tgamma_h
\mathop{\longrightarrow}\limits^{\alpha_2} \tgamma_{h+1} \longrightarrow \cdots \longrightarrow \tgamma_{s-1}, \]
and the arrow $\alpha$ is either $\alpha_1$ or $\alpha_2$.
Without loss of generality, suppose $\alpha=\alpha_1$. Since $M(c)\cong \alpha_1 A$ is Gorenstein projective,
the arcs $\tgamma_{h-1}$ and $\tgamma_h$ are adjacent sides of some $\infty$-elementary polygon $\Delta_j$ of $\Surf_A$.
Let $\tgamma_{h-1} = \gamma_1$,  $\tgamma_{h-2} = \gamma_2$, $\cdots$, $\tgamma_1 = \gamma_{h-1}$.
Then the permissible curve $c$ consecutively crossing $\gamma_1, \cdots, \gamma_{h-1}$ satisfies the condition (GP2).
\end{proof}

\begin{example} \label{Ex. GProj.} \rm
Let $A$ be a gentle algebra whose marked ribbon surface $\Surf_A$ is given in \Pic \ref{Ex. GProj. in Surf.}.
We can find 12 indecomposable Gorenstein projective modules including
\begin{itemize}
\item 9 indecomposable projective modules (see green permissible curves):
\[ {{}_{^4_5}} {{}^1} {_{^2_{^6_7}}},\ {{}_{^6_7}} {{}^2} {_{^3_{^8_9}}},\ {{}_{^8_9}} {{}^3} {_{^1_{^4_5}}},\
{^4_5},\ {^6_7},\ {^8_9},\ {5},\ {7},\ {9}; \]
\item and 3 indecomposable non-projective Gorenstein projective modules (see pink permissible curves):
\[ {^{_2}_{^6_7}} \big(\le_{\oplus} \mathrm{rad} \big( {{}_{^4_5}} {{}^1} {_{^2_{^6_7}}} \big)\big),
\  {^{_3}_{^8_9}} \big(\le_{\oplus} \mathrm{rad} \big( {{}_{^6_7}} {{}^2} {_{^3_{^8_9}}} \big)\big),
\  {^{_1}_{^4_5}} \big(\le_{\oplus} \mathrm{rad} \big( {{}_{^8_9}} {{}^3} {_{^1_{^4_5}}} \big)\big). \]
\end{itemize}
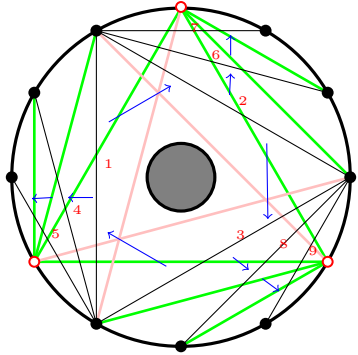
\begin{figure}[htbp]
\begin{center}
\begin{tikzpicture}[scale=1.5]
% proj. perm. curves
\draw [line width=1pt,green] (-1.3,-0.75)-- (0,1.5);
\draw [line width=1pt,green] (-1.3,-0.75)-- (1.3,-0.75);
\draw [line width=1pt,green] (1.3,-0.75)-- (0,1.5);
\draw [line width=1pt,green] (-1.3,-0.75)-- (-0.75,1.3);
\draw [line width=1pt,green] (-1.3,-0.75)-- (-1.3,0.75);
\draw [line width=1pt,green] (1.3,-0.75)-- (-0.75,-1.3);
\draw [line width=1pt,green] (1.3,-0.75)-- (0,-1.5);
\draw [line width=1pt,green] (0,1.5)-- (1.5,0);
\draw [line width=1pt,green] (0,1.5)-- (1.3,0.75);
% Gproj. perm. curves
\draw [line width=1pt,pink] (1.5,0)-- (-1.3,-0.75);
\draw [line width=1pt,pink] (-0.75,-1.3)-- (0,1.5);
\draw [line width=1pt,pink] (-0.75,1.3)-- (1.3,-0.75);
%%%%%%%%%%%%%%%%%%%%%%%%%%%%%%%%%%%%%%%%%%%%%%%
%%%%%%%%%%%%%%%%%%%%%%%%%%%%%%%%%%%%%%%%%%%%%%%
%%%%%%%%%%%%%%%%%%%%%%%%%%%%%%%%%%%%%%%%%%%%%%%
\filldraw[gray] (0,0) circle (0.3);
\draw [line width=1.2pt] (0,0) circle (1.5cm);
\draw [line width=1.2pt] (0,0) circle (0.3cm);
% arcs
\draw (-0.75, 1.30)-- (-0.75,-1.30);
\draw (-0.75,-1.30)-- ( 1.50, 0.00);
\draw ( 1.50, 0.00)-- (-0.75, 1.30);
\draw (-0.75, 1.30)-- ( 0.75, 1.30);
\draw ( 1.50, 0.00)-- ( 0.75,-1.30);
\draw (-0.75,-1.30)-- (-1.50, 0.00);
\draw ( 1.50, 0.00)-- ( 0.00,-1.50);
\draw (-0.75,-1.30)-- (-1.30, 0.75);
\draw (-0.75, 1.30)-- ( 1.30, 0.75);
\draw [->, blue] ( 0.76, 0.30) -- ( 0.77,-0.37);
\draw [->, blue] (-0.64, 0.49) -- (-0.09, 0.81);
\draw [->, blue] (-0.13,-0.79) -- (-0.65,-0.49);
\draw [->, blue] ( 0.43, 0.73) -- ( 0.44, 0.92);
\draw [->, blue] ( 0.44, 1.08) -- ( 0.44, 1.26);
\draw [->, blue] ( 0.46,-0.71) -- ( 0.60,-0.82);
\draw [->, blue] ( 0.72,-0.90) -- ( 0.87,-1.01);
\draw [->, blue] (-0.78,-0.18) -- (-1.00,-0.18);
\draw [->, blue] (-1.14,-0.18) -- (-1.32,-0.19);
\draw [red] (-0.64, 0.12) node {\tiny$1$};
\draw [red] ( 0.55, 0.68) node {\tiny$2$};
\draw [red] ( 0.53,-0.51) node {\tiny$3$};
\draw [red] (-0.92,-0.29) node {\tiny$4$};
\draw [red] (-1.11,-0.50) node {\tiny$5$};
\draw [red] ( 0.31, 1.09) node {\tiny$6$};
\draw [red] ( 0.12, 1.33) node {\tiny$7$};
\draw [red] ( 0.91,-0.59) node {\tiny$8$};
\draw [red] ( 1.17,-0.65) node {\tiny$9$};
% points
\fill ( 1.50, 0.00) circle (1.5pt);
\fill ( 0.75, 1.30) circle (1.5pt);
\fill (-0.75, 1.30) circle (1.5pt);
\fill (-1.50, 0.00) circle (1.5pt);
\fill (-0.75,-1.30) circle (1.5pt);
\fill ( 0.75,-1.30) circle (1.5pt);
\fill ( 0.00,-1.50) circle (1.5pt);
\fill (-1.30, 0.75) circle (1.5pt);
\fill ( 1.30, 0.75) circle (1.5pt);
\fill [red] ( 0.00, 1.51) circle (1.5pt); \filldraw[white] ( 0.00, 1.51) circle (0.03);
\fill [red] ( 1.30,-0.75) circle (1.5pt); \filldraw[white] ( 1.30,-0.75) circle (0.03);
\fill [red] (-1.30,-0.75) circle (1.5pt); \filldraw[white] (-1.30,-0.75) circle (0.03);
\end{tikzpicture}
\caption{An example for calculating all indecomposable Gorenstein projective modules}
\label{Ex. GProj. in Surf.}
\end{center}
\end{figure}
\end{example}

\subsection{The descriptions of self-injective dimension}

Gei\ss\ and Reiten showed that gentle algebras are Gorenstein \cite[Theorem 3.4]{GR2005}.
In this subsection we provide a geometric interpretation of this result by marked ribbon surfaces.
We use $\finiteDelta$ to denote the set of all non-$\infty$-elementary polygons,
and use $\fforbid_A$ to denote the set of all forbidden threads of finite length. We first prove the following proposition.

\begin{proposition}\label{prop-6.5}
Let $A=kQ/I$ be a non-simple gentle algebra and $\Surf_A=(\Surf_A,\M_A,\Gamma_A)$ its marked ribbon surface. Then
$$ \idim_AA = \idim A_A =
\begin{cases}
0,\ \mbox{\text{if $Q$ is an oriented cycle with full relations};}\\
\max\limits_{\Delta_i\in\finiteDelta}\big\{1, {\mathfrak{C}(\Delta_i)}-1 \big\},\ \mbox{\text{otherwise}.}
\end{cases}$$
% $$\color{red} \max\big\{1, \max\limits_{\Delta_i\in\finiteDelta}{\mathfrak{C}(\Delta_i)}-1 \big\}.$$
\end{proposition}

\begin{proof}
If $Q$ is an oriented cycle with full relations, then $A$ is a self-injective Nakayama algebra.

Now suppose that $Q$ is not an oriented cycle with full relations.
Since $A$ is a gentle algebra, it follows from \cite[Theorem 3.4]{GR2005} that $A$ is Gorenstein.
Then by \cite[Theorem]{Ho} and \cite[Theorem 1.4]{HuH}, we have $\idim_AA=\idim A_A = \sup\{\Gpdim S\mid S$ is simple$\}$.

The proof is divided to the following two cases:
(i) $(Q, I)$ has at least one oriented cycle with full relations;
(ii) $(Q, I)$ has no oriented cycle with full relations.

(i) Note that if $v\in Q_0$ is a vertex on the oriented cycle with full relations, then $\pdim S(v)=\infty$.
Consider the arrow $\alpha$  on this oriented cycle with full relations such that its sink is $v$.
We have that $\alpha A$ is an indecomposable non-projective Gorenstein projective module,
and moreover it is the Gorenstein projective cover of $S(v)$.
In the marked ribbon surface $\Surf_A$, $v$ corresponds to a side, denoted by $\gamma_1$, of an $\infty$-elementary polygon $\Delta_j$
and $M^{-1}(\alpha A)$ is such a permissible curve $c$ consecutively crossing $\gamma_1,  \cdots, \gamma_r$ ($r\ge 1$),
where $\gamma_1, \cdots, \gamma_r$ satisfy the condition (GP2) given in Proposition \ref{prop. GProj} and $\gamma_1=\mathfrak{v}(v)$.

In the case for $r=1$, it is obvious that $S(v)$ is a Gorenstein projective module and $\Gpdim S(v)=0$.
Suppose $r>1$. Then there exists an elementary polygon $\Delta_{j'}$ such that $\gamma_1$ and $\gamma_2$ are two sides of $\Delta_{j'}$.
If $\Delta_{j'}$ is not $\infty$-elementary then we get $\Gpdim S(v)=\mathfrak{C}(\Delta_{j'})-1$
by using an argument similar to that of Proposition \ref{prop. pdM and idM}
(the arcs $\gamma_1$ and $\gamma_2$ can be viewed as $\gamma^0$ and $\gamma^1$ shown in Proposition \ref{prop. pdM and idM}, respectively).
If $\Delta_{j'}$ is $\infty$-elementary, then consider the permissible curve $c'$ consecutively crossing $\gamma_2, \cdots, \gamma_r$
(in this case $\gamma_2\in \mathfrak{S}(\Delta_{j'})$) where $\gamma_2, \cdots, \gamma_r$ satisfy (GP2).
We have that $M(c')$ is the kernel of $M(c) \twoheadrightarrow S(v) \cong M(c_{\simp}^{\gamma_1})$ by Theorem \ref{thm. s. e. s.}(1)
because the positional relationship of $c'$, $c$ and $c_{\simp}^{\gamma_1}$ satisfies the conditions
given in Proposition \ref{prop. s. e. s.}. Then $\Gpdim S(v)=1$.
Consequently we have
$$\idim_AA=\idim A_A=\max\big\{1, \max\limits_{\Delta_i\in\finiteDelta}{\mathfrak{C}(\Delta_i)}-1 \big\}.$$

(ii) If $(Q, I)$ has no oriented cycle with relation, then $A$ has finite global dimension by Theorem \ref{thm. gl.dim}. In this case we have
\[\idim_AA=\idim A_A=\gldim A = \max\limits_{1\le i\le d} \mathfrak{C}(\Delta_i)-1,\]
where $\{\Delta_i\mid 1\le i\le d\} = \finiteDelta$ is the set of all elementary polygons of $\Surf_A$.
\end{proof}

The proof of Proposition \ref{prop-6.5} depends on Theorem \ref{thm. Gproj}.
In the following, we will give its another proof which does not need to use the properties of
Gorenstein projective modules.
We assume that the quiver of $A$ is neither an oriented cycle with full relations nor a point.

\begin{lemma}\label{lemm. proj. covers of ind. inj. mods.}
Let $\Surf_A=(\Surf_A,\M_A,\Gamma_A)$ be a marked ribbon surface of a gentle algebra $A$,
and let $\gamma$ be an arc on $\Surf_A$ with endpoints $p$ and $q$. As shown in \Pic \ref{fig:Lemm-6.6},
suppose that $c_{\inj}^{\gamma}$ consecutively crosses arcs
$\gamma^0_1, \cdots, \gamma^0_{i-1}, \gamma$, $\tgamma^0_{j-1}, \cdots, \tgamma^0_1$,
and $c_{\proj}^{\gamma}$ consecutively crosses arcs
$\tgamma^0_n$, $\cdots$, $\tgamma^0_{j+1}$, $\gamma$, $\gamma^0_{i+1}$, $\cdots$, $\gamma^0_m$,
where $p$ is the common endpoint of $\gamma^0_1, \cdots, \gamma^0_m$ and $q$ is the common endpoint of $\tgamma^0_1, \cdots, \tgamma^0_n$
{\rm(}$1\le i\le m$, $1\le j\le n$ and $\gamma=\gamma^0_i=\tgamma^0_j${\rm)}. Then:
\begin{enumerate}
\item[{\rm(1)}] the projective cover of $M(c_{\inj}^{\gamma})$ is $p(M(c_{\inj}^{\gamma})):
M(c_{\proj}^{\gamma^0_1}\oplus c_{\proj}^{\tgamma^0_1}) \twoheadrightarrow M(c_{\inj}^{\gamma})$;
\item[{\rm(2)}]
$\Ker p(M(c_{\inj}^{\gamma})) \cong M(c'\oplus c_{\proj}^{\gamma} \oplus \tilde{c}')$,
where $c'$ and $c''$ are two permissible curves shown in \Pic \ref{fig:Lemm-6.6}.
%suppose that
%$c_{\proj}^{\gamma^0_1}$ consecutively crosses arcs $\gamma^1_r, \cdots, \gamma^1_1, \gamma^0_1, \gamma^0_2, \cdots, \gamma^0_m$
%and $c_{\proj}^{\tilde{\gamma}^0_1}$ consecutively crosses arcs
%$\tgamma^0_n$, $\cdots$, $\tgamma^0_2$, $\tgamma^0_1$, $\tgamma^1_1$, $\cdots$, $\tgamma^1_s$.
%Let $c'$ and $\tilde{c}'$ be two permissible curves consecutively crossing $\gamma^1_r, \cdots, \gamma^1_1$ and
%$\tgamma^1_s, \cdots, \tgamma^1_1$, respectively.
%Then $\Ker p(M(c_{\inj}^{\gamma})) \cong M(c'\oplus c_{\proj}^{\gamma} \oplus \tilde{c}')$.
\end{enumerate}
\end{lemma}

\begin{figure}[htbp]
\begin{center}
\begin{tikzpicture}[scale=1.2]
\draw (0,3)-- (0,-3);
\draw ( 0, 3) arc (   0:-120 :3);
\draw ( 0, 3) arc (   0: -90 :4);
\draw (-5, 3) arc ( 180: 360 :2.5);
\draw (-5, 3) arc ( 180: 360 :1.5);
\draw (-5, 3) arc ( 180: 360 :0.75);
\draw ( 0,-3) arc ( 180:  60 :3);
\draw ( 0,-3) arc ( 180:  90 :4);
\draw ( 0, 3) arc (-180:   0 :2);
\draw ( 0, 3) arc (-180:   0 :1.5);
\draw ( 0, 3) arc (-180:   0 :1);
\draw ( 0,-3) arc (   0: 180 :2);
\draw ( 0,-3) arc (   0: 180 :1.5);
\draw ( 0,-3) arc (   0: 180 :1);
\draw [dotted] (-4.99,3)-- (4,3);
\draw ( 0,-3) arc ( 180:   0 :2.5);
\draw [dotted] (-4,-3)-- (5,-3);
\draw ( 5,-3) arc (   0: 180 :1.5);
\draw ( 5,-3) arc (   0: 180 :0.75);
\draw (-2,3) arc (180:360 :0.5);
\draw (-2,3) arc (180:360 :0.3);
% inj. per. curve
\draw[orange, line width = 0.9pt] (-2,3) to[out=-90, in=180] (0,0) to[out=0, in=90] (2,-3);
\draw[orange] ( 0.80, 0.00) node{$c_{\inj}^{\gamma}$};
% proj. cover of inj. per. curve
\draw[red, line width = 0.9pt] (-1,-3) to[out=90, in=-90] (1,3);
\draw[red] ( 0.40, 0.70) node{$c_{\proj}^{\gamma}$};
\draw[red, line width = 0.9pt] (-4.3,3) to[out=-60, in=-120] (1,3);
\draw[red] (-1.30, 1.50) node{$c_{\proj}^{\gamma^1_1}$};
\draw[red, line width = 0.9pt] (4.3,-3) to[out=120, in=60] (-1,-2.9);
\draw[red] ( 1.30,-1.50) node{$c_{\proj}^{\tgamma^1_1}$};
% cokernel
\draw[blue, line width = 0.9pt] (-4.25,3) to[out=-55, in=-135] (0,3); \draw[blue] (-1.30, 2) node{$c'$};
\draw[blue, line width = 0.9pt] (4.25,-3) to[out=135, in= 55] (0,-3); \draw[blue] ( 1.30,-2) node{$\tilde{c}'$};
% marked points
\fill [color=black] (4,3) circle (1.5pt);
\fill [color=black] (-4,-3) circle (1.5pt);
\fill [color=black] (0,3) circle (1.5pt);
\fill [color=black] (0,-3) circle (1.5pt);
\fill [color=black] (-4.99,2.98) circle (1.5pt);
\fill [color=black] (3,3) circle (1.5pt);
\fill [color=black] (-3,-3) circle (1.5pt);
\fill [color=black] (2,3) circle (1.5pt);
\fill [color=black] (-2,-3) circle (1.5pt);
\fill [color=black] (5,-3) circle (1.5pt);
\fill [color=black] (-3.5,2.99) circle (1.5pt);
\fill [color=black] (-2,2.99) circle (1.5pt);
\fill [color=black] (3.5,-3) circle (1.5pt);
\fill [color=black] (2,-3) circle (1.5pt);
\draw [color=black] (-4.3,3) node{$\pmb{\times}$};
\draw [color=black] ( 4.3,-3) node{$\pmb{\times}$};
\draw [color=black] ( 1,3) node{$\pmb{\times}$};
\draw [color=black] (-1,-3) node{$\pmb{\times}$};
\fill [color=black] (-1,3) circle (1.5pt);
\fill [color=black] (-1.4,3) circle (1.5pt);
% p and q
\draw (0, 3) node[above]{$p$};
\draw (0,-3) node[below]{$q$};
% intersections
\draw (-2.00, 1.50) node{$x$};
\draw ( 2.00,-1.50) node{$y$};
% arcs
\draw ( 0.00, 0.35) node{$\gamma$};
\draw (-2.20, 0.50) node{$\gamma^0_1$};
\draw (-2.20, 0.00) node{$\gamma^0_2$};
\draw (-2.40,-0.25) node{$\vdots$};
\draw (-2.00,-0.50) node{$\gamma^0_{i-1}$};
\draw (-2.20,-1.00) node{$\tgamma^0_{j+1}$};
\draw (-1.50,-1.50) node{$\tgamma^0_{j+2}$};
\draw (-2.00,-2.00) node{$\vdots$};
\draw (-1.00,-2.00) node{$\tgamma^0_n$};
\draw ( 1.00, 2.00) node{$\gamma^0_m$};
\draw ( 2.00, 2.00) node{$\vdots$};
\draw ( 1.50, 1.50) node{$\gamma^0_{i+2}$};
\draw ( 2.20, 1.00) node{$\gamma^0_{i+1}$};
\draw ( 2.00, 0.50) node{$\tgamma^0_{j-1}$};
\draw ( 2.40, 0.25) node{$\vdots$};
\draw ( 2.20, 0.00) node{$\tgamma^0_2$};
\draw ( 2.20,-0.50) node{$\tgamma^0_1$};
\draw (-4.25, 2.25) node{$\gamma^1_r$};
\draw (-3.50, 2.20) node{$\vdots$};
\draw (-3.50, 1.50) node{$\gamma^1_1$};
\draw ( 4.25,-2.25) node{$\tgamma^1_s$};
\draw ( 3.50,-2.20) node{$\vdots$};
\draw ( 3.50,-1.50) node{$\tgamma^1_1$};
\draw (-1.36, 2.48) node{$\gamma^2_1$};
\draw (-1.36, 2.65) node{$\vdots$};
\end{tikzpicture}
\caption{The formal direct sum $c_{\proj}^{\gamma_1^1}\oplus c_{\proj}^{\tgamma_1^1}$
of permissible curves provides the projective cover of
the injective module $M(c_{\inj}^{\gamma})$. }
\label{fig:Lemm-6.6}
\end{center}
\end{figure}
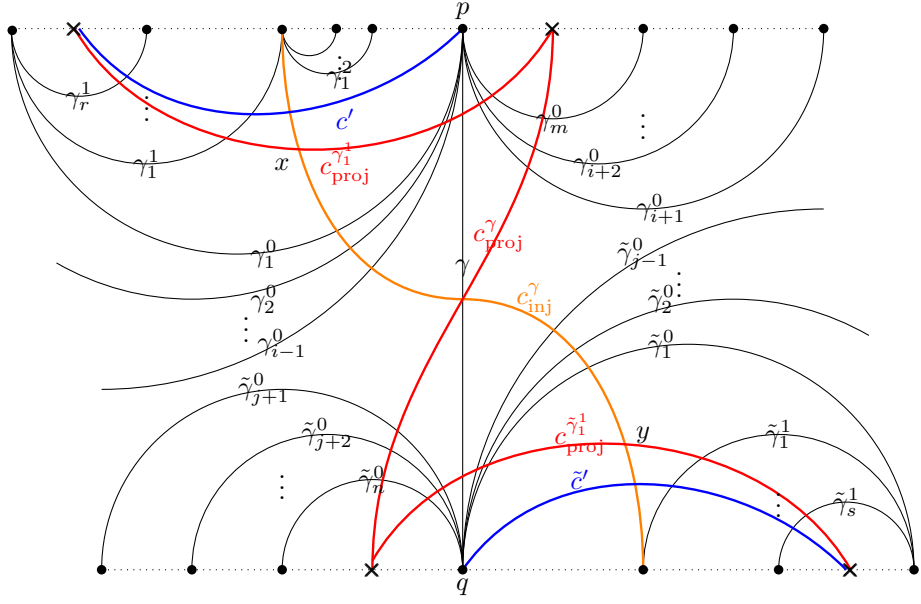

\begin{proof}
(1) This follows from Proposition \ref{prop. tops of mods.} and Theorem \ref{thm. proj. cover}.

(2) %The proof of (2) is similar to that of Proposition \ref{prop. quotientmod.}.
Indeed, $\mathrm{top}M(c_{\inj}^{\gamma}) = M(c_{\simp}^{\gamma^0_1}) \oplus M(c_{\simp}^{\tgamma^0_1})$,
and there exists no arc $\gamma^0$ (respectively, $\tgamma^0$) with endpoint $p$ (respectively, $q$)
such that $\gamma^0, \gamma^0_1, \cdots, \gamma^0_{i-1}, \gamma$
(respectively, $\gamma$, $\tgamma^0_{j-1}$, $\cdots$, $\tgamma^0_1$, $\tgamma$) surround $p$ (respectively, $q$) in the counterclockwise order.
Thus $M(c_{\inj}^{\gamma})$ corresponds to the string
\[\gamma^0_1 \longrightarrow \gamma^0_2 \longrightarrow \cdots \longrightarrow \gamma^0_{i-1} \longrightarrow \gamma
\longleftarrow \tgamma^0_{j-1} \longleftarrow \cdots \longleftarrow \tgamma^0_2 \longleftarrow \tgamma^0_1, \]
and $M(c_{\simp}^{\gamma^0_1})$ and $M(c_{\simp}^{\tgamma^0_1})$ correspond to the strings
\begin{align}
& \gamma^1_r \longleftarrow \cdots \longleftarrow \gamma^1_1 \longleftarrow
\gamma^0_1 \longrightarrow \gamma^0_2 \longrightarrow \cdots \longrightarrow \gamma^0_{i-1} \longrightarrow \gamma
\longrightarrow \gamma^0_{i+1} \longrightarrow \cdots \longrightarrow \gamma^0_m\ \text{ and } \nonumber  \\
& \tgamma^0_n \longleftarrow \cdots \longleftarrow \tgamma^0_{j+1}
\longleftarrow \gamma \longleftarrow \tgamma^0_{j-1} \longleftarrow \cdots \longleftarrow \tgamma^0_2 \longleftarrow \tgamma^0_1
\longrightarrow \tgamma^1_1 \longrightarrow \cdots \longrightarrow \tgamma^1_s
\text{, respectively.} \nonumber
\end{align}
Then the kernel of $p(M(c_{\inj}^{\gamma}))$ is isomorphic to $M_1\oplus M_2\oplus M_3$,
where $M_1$, $M_2$ and $M_3$ correspond to strings
\begin{align}
\gamma^1_r \longleftarrow \cdots \longleftarrow \gamma^1_1,\
\tgamma^0_n \longleftarrow \cdots \longleftarrow \tgamma^0_{j+1}  \longleftarrow \gamma
\longrightarrow \gamma^0_{i+1} \longrightarrow \cdots \longrightarrow \gamma^0_m
\text{ and } \tgamma^1_1 \longrightarrow \cdots \longrightarrow \tgamma^1_s,\ \text{ respectively. } \nonumber
\end{align}
Since $c_{\proj}^{\gamma^0_1}$ (respectively, $c_{\proj}^{\tgamma^0_1}$) is projective, then both conditions
(\textbf{LMC}) and (\textbf{RMC}) are satisfied.
The condition (\textbf{RMC}) (respectively, (\textbf{LMC})) shows that there is no arc $\gamma^0_{m+1}$ (respectively, $\tgamma^0_{n+1}$)
such that $\gamma^0_1, \cdots, \gamma^0_m, \gamma^0_{m+1}$ (respectively, $\tgamma^0_1, \cdots, \tgamma^0_n, \tgamma^0_{n+1}$)
have the same endpoint $p$ (respectively, $q$) and surround $p$ (respectively, $q$) in the counterclockwise order.
Then the string of $M_2$ corresponds to the permissible curve $c_{\proj}^{\gamma}$ which is projective,
and  the strings of $M_1$ and $M_3$ correspond to $c'$ and $\tilde{c}'$, respectively.
\end{proof}

\begin{remark}\rm \label{rmk. pdim. of inj. mod.}
If $\gamma$, $\gamma^0_1$ and $\tgamma^0_1$ satisfy the {\rm P}-condition (we have $i<m$ and $j<n$ in this case),
then in $\Surf_A\backslash \partial\Surf_A$, $c_{\inj}^\gamma$ and $c_{\proj}^{\gamma_1^0}$ have an intersection $x$,
and $c_{\inj}^\gamma$ and $c_{\proj}^{\tgamma_1^0}$ have an intersection $y$.
The permissible curves $c'$, $\tilde{c}'$ and $c_{\proj}^{\gamma}$ can be understood as obtained by
``negative rotating respect to $c_{\proj}^{\gamma^1_1} \oplus c_{\proj}^{\tgamma^1_1}$" of $c_{\inj}^{\gamma}$.
To be precise, $c_{\inj}^{\gamma}$ is divided into three parts $c_{\mathrm{I}}$, $c_{\mathrm{II}}$ and $c_{\mathrm{III}}$,
$x$ is the common endpoint of $c_{\mathrm{I}}$ and $c_{\mathrm{II}}$, and $y$ is the common endpoint of $c_{\mathrm{II}}$ and $c_{\mathrm{III}}$.
Then $c'$ (respectively, $\tilde{c}'$) can be obtained as follows:
\begin{itemize}
\item[{Step 1.}] following the positive direction of boundary, move the endpoint of $c_{\mathrm{I}}$
(respectively, $c_{\mathrm{III}}$) lying in $\M_A$ to the previous vertex of the elementary polygon $\Delta_{\gamma^0_1}$
(respectively, $\Delta_{\tgamma^0_1}$), where $\gamma^0_1$ and $\gamma^1_1$ (respectively, $\tgamma^0_1$ and $\tgamma^1_1$)
are sides of $\Delta_{\gamma^0_1}$ (respectively, $\Delta_{\tgamma^0_1}$);
\item[{Step 2.}] move $x$ (respectively, $y$) to the endpoint of $c_{\proj}^{\gamma^1_1}$ (respectively, $c_{\proj}^{\tgamma^1_1}$)
such that $c_{\mathrm{I}}$ (respectively, $c_{\mathrm{III}}$) makes a negative rotation.
\end{itemize}
\end{remark}

\begin{proposition} \label{prop. pdim. of inj. mod.}
Keeping the notations in Lemma \ref{lemm. proj. covers of ind. inj. mods.}, we have
$$\pdim M(c_{\inj}^\gamma) = \max\big\{1, \mathfrak{C}(\Delta_{\gamma^0_1})-1, \mathfrak{C}(\Delta_{\tgamma^0_1})-1 \big\},$$
where $\Delta_{\gamma^0_1}$ $($respectively, $\Delta_{\tgamma^0_1}$$)$ is the elementary polygon with sides
$\gamma^0_1$ and $\gamma^1_1$ $($respectively, $\tgamma^0_1$ and $\tgamma^1_1$$)$.
To be more precise,
\begin{itemize}
\item[\rm(1)] if $\gamma^0_1$ and $\gamma^1_1$ $($respectively, $\tgamma^0_1$ and $\tgamma^1_1$$)$ exist, then
$\pdim M(c_{\inj}^\gamma) = \max\{\mathfrak{C}(\Delta_{\gamma^0_1}), \mathfrak{C}(\Delta_{\tgamma^0_1})\}-1$;
\item[\rm(2)] if at least one of $\gamma^0_1$ and $\gamma^1_1$ does not exists
and at least one of $\tgamma^0_1$ and $\tgamma^1_1$ does not exists, then $\pdim M(c_{\inj}^\gamma)=1$.
\end{itemize}
\end{proposition}

\begin{proof} Suppose that $\gamma^0_1$ and $\gamma^1_1$ (respectively, $\tgamma^0_1$ and $\tgamma^1_1$) exist.
Similarly to Proposition \ref{prop. pdM and idM}, regard $\gamma^0_1$ and $\gamma^1_1$ as the functions
$\gamma^0_1, \gamma^1_1 : (0,1)\to \Surf_A\backslash\partial\Surf_A$, where $\gamma^0_1(0^+)=p$
and $\gamma^0_1(1^-) = \gamma^1_1(0^+)$.
If there are arcs $\gamma^2_1, \cdots, \gamma^{\theta}_1: (0,1)\to \Surf_A\backslash\partial\Surf_A$
such that $\gamma^0_1,\cdots,\gamma^{\theta}_1\in\mathfrak{C}(\Delta_{\gamma^0_1})$ satisfy (PD1)--(PD3),
then the minimal projective resolution of $M(c')$ is of the form:
\begin{align}
\cdots \to M(c_{\proj}^{\gamma^{\theta}_1}) \to \cdots\to M(c_{\proj}^{\gamma^2_1})
\to M(c_{\proj}^{\gamma^1_1}) \to M(c') \to 0. \nonumber
\end{align}
We claim that $\Delta_{\gamma_1^0}$ is not $\infty$-elementary.
Otherwise, there exists an arc $\gamma^{\vartheta}_1$ with $\vartheta \ge \theta$ such that
$\gamma^{\vartheta}_1(1^-) = \gamma^0_1(0^+)$, this is a contradiction because $c_{\inj}^{\gamma}$ is injective.
Since $\sharp\M_A < \infty$, there is a sequence of arcs $\gamma^{\theta}_1, \gamma^{\theta+1}_1, \cdots, \gamma^d_1$
($d\ge \theta$) such that $\gamma^0_1, \gamma^1_1, \cdots, \gamma^d_1$ satisfy (\textbf{PD}) by Proposition \ref{prop. pdM and idM}.
Then the minimal projective resolution of $M(c')$ is as follows:
\begin{align}\label{(6.1)}
0 \to M(c_{\proj}^{\gamma^{d}_1}) \to \cdots\to M(c_{\proj}^{\gamma^2_1})
\to M(c_{\proj}^{\gamma^1_1}) \to M(c') \to 0. %\eqno{(6.1)}
\end{align}
%We can consider the elementary $\Delta_{\tgamma_0^1}$ by the same method.}
Without loss of generality, suppose that
\begin{itemize}
  \item there are arcs $\gamma^0_1, \gamma^1_1, \cdots, \gamma^d_1$
    with $d = \mathfrak{C}(\Delta_{\gamma^0_1})$ satisfy (\textbf{PD});

  \item and dually, there are arcs $\tgamma^0_1, \tgamma^1_1, \cdots, \tgamma^{\tilde{d}}_1$
    with $\tilde{d} = \mathfrak{C}(\Delta_{\tgamma^0_1})$ satisfy (\textbf{PD})$'$.
\end{itemize}
\noindent
Similarly, we get the minimal projective resolution of $M(\tilde{c}')$ as follows:
\begin{align}\label{(6.2)}
0 \to M(c_{\proj}^{\tgamma^{\tilde{d}}_1}) \to \cdots \to M(c_{\proj}^{\tgamma^2_1})
\to M(c_{\proj}^{\tgamma^1_1}) \to M(\tilde{c}') \to 0. %\eqno{(6.2)}
\end{align}
Thus
\[\pdim M(c_{\inj}^\gamma) = \max\{ \pdim M(c'), \pdim M(\tilde{c}') \}-1 = \max\{d,\tilde{d}\}-1.\]
%where $\theta=\mathfrak{C}(\Delta_{\gamma^0_1})$ and $\tilde{\theta}=\mathfrak{C}(\Delta_{\tgamma^0_1})$

Now suppose that at least one of $\gamma^0_1$ and $\gamma^1_1$ does not exists
and at least one of $\tgamma^0_1$ and $\tgamma^1_1$ does not exists.
Then $M(c_{\proj}^{\gamma^1_1})=0=M(c_{\proj}^{\tgamma^1_1})$ in (\ref{(6.1)}) and (\ref{(6.2)}),
and hence $ M(c')=0= M(\tilde{c}')$. By Lemma \ref{lemm. proj. covers of ind. inj. mods.}(2),
we have that $\Ker\ p(M(c_{\inj}^{\gamma})) \cong M(c_{\proj}^{\gamma})$ is projective and
$\pdim M(c_{\inj}^\gamma)=1$.
\end{proof}

For any basic finite dimensional $k$-algebra $A$,
we know that $A\cong \bigoplus_{i\in I} e_iA$ and $D(A) \cong \bigoplus_{i\in I}D(Ae_i)$
where $D=\Hom_k(-,k)$ and $\{e_i\mid i\in I\}$ is a complete set of primitive orthogonal idempotents of $A$,
each $e_iA$ is an indecomposable projective right $A$-module, and each $D(Ae_i)$ is an indecomposable injective right $A$-module.
Thus Proposition \ref{prop. pdim. of inj. mod.} provides a method for calculating the projective dimension of $D(A)$.
Similarly, we can establish the dual of Proposition \ref{prop. pdim. of inj. mod.} and calculate the injective dimension of $A$.

\begin{theorem}\label{thm. Gdim A}
Let $A=kQ/I$ be a non-simple gentle algebra and $\Surf_A=(\Surf_A,\M_A,\Gamma_A)$ its marked ribbon surface.
If the quiver $(Q, I)$ of $A$ is not an oriented cycle with full relations, then
\begin{align}\label{(6.3)}
\idim_AA=\idim A_A= \max\limits_{\Delta_j\in\finiteDelta}\big\{1, {\mathfrak{C}(\Delta_j)}-1\big\}
=\max\limits_{\mathit{\Pi}\in\fforbid_A} l(\mathit{\Pi}).
\end{align}
\end{theorem}

\begin{proof}
Let $A=kQ/I$ be a gentle algebra. Then $\idim_AA=\idim A_A= \pdim D(A)$.
Similarly to the proofs of Proposition \ref{prop. pdM and idM} and Theorem \ref{thm. gl.dim}, we get
$$\pdim D(A) = \max\big\{1, \max\limits_{\Delta_i\in\finiteDelta}{\mathfrak{C}(\Delta_i)}-1 \big\}=
\max\limits_{\Delta_j\in\finiteDelta}\big\{1, {\mathfrak{C}(\Delta_j)}-1\big\}=
\max\limits_{\mathit{\Pi}\in\fforbid_A} l(\mathit{\Pi})$$
by Proposition \ref{prop. pdim. of inj. mod.}.
\end{proof}

Note that the number of all elements in $\mathfrak{C}(\Delta_i)$ is equal to that of all sides of $\Delta_i$ minus $1$.
Thus, Theorem \ref{thm. Gdim A} shows that the self-injective dimension of $A$ is equal to the number of sides
of element polygon(s) with the largest number of sides minus $2$.
Moreover, by \cite[Theorem 3.4]{GR2005}, any gentle algebra is Gorenstein.
Thus, the finitistic dimension of a gentle algebra equals to its left and right self-injective dimensions (see \cite{H1991}).
Thus Theorem \ref{thm. Gdim A} provides a description of the finitistic dimension of a gentle algebra.

% notations
\newcommand{\ocf}{\mathrm{ocf}}

\newcommand{\GP}{\mathcal{GP}}
\newcommand{\nonprojGP}{\mathrm{nproj}\text{-}\mathcal{GP}}
\subsection{The number of indecomposable Gorenstein projective modules}
Recall that an algebra is said to be {\it CM-finite} if the number of indecomposable Gorenstein projective modules (up to isomorphism) is finite.
It is known that gentle algebras are CM-finite \cite{CL2019}.
%We use $\GP(A)$ to denote the subcategory of $\bmod A$ consisting of all Gorenstein projective modules,
We use $\ind(\nonprojGP(A))$ to denote the subcategory of $\bmod A$ consisting of all indecomposable non-projective Gorenstein projective modules.
%We use $\sharp\nonprojGP(A)$ to denote the number of indecomposable non-projective Gorenstein projective modules up to isomorphism.
In this subsection, we obtain a formula for calculating $\sharp\ind(\nonprojGP(A))$ by AG-invariants.
%(see Proposition \ref{prop. CM-corresponding}).
%and furthermore, re-prove the conclusion proved by Chen and Lu through this formula.

\begin{lemma} \label{lemm. CM1}
Let $A=kQ/I$ be a gentle algebra, and let $c_1=\alpha_1\alpha_2\cdots\alpha_m$ and $c_2=\beta_1\beta_2\cdots\beta_n$
be two oriented cycles with full relations of $A$.
If $c_1$ and $c_2$ have a common vertex $v_0$, that is, there are arrows $\alpha_h$, $\alpha_{h+1}$,
$\beta_{\hbar}$ and $\beta_{\hbar+1}$ such that $v_0 = t(\alpha_h) = t(\beta_\hbar) = s(\alpha_{h+1}) = s(\beta_{\hbar+1})$,
then $\alpha_{h} A$ and $\beta_{\hbar} A$ are indecomposable Gorenstein projective modules satisfying
$\mathrm{top}(\alpha_{h} A) = \mathrm{top} (\beta_{\hbar} A) \cong M(c_{\simp}^{\gamma_0}) = S(v_0)$.
\end{lemma}

\begin{proof}
We know that $\alpha_{h}A$ and $\beta_{\hbar}A$ are indecomposable non-projective Gorenstein projective by Theorem \ref{thm. Gproj},
Now we show that $\mathrm{top}(\alpha_{h} A) = \mathrm{top} (\beta_{\hbar} A) \cong M(c_{\simp}^{\gamma_0}) = S(v_0)$ by using
the marked ribbon surface $\Surf_A = (\Surf_A,\M_A,\Gamma_A)$ of $A$.
Let $\gamma_0$ be the arc in $\Gamma_A$ corresponding to $v_0$ by the bijection $\mathfrak{v}: Q_0 \to \Gamma_A$.
Then $\gamma_0$ is a common side of two $\infty$-elementary polygons $\Delta_1$ and $\Delta_2$, that is,
there are arcs $\gamma_{-1}, \gamma_1, \gamma_2, \cdots, \gamma_a$ and $\tgamma_{-1}, \tgamma_1, \tgamma_2, \cdots, \tgamma_b$ such that
\begin{itemize}
\item $\gamma_{-1}, \gamma_0 \gamma_1, \cdots, \gamma_a$ (respectively, $\tgamma_{-1}, \tgamma_0=\gamma_0, \tgamma_1, \cdots, \tgamma_b$)
have a common endpoint $p$ (respectively, $q$) and $\gamma_{i+1}$ (respectively, $\tgamma_{j+1}$) appears after $\gamma_{i}$
(respectively, $\tgamma_{j}$) for any $-1\le i\le a-1$ (respectively, $-1\le j\le b-1$);
\item $\gamma_{-1}, \gamma_0\in\mathfrak{S}(\Delta_2)$, $\gamma_0, \gamma_1\in\mathfrak{S}(\Delta_1)$ and
$\gamma_i \notin\mathfrak{S}(\Delta_1)$ for any $3\le i\le a$;
\item $\tgamma_{-1}, \tgamma_0\in\mathfrak{S}(\Delta_1)$, $\tgamma_0, \tgamma_1\in\mathfrak{S}(\Delta_2)$
and $\gamma_j \notin\mathfrak{S}(\Delta_2)$ for any $3\le j\le b$;
\item $\mathfrak{v}^{-1}(\gamma_1) = t(\alpha_{h+1})$ and $\mathfrak{v}^{-1}(\tgamma_1) = t(\beta_{\hbar+1})$.
\end{itemize}
For the permissible curve $c$ consecutively crossing $\gamma_0, \gamma_1, \cdots, \gamma_a$ and $\tilde{c}$ consecutively
crossing $\tgamma_0, \tgamma_1, \cdots, \tgamma_b$ such that

\begin{align}
   & p = p(\gamma_0, \gamma_1, c) = p(\gamma_1, \gamma_2, c) = \cdots = p(\gamma_{a-1}, \gamma_a, c) \nonumber \text{ and }\\
  & q= p(\tgamma_0, \tgamma_1, \tilde{c}) = p(\tgamma_1, \tgamma_2, \tilde{c}) = \cdots = p(\tgamma_{b-1}, \tgamma_b, \tilde{c}), \nonumber
\end{align}
we have $M(c)\cong \alpha_{h} A$, $M(\tilde{c})\cong \beta_{\hbar}A$, and $c$ and $\tilde{c}$ correspond to the strings
\begin{align}
& \mathfrak{v}^{-1}(\gamma_a)  \longleftarrow \cdots
\mathop{\longleftarrow}\limits^{\alpha_{h+2}} \mathfrak{v}^{-1}(\gamma_1)
\mathop{\longleftarrow}\limits^{\alpha_{h+1}} \mathfrak{v}^{-1}(\gamma_0) \text{ and }
\nonumber \\
&
\mathfrak{v}^{-1}(\tgamma_0) \mathop{\longrightarrow}\limits^{\beta_{\hbar+1}}
\mathfrak{v}^{-1}(\tgamma_1) \mathop{\longrightarrow}\limits^{\beta_{\hbar+2}} \cdots \longrightarrow  \mathfrak{v}^{-1}(\tgamma_b),\ \text{respectively,}
\nonumber
\end{align}
where $\mathfrak{v}^{-1}(\gamma_0) = \mathfrak{v}^{-1}(\tgamma_0) = v_0$.
So $\mathrm{top}(\alpha_{h} A) \cong \mathrm{top} M(c) \cong M(c_{\simp}^{\gamma_0})$
and $\mathrm{top}(\beta_{\hbar} A) \cong \mathrm{top} M(\tilde{c}) \cong M(c_{\simp}^{\tgamma_0})$.
Notice that $\gamma_0=\tgamma_0$, thus $M(c_{\simp}^{\gamma_0}) = M(c_{\simp}^{\tgamma_0})$ is
isomorphic to the simple module $S(\mathfrak{v}^{-1}(\gamma_0)) = S(v_0)$.
\end{proof}

\begin{lemma}\label{lemm. CM2: CM-corresponding}
Let $A=kQ/I$ be a gentle algebra, and let $Q_0^{\mathrm{ocf}}$ be the set of all vertices on the oriented cycle with full relations.
Then the map $f: \ind(\nonprojGP(A))\to Q_0^{\mathrm{ocf}}$ via $\alpha A \mapsto t(\alpha)$ is surjective;
furthermore, if $v\in Q_0^{\mathrm{ocf}}$ is a common vertex of two oriented cycles with full relations,
then $\sharp f^{-1}(v) = \sharp \{\alpha A \in \ind(\nonprojGP(A))| t(\alpha)=v \} = 2$;
and if otherwise, then $\sharp f^{-1}(v) = 1$.
\end{lemma}

\begin{proof}
For any $v \in Q_0^{\mathrm{ocf}}$, there exists at least one oriented cycle with full relations $\alpha_0\cdots\alpha_{l-1}$
such that $v = t(\alpha_{\overline{i}})=s(\alpha_{\overline{i+1}})$ for some $0\le i\le l$ where $\overline{i}$ is equal to $i$ modulo $l$.
Then by Theorem \ref{thm. Gproj}, we have that $\alpha_{\overline{i}} A$ is an indecomposable non-projective Gorenstein projective module and $f$ is surjective.

If $v$ is a common vertex of two oriented cycles with full relations $\alpha_0\alpha_1\cdots\alpha_{l-1}$ and $\beta_0\beta_1\cdots\beta_{\ell-1}$,
then there are arrows $\alpha_{\overline{i}}$, $\alpha_{\overline{i+1}}$, $\beta_{\underline{j}}$ and $\beta_{\underline{j+1}}$
such that $v = t(\alpha_{\overline{i}}) = s(\alpha_{\overline{i+1}}) = t(\beta_{\underline{j}}) = s(\beta_{\underline{j+1}})$
where $\overline{i}$ is equal to $i$ modulo $l$ and $\underline{j}$ is equal to $j$ modulo $\ell$.
Thus we have $\{\alpha_{\overline{i}} A, \beta_{\underline{j}} A\}\subseteq f^{-1}(v)$ by Lemma \ref{lemm. CM1};
and furthermore $\sharp f^{-1}(v)\ge 2$.
Notice that $A$ is gentle, thus there is no other arrow whose source or sink is $v$. Therefore $\sharp f^{-1}(v)\le 2$
and $\sharp f^{-1}(v)= 2$.

It is obvious that if there exists a unique oriented cycle $\mathsf{s} = \alpha_0\alpha_1\cdots\alpha_l$
with full relations of $A$ such that $v\in Q_0^{\ocf}$ is a vertex of $\mathsf{s}$, then $\sharp f^{-1}(v)=1$.
Indeed, suppose $t(\alpha_i)=v$ with $0\le i\le l$. Then $f^{-1}(v) = \{\alpha_i A\} \subseteq \ind(\nonprojGP(A))$.
\end{proof}

On the marked ribbon surface $\Surf_A = (\Surf_A, \M_A, \Gamma_A)$ of a gentle algebra $A=kQ/I$,
for each $\infty$-elementary polygon, its each side provides an indecomposable non-projective
Gorenstein projective module by the permissible curve satisfying Proposition \ref{prop. GProj} (GP2),
and the arc which is a common side of two $\infty$-elementary polygons provides two indecomposable
non-projective Gorenstein projective modules. From this observation, we can also prove Lemma
\ref{lemm. CM2: CM-corresponding}. Moreover, we have the following proposition.

\begin{proposition} \label{prop. CM-corresponding}
Let $A=kQ/I$ be a gentle algebra with the AG-invariant $\phi_A$, and let $L\subseteq\mathbb{N}^+$
be a finite set such that $\phi_A(0, \ell)>0$ if $\ell \in L$, otherwise, $\phi_A(0, \ell)=0$.
Then we have
\begin{enumerate}
\item[{\rm (1)}]
\[\sharp \ind(\nonprojGP (A)) = \sum\limits_{\ell\in L} \ell\cdot\phi_A(0, \ell), \]
that is, $\sharp \ind(\nonprojGP (A))$ is the number of arrows on oriented cycles with full relations.
\item[{\rm (2)}] Furthermore, if two gentle algebras $A$ and $A'$ are AG-equivalent, then
\[\sharp \ind(\nonprojGP(A)) = \sharp\ind(\nonprojGP(A')).\]
\end{enumerate}
\end{proposition}

\begin{proof}
Recall from the proof of Proposition \ref{prop. infty. gl.dim.} that for each $\ell\in L$, the number of oriented cycles
with full relations of length $\ell$ is $\phi_A(0, \ell)$.
Thus we can suppose that $A$ has $t=\sum\limits_{\ell\in L} \phi_A(0, \ell)$ oriented cycles with full relations
$\mathsf{s}^i = \alpha_1^i\alpha_2^i \cdots \alpha_{\ell_i}^i$ with $1\le i \le t$.
Then there exist are the following maps:
\[   Q_1^{\ocf} \mathop{\longrightarrow}\limits^{g}_{\sim} \ind(\nonprojGP (A))  \mathop{\longrightarrow}\limits^{f} Q_0^{\ocf}, \]
where $Q_1^{\ocf} = \bigcup\limits_{1\le i\le t} \{\alpha_j^{i} | 1\le j\le \ell_i\}$ and
$f$ is surjective (see Lemma \ref{lemm. CM2: CM-corresponding}),
and $g$ defined by $\alpha\mapsto \alpha A$ is bijective (see Theorem \ref{thm. Gproj}).
Thus $\sharp \ind(\nonprojGP(A'))$ is equal to the number $\sum_{\ell\in L} \ell\cdot\phi_A(0, \ell)$
of arrows on oriented cycles with full relations.

Let $A'$ be a gentle algebra such that $\phi_{A'}=\phi_A$. Then
\[\sharp \ind(\nonprojGP(A')) = \sum\limits_{\ell\in L} \ell\cdot\phi_{A'}(0, \ell)
= \sum\limits_{\ell\in L} \ell\cdot\phi_A(0, \ell) =  \sharp \ind(\nonprojGP (A)).\]
The second equality is obtained by $\phi_{A'}=\phi_A$.
\end{proof}

\begin{corollary} \label{coro. gentle is CM finite} {\rm \cite{CL2019,K2015}}
Gentle algebras are CM-finite.
\end{corollary}

\begin{proof}
Let $A=kQ/I$ be a gentle algebra. Since $(Q,I)$ is a finite quiver, the number of oriented cycles with full relations is finite,
and so there exists a finite set $L \subseteq \mathbb{N}^+$ such that $0<\phi_A(0, \ell)<\infty$ for any $\ell\in L$,
and $\phi_A(0,\ell)=0$ for any $\ell\notin L$.
By Proposition \ref{prop. CM-corresponding}, we have $\sharp \ind(\nonprojGP (A))
= \sum\limits_{\ell\in L} \ell\cdot\phi_A(0, \ell) < \infty$.
\end{proof}

\section{\textbf{Examples}} \label{Section appl.}

\subsection{Some examples for calculating global dimension of gentle algebras}
In this subsection, we provide some examples to calculate the global dimension of gentle algebras
and the projective dimension of simple modules.

\begin{example} \rm \label{exp. 1}
For any $n\geq 1$, there exists a gentle one-cycle algebra $A$, that is,
its quiver has exactly one cycle, such that $\mathrm{gl.dim}A = n$.
Indeed, if $A$ is such a gentle one-cycle algebra whose marked ribbon surface $\Surf_A$
is shown in the \Pic \ref{Examples 1} subfigure 1,
then $A\cong kQ/I$, where $Q_0=\{\gamma^0,\gamma^1,\cdots,\gamma^{n-1}\}$,
$Q$ is shown in the \Pic \ref{Examples 1} subfigure 2 and
$I = \langle \alpha_1\alpha_2, \alpha_2\alpha_3, \cdots, \alpha_{n-1}\alpha_n \rangle$.
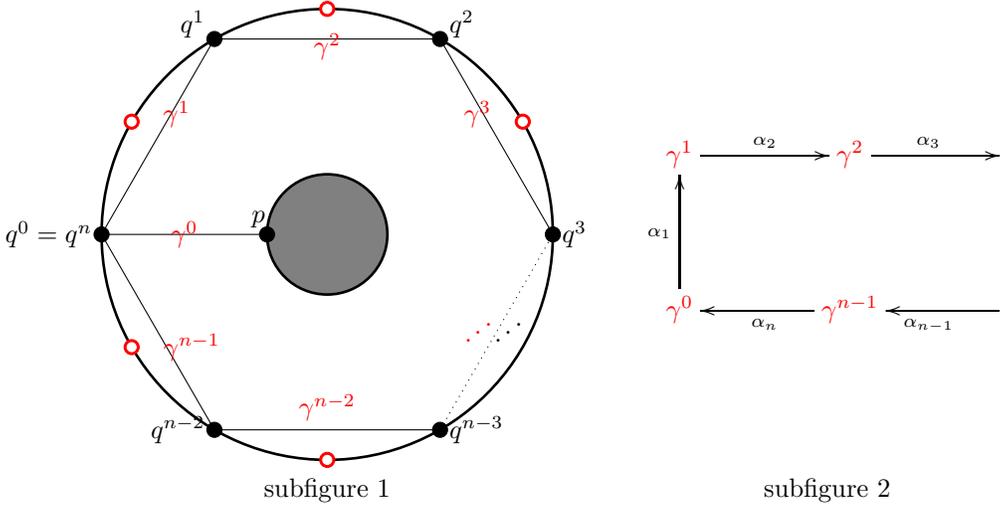
\begin{figure}[htbp]
\begin{center}
\begin{tikzpicture}[scale=2]
 \filldraw[gray] (0,0) circle (0.4);
 \draw[line width=1pt] (0,0) circle (1.5);
 \draw[line width=1pt] (0,0) circle (0.4);
 %%%%%%%%%% half-edges and points %%%%%%%%%%
 \filldraw ( 1.50, 0.00) circle (0.05); \draw ( 1.50, 0.00) node[right] {$q^3$};
 \filldraw ( 0.75, 1.30) circle (0.05); \draw ( 0.75, 1.40) node[right] {$q^2$};
 \filldraw (-0.75, 1.30) circle (0.05); \draw (-0.75, 1.40) node[left]  {$q^1$};
 \filldraw (-1.50, 0.00) circle (0.05); \draw (-1.50, 0.00) node[left]  {$q^0=q^n$};
 \filldraw (-0.75,-1.30) circle (0.05); \draw (-0.75,-1.30) node[left]  {$q^{n-2}$};
 \filldraw ( 0.75,-1.30) circle (0.05); \draw ( 0.75,-1.30) node[right] {$q^{n-3}$};
 \filldraw[color=red]   ( 0.00, 1.50) circle (0.05); \filldraw[color=white] ( 0.00, 1.50) circle (0.03);
 \filldraw[color=red]   (-1.30, 0.75) circle (0.05); \filldraw[color=white] (-1.30, 0.75) circle (0.03);
 \filldraw[color=red]   (-1.30,-0.75) circle (0.05); \filldraw[color=white] (-1.30,-0.75) circle (0.03);
 \filldraw[color=red]   ( 0.00,-1.50) circle (0.05); \filldraw[color=white] ( 0.00,-1.50) circle (0.03);
 \filldraw[color=red]   ( 1.30, 0.75) circle (0.05); \filldraw[color=white] ( 1.30, 0.75) circle (0.03);
%\filldraw (-1.06, 1.06) circle (0.05);
%\filldraw (-1.45, 0.39) circle (0.05);
 \filldraw (-0.40, 0.00) circle (0.05); \draw (-0.35, 0.10) node[left] {$p$};
 \draw (-1.50, 0.00) to (-0.40, 0.00);
 \draw ( 1.50, 0.00) to ( 0.75, 1.30);
 \draw ( 0.75, 1.30) to (-0.75, 1.30);
 \draw (-0.75, 1.30) to (-1.50, 0.00); % \draw (-0.75, 1.30) to[out=-80,in=10] (-1.50, 0.00);
 \draw (-1.50, 0.00) to (-0.75,-1.30);
 \draw (-0.75,-1.30) to ( 0.75,-1.30);
 \draw[dotted] ( 0.75,-1.30) to ( 1.50, 0.00);
% \draw (-0.75, 1.30) to[out=-100,in= -10] (-1.06, 1.06);
% \draw (-1.20, 0.90) node{$\vdots$};
% \draw (-0.75, 1.30) to[out=-90,in=10] (-1.45, 0.39);
 \draw ( 1.20,-0.60) node{$\iddots$};
 %%%%% numbers %%%%%
% \draw[red] (-0.90, 1.15) node{$\gamma^{-m}$};
% \draw[red] (-1.00, 0.95) node{$\vdots$};
% \draw[red] (-0.90, 0.75) node{$\gamma^{-1}$};
 \draw[red] (-1.00, 0.80) node{$\gamma^1$}; % \draw[red] (-0.85, 0.50) node{$\gamma^1$};
 \draw[red] ( 0.00, 1.25) node{$\gamma^2$};
 \draw[red] ( 1.00, 0.80) node{$\gamma^3$};
 \draw[red] ( 1.00,-0.60) node{$\iddots$};
 \draw[red] ( 0.00,-1.15) node{$\gamma^{n-2}$};
 \draw[red] (-0.90,-0.75) node{$\gamma^{n-1}$};
 \draw[red] (-0.95, 0.00) node{$\gamma^{0}$};
\draw (0,-1.70) node{subfigure 1};
\end{tikzpicture}
\ \ \
\begin{tikzpicture}[scale=2]
\draw (0,0) node{$
\xymatrix@C=1.5cm @R=1.5cm{  {\color{red}\gamma^1} \ar[r]^{\alpha_2}&  {\color{red}\gamma^2} \ar[r]^{\alpha_3} & \ar@{.}[d] \\
 {\color{red}\gamma^0} \ar[u]^{\alpha_1} &  {\color{red}\gamma^{n-1}} \ar[l]^{\alpha_{n}} & \ar[l]^{\alpha_{n-1}} }
$};
\draw (0,-1.70) node{subfigure 2};
\end{tikzpicture}
\caption{The subfigure 2 is the gentle algebra whose marked ribbon surface is shown in subfigure 1.}
\label{Examples 1}
\end{center}
\end{figure}

\noindent Thus $\mathrm{gl.dim}A=n$ by Theorem \ref{thm. gl.dim}.
Let $\gamma^0$ be the arc $\gamma^0: (0,1)\to \Surf_A\backslash\partial\Surf_A$
such that $\gamma^0(0^+) = p$ and $\gamma^0(1^-) = q^0$,
and let $\gamma^i$ be the arc $\gamma^i: (0,1)\to \Surf_A\backslash\partial\Surf_A$
such that $\gamma^i(0^+) = q^{i-1}$ and $\gamma^i(1^-) = q^i$
for any $1\le i\le n-1$, and let $\gamma^n$ be the arc $\gamma^0(1-x)$
(note: $\gamma^n=\gamma^0(1-x) \simeq \gamma^0(x)=\gamma^0$).
Then the arcs $\gamma^0, \gamma^1, \cdots, \gamma^{n-1}, \gamma^n$ satisfy the condition (\textbf{PD}),
and the minimal projective resolution of the simple module $M(c_{\simp}^{\gamma^0})$
can be determined by Proposition \ref{prop. pdM and idM} as shown in \Pic \ref{fig:exp-7.1},
where $c_i \simeq c_{\proj}^{\gamma^i}$; $c_{i+1}'$ is the permissible curve corresponding to
the kernel of $\delta_{i}$ ($1\le i\le n-1$) and $c_n \simeq c_n' \simeq c_{\proj}^{\gamma^n}$.
\begin{figure}[htbp]
\begin{center}
\begin{tikzpicture}[scale=2.5]
 \filldraw[black!25] (0,0) circle (0.4);
 \draw[line width=1pt] (0,0) circle (1.5);
 \draw[line width=1pt] (0,0) circle (0.4);
 %%%%%%%%%% half-edges and points %%%%%%%%%%
 \filldraw ( 1.50, 0.00) circle (0.05); \draw ( 1.50, 0.00) node[right] {$q^3$};
 \filldraw ( 0.75, 1.30) circle (0.05); \draw ( 0.75, 1.40) node[right] {$q^2$};
 \filldraw (-0.75, 1.30) circle (0.05); \draw (-0.75, 1.40) node[left]  {$q^1$};
 \filldraw (-1.50, 0.00) circle (0.05); \draw (-1.50, 0.00) node[left]  {$q^0=q^n$};
 \filldraw (-0.75,-1.30) circle (0.05); \draw (-0.75,-1.30) node[left]  {$q^{n-2}$};
 \filldraw ( 0.75,-1.30) circle (0.05); \draw ( 0.75,-1.30) node[right] {$q^{n-3}$};
 \filldraw[color=red]   ( 0.00, 1.50) circle (0.05); \filldraw[color=white] ( 0.00, 1.50) circle (0.03);
 \filldraw[color=red]   (-1.30, 0.75) circle (0.05); \filldraw[color=white] (-1.30, 0.75) circle (0.03);
 \filldraw[color=red]   (-1.30,-0.75) circle (0.05); \filldraw[color=white] (-1.30,-0.75) circle (0.03);
 \filldraw[color=red]   ( 0.00,-1.50) circle (0.05); \filldraw[color=white] ( 0.00,-1.50) circle (0.03);
 \filldraw[color=red]   ( 1.30, 0.75) circle (0.05); \filldraw[color=white] ( 1.30, 0.75) circle (0.03);
 \filldraw (-0.40, 0.00) circle (0.05); \draw (-0.40, 0.00) node[right] {$p$};
 \draw (-1.50, 0.00) to (-0.40, 0.00);
 \draw ( 1.50, 0.00) to ( 0.75, 1.30);
 \draw ( 0.75, 1.30) to (-0.75, 1.30);
 \draw (-0.75, 1.30) to (-1.50, 0.00); % \draw (-0.75, 1.30) to[out=-80,in=10] (-1.50, 0.00);
 \draw (-1.50, 0.00) to (-0.75,-1.30);
 \draw (-0.75,-1.30) to ( 0.75,-1.30);
 \draw[dotted] ( 0.75,-1.30) to ( 1.50, 0.00);
 \draw ( 1.20,-0.60) node{$\iddots$};
 %%%%% numbers %%%%%
% \draw[red] (-1.00, 0.80) node{$\gamma^1$}; % \draw[red] (-0.85, 0.50) node{$\gamma^1$};
% \draw[red] ( 0.00, 1.25) node{$\gamma^2$};
% \draw[red] ( 1.00, 0.80) node{$\gamma^3$};
% \draw[red] ( 1.00,-0.60) node{$\iddots$};
% \draw[red] ( 0.00,-1.15) node{$\gamma^{n-2}$};
% \draw[red] (-1.00,-0.85) node{$\gamma^{n-1}$};
% \draw[red] (-0.95, 0.00) node{$\gamma^{0}$};
% proj resolution
 \draw[orange] [line width=0.8pt] (-0.75,-1.30) to[out=80,in=-80] (-0.75, 1.30);
 \draw[orange]                    (-0.6,0.6) node{$c_{\simp}^{\gamma^0}$};
 \draw[ red  ] [line width=2.5pt] (-0.75,-1.30) to[out=85,in=-45] (-1.30, 0.75);
 \draw[ red  ]                    (-1.00, 0.10) node{$c_0$};
 \draw[ blue ] [line width=1.2pt] (-1.30, 0.75) -- ( 0.75, 1.30);
 \draw[ blue ]                    (-0.27, 1.02) node[below]{$c_1'$};
 \draw[ red  ] [line width=2.5pt] (-1.30, 0.75) -- ( 0.00, 1.50) [opacity=0.80];
 \draw[ red  ]                    (-0.65, 1.10) node[right]{$c_1$};
 \draw[ blue ] [line width=1.2pt] ( 0.00, 1.50) -- ( 1.50, 0.00) [opacity=0.80];
 \draw[ blue ]                    ( 0.75, 0.75) node[below]{$c_2'$};
 \draw[ red  ] [line width=2.5pt] ( 0.00, 1.50) -- ( 1.30, 0.75) [opacity=0.70];
 \draw[ red  ]                    ( 0.65, 1.10) node[left]{$c_2$};
 \draw[ blue ] [line width=1.2pt] ( 1.30, 0.75) -- ( 0.75,-1.30) [opacity=0.70][dotted];
 \draw[ red  ] [line width=2.5pt] ( 1.30, 0.75) -- ( 1.30,-0.75) [opacity=0.60][dotted];
 \draw[ blue ] [line width=1.2pt] ( 1.30,-0.75) -- (-0.75,-1.30) [opacity=0.60];
 \draw[ blue ]                    ( 0.25,-0.75) node[below]{$c_{n-3}'$};
 \draw[ red  ] [line width=2.5pt] ( 1.30,-0.75) -- ( 0.00,-1.50) [opacity=0.50];
 \draw[ red  ]                    ( 0.60,-1.20) node[left]{$c_{n-3}$};
 \draw[ blue ] [line width=1.2pt] ( 0.00,-1.50) -- (-1.50, 0.00) [opacity=0.50];
 \draw[ blue ]                    (-0.40,-0.75) node[below]{$c_{n-2}'$};
 \draw[ red  ] [line width=2.5pt] ( 0.00,-1.50) -- (-1.30,-0.75) [opacity=0.40];
 \draw[ red  ]                    (-1.20,-0.80) node[right]{$c_{n-2}$};
 \draw[ blue ] [line width=1.2pt] (-1.30,-0.75) -- (-0.40, 0.00) [opacity=0.40];
 \draw[ black] [dash pattern=on 2pt off 2pt][->](-0.85,-0.37) -- (-1.7,-0.37) node[left]{{\color{blue}$c_{n-1}'$}};
 \draw[ red  ] [line width=2.5pt] (-1.30,-0.75) -- (-1.30, 0.75) [opacity=0.30];
 \draw[ blue ] [line width=1.2pt] [opacity=0.30]
   (-1.30, 0.75) to[out= -30,in=  90] (-0.75, 0.00) to[out= -90,in= 180]
   ( 0.00,-0.75) to[out=   0,in= -90] ( 0.75, 0.00) to[out=  90,in=   0]
   ( 0.00, 0.75) to[out= 180,in= 135] (-0.40, 0.00);
 \draw[ red  ]                    (-1.10,-0.10) node{$c_{n-1}$};
 \draw[ red  ] [line width=1.2pt] [opacity=0.20]
   (-1.30, 0.75) to[out= -30,in=  90] (-0.70, 0.00) to[out= -90,in= 180]
   ( 0.00,-0.70) to[out=   0,in= -90] ( 0.70, 0.00) to[out=  90,in=   0]
   ( 0.00, 0.70) to[out= 180,in= 135] (-0.40, 0.00);
 \draw (0.725,0.00) node{{\color{red}$c_n$}$\simeq${\color{blue}$c_n'$}};
\end{tikzpicture}
\end{center}
\caption{The minimal projective resolution of $M(c_{\simp}^{\gamma^0})$. }
\label{fig:exp-7.1}
\end{figure}
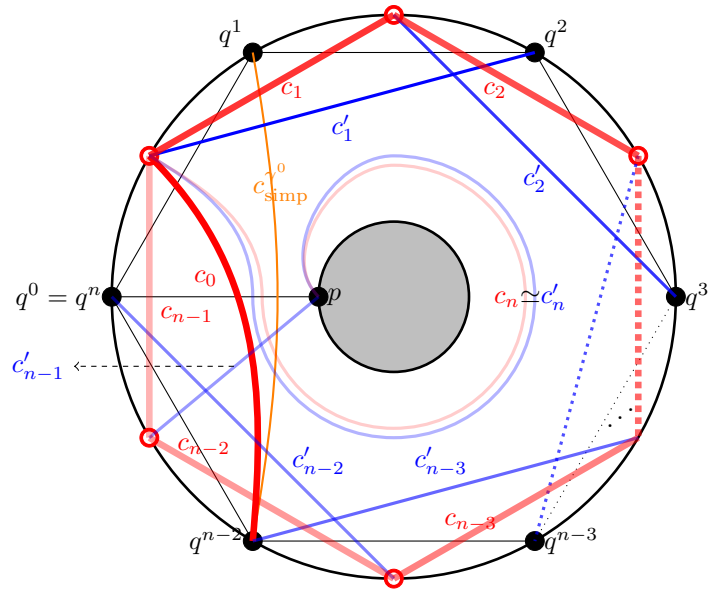
Furthermore, by formula (\ref{proj. res.}), we obtain
\[ 0 \longrightarrow M(c_{\proj}^{\gamma^n})
\mathop{\longrightarrow}\limits^{\delta_n} M(c_{\proj}^{\gamma^{n-1}})
\mathop{\longrightarrow}\limits^{\delta_{n-1}} \cdots
\mathop{\longrightarrow}\limits^{\delta_2} M(c_{\proj}^{\gamma^1})
\mathop{\longrightarrow}\limits^{\delta_1} M(c_{\proj}^{\gamma^0})
\mathop{\longrightarrow}\limits^{\delta_0} M(c_{\simp}^{\gamma^0}) \longrightarrow 0, \]
where $M(c_{\proj}^{\gamma^{n-1}}) \cong {_{\gamma^1}^{_{\gamma^0}^{\gamma^{n-1}}}}$
and $\Ker\delta\cong {_{\gamma^1}^{\gamma^0}} \cong M(c_{\proj}^{\gamma^n}). $
Thus $\pdim M(c_{\simp}^{\gamma^0}) = n$. Similarly, $\pdim M(c_{\simp}^{\gamma^i})=n-i$ for any $0\le i\le n-1$.
\end{example}

\begin{example} \rm \label{exp. 2} See \Pic \ref{Examples 2} subfigure 1, it is the marked ribbon surface of
the gentle algebra $A=kQ/I$, where $Q$ is shown subfigure 2 and
$$I= \langle a_1a_3, a_5a_4, a_8a_7, a_7a_6, a_{10}a_2, a_9a_{10}', a_{10}'a_{14}, a_{14}a_{13} \rangle.$$
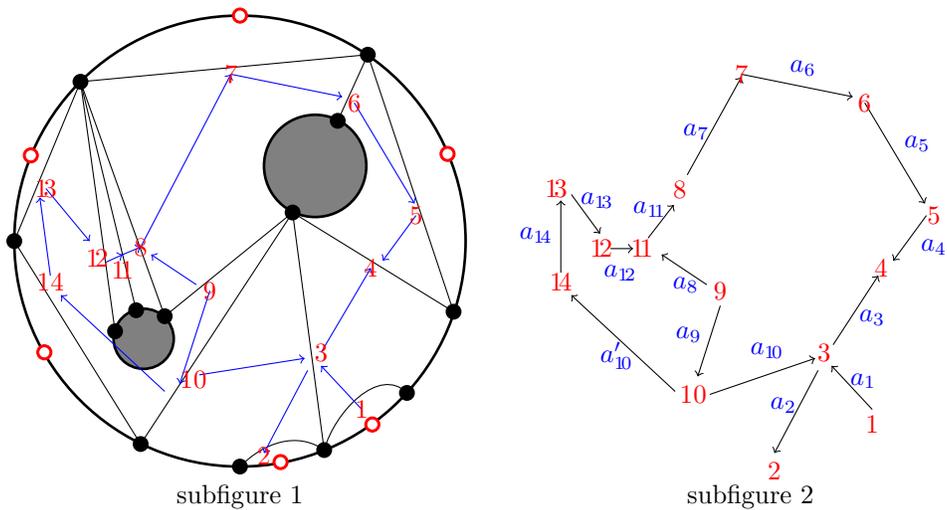
\begin{figure}[htbp]
\begin{center}
\begin{tikzpicture}[scale=2]
\draw [line width=1pt] (0,0) circle (1.5cm);
\filldraw [gray] (0.5,0.5) circle (0.34cm); \draw [line width=1pt] (0.5,0.5) circle (0.34cm);
\filldraw [gray] (-0.64,-0.65) circle (0.2cm);\draw [line width=1pt] (-0.64,-0.65) circle (0.2cm);
\draw (-1.06, 1.06)-- (-0.69,-0.46);
\draw (-1.06, 1.06)-- (-0.50,-0.50);
\draw (-1.06, 1.06)-- (-0.83,-0.60);
\draw (-0.50,-0.50)-- ( 0.35, 0.19);
\draw ( 0.35, 0.19)-- (-0.66,-1.35);
\draw (-0.66,-1.35)-- (-1.50, 0.00);
\draw (-1.06, 1.06)-- (-1.50, 0.00);
\draw (-1.06, 1.06)-- ( 0.85, 1.24);
\draw ( 0.85, 1.24)-- ( 1.42,-0.47);
\draw ( 1.42,-0.47)-- ( 0.35, 0.19);
\draw ( 0.35, 0.19)-- ( 0.56,-1.39);
\draw ( 0.56,-1.39) to[out=145, in=45] (0,-1.5);
\draw ( 0.56,-1.39) to[out=75, in=135] (1.11,-1.01);
\draw ( 0.85, 1.24)-- ( 0.65, 0.80);
\fill [color=black] (-1.06,1.06) circle (1.5pt);
\fill [color=black] (-0.69,-0.46) circle (1.5pt);
\fill [color=black] (-0.5,-0.5) circle (1.5pt);
\fill [color=black] (-0.83,-0.6) circle (1.5pt);
\fill [color=black] (0.35,0.19) circle (1.5pt);
\fill [color=black] (-0.66,-1.35) circle (1.5pt);
\fill [color=black] (-1.5,0) circle (1.5pt);
\fill [color=black] (0.85,1.24) circle (1.5pt);
\fill [color=black] (1.42,-0.47) circle (1.5pt);
\fill [color=black] (0.56,-1.39) circle (1.5pt);
\fill [color=black] (0,-1.5) circle (1.5pt);
\fill [color=black] (1.11,-1.01) circle (1.5pt);
\fill [color=black] (0.65,0.8) circle (1.5pt);
% extra
\filldraw[color=red]   ( 0.00, 1.50) circle (0.05); \filldraw[color=white] ( 0.00, 1.50) circle (0.03);
\filldraw[color=red]   ( 1.38, 0.58) circle (0.05); \filldraw[color=white] ( 1.38, 0.58) circle (0.03);
\filldraw[color=red]   (-1.39, 0.57) circle (0.05); \filldraw[color=white] (-1.39, 0.57) circle (0.03);
\filldraw[color=red]   (-1.30,-0.74) circle (0.05); \filldraw[color=white] (-1.30,-0.74) circle (0.03);
\filldraw[color=red]   ( 0.27,-1.47) circle (0.05); \filldraw[color=white] ( 0.27,-1.47) circle (0.03);
\filldraw[color=red]   ( 0.88,-1.22) circle (0.05); \filldraw[color=white] ( 0.88,-1.22) circle (0.03);
% arcs
\draw [->, blue] (-1.29, 0.35) -- (-1.00, 0.00);
\draw [->, blue] (-0.89,-0.14) -- (-0.78,-0.09);
\draw [->, blue] (-0.78,-0.09) -- (-0.66,-0.04);
\draw [->, blue] (-0.66,-0.04) -- (-0.06, 1.11);
\draw [->, blue] (-0.06, 1.11) -- ( 0.68, 0.96);
\draw [->, blue] ( 0.76, 0.92) -- ( 1.16, 0.25);
\draw [->, blue] ( 1.17, 0.17) -- ( 0.95,-0.13);
\draw [->, blue] ( 0.54,-0.74) -- ( 0.87,-0.18);
\draw [->, blue] ( 0.81,-1.12) -- ( 0.54,-0.83);
\draw [->, blue] ( 0.45,-0.86) -- ( 0.16,-1.41);
\draw [->, blue] (-0.27,-0.89) -- ( 0.43,-0.78);
\draw [->, blue] (-0.20,-0.33) -- (-0.40,-0.95);
\draw [->, blue] (-0.29,-0.29) -- (-0.59,-0.09);
\draw [->, blue] (-1.26,-0.23) -- (-1.33, 0.29);
\draw [->, blue] (-0.50,-1.00) -- (-1.19,-0.36);
\draw[red] (-1.29, 0.35) node{$1\!3$};
\draw[red] (-0.95,-0.11) node{$1\!2$};
\draw[red] (-0.78,-0.19) node{$1\!1$};
\draw[red] (-0.66,-0.04) node{$8$};
\draw[red] (-0.06, 1.11) node{$7$};
\draw[red] ( 0.76, 0.92) node{$6$};
\draw[red] ( 1.17, 0.17) node{$5$};
\draw[red] ( 0.54,-0.74) node{$3$};
\draw[red] ( 0.87,-0.18) node{$4$};
\draw[red] ( 0.81,-1.12) node{$1$};
\draw[red] ( 0.16,-1.43) node{$2$};
\draw[red] (-0.31,-0.92) node{$10$};
\draw[red] (-0.20,-0.33) node{$9$};
\draw[red] (-1.26,-0.27) node{$14$};
\draw (0,-1.70) node{subfigure 1};
\end{tikzpicture}
\ \ \ \
\begin{tikzpicture}[scale=2]
%quiver
\draw [->] (-1.19, 0.31) -- (-1.00, 0.05);  \draw[red] (-1.29, 0.35) node{$1\!3$};
\draw [->] (-0.93,-0.05) -- (-0.78,-0.05);  \draw[red] (-0.99,-0.05) node{$1\!2$};
\draw [->] (-0.68, 0.02) -- (-0.51, 0.24);  \draw[red] (-0.72,-0.05) node{$1\!1$};
\draw [->] (-0.42, 0.44) -- (-0.06, 1.09);  \draw[red] (-0.47, 0.34) node{$8$};
\draw [->] (-0.06, 1.11) -- ( 0.68, 0.96);  \draw[red] (-0.06, 1.11) node{$7$};
\draw [->] ( 0.76, 0.92) -- ( 1.16, 0.25);  \draw[red] ( 0.76, 0.92) node{$6$};
\draw [->] ( 1.17, 0.17) -- ( 0.95,-0.13);  \draw[red] ( 1.22, 0.17) node{$5$};
\draw [->] ( 0.55,-0.69) -- ( 0.85,-0.23);  \draw[red] ( 0.49,-0.74) node{$3$};
                                            \draw[red] ( 0.87,-0.18) node{$4$};
\draw [->] ( 0.81,-1.12) -- ( 0.54,-0.83);  \draw[red] ( 0.81,-1.22) node{$1$};
\draw [->] ( 0.45,-0.86) -- ( 0.16,-1.41);  \draw[red] ( 0.16,-1.53) node{$2$};
\draw [->] (-0.27,-1.02) -- ( 0.43,-0.78);  \draw[red] (-0.38,-1.02) node{$10$};
\draw [->] (-0.20,-0.43) -- (-0.35,-0.90);  \draw[red] (-0.20,-0.33) node{$9$};
\draw [->] (-0.29,-0.29) -- (-0.59,-0.09);
\draw [->] (-1.26,-0.21) -- (-1.26, 0.27);  \draw[red] (-1.26,-0.27) node{$1\!4$};
\draw [->] (-0.50,-1.00) -- (-1.19,-0.36);
%arrows
%\draw[blue] ( *0.5 *0.5, *0.5 *0.5) node{$ $};
\draw[blue] ( 0.75,-0.93) node{$a_1$};
\draw[blue] ( 0.22,-1.09) node{$a_2$};
\draw[blue] ( 0.81,-0.51) node{$a_3$};
\draw[blue] ( 1.22,-0.04) node{$a_4$};
\draw[blue] ( 1.11, 0.65) node{$a_5$};
\draw[blue] ( 0.35, 1.15) node{$a_6$};
\draw[blue] (-0.36, 0.73) node{$a_7$};
\draw[blue] (-0.43,-0.29) node{$a_8$};
\draw[blue] (-0.41,-0.62) node{$a_9$};
\draw[blue] ( 0.11,-0.73) node{$a_{1\!0}$};
\draw[blue] (-0.67, 0.21) node{$a_{1\!1}$};
\draw[blue] (-0.95*0.5-0.78*0.5, -0.11*0.5-0.31*0.5) node{$a_{1\!2}$};
\draw[blue] (-1.02, 0.27) node{$a_{1\!3}$};
\draw[blue] (-1.46*0.5-1.39*0.5, -0.27*0.5+0.35*0.5) node{$a_{1\!4}$};
\draw[blue] (-0.89,-0.76) node{$a_{1\!0}'$};
\draw (0,-1.70) node{subfigure 2};
\end{tikzpicture}
\caption{The subfigure 2 is the gentle algebra whose marked ribbon surface is shown in subfigure 1.}
\label{Examples 2}
\end{center}
\end{figure}

\noindent The elementary polygon of the marked ribbon surface gives the projective dimension of each simple module.
Such as the arc $S(8)$, by Proposition \ref{prop. pdM and idM}, each term in
the minimal projective resolution of $S(8)$ corresponds to an arc of the elementary polygon $\Delta_j$
whose arc set $\mathfrak{S}(\Delta_j) = \{6,7,8,9\}$.
Then the minimal projective resolution of $S(8)$ is
\[ 0 \longrightarrow M(c_{\proj}^{6}) \longrightarrow M(c_{\proj}^{7}) \longrightarrow
M(c_{\proj}^{8}) \longrightarrow S(8)=M(c_{\simp}^{8}) \to 0,\]
thus $\pdim S(8) = 2$. Similarly, we get
\begin{align}
 & \pdim S(1)=2, &&\pdim S(2)=0,  && \pdim S(4)=0,  && \pdim S(5)=1,  && \pdim S(6)=1, \nonumber\\
 & \pdim S(7)=1, &&\pdim S(11)=1, && \pdim S(12)=1, && \pdim S(13)=1, && \pdim S(14)= 2.  \nonumber
\end{align}
The simple module $S(9)=M(c_{\simp}^9)$ corresponds to the arc $9$
which satisfies the {\rm P}-condition, its minimal projective resolution is
\begin{align}
0 & \longrightarrow M(c_{\proj}^{12}) \longrightarrow M(c_{\proj}^{6}) \oplus M(c_{\proj}^{14})
\longrightarrow M(c_{\proj}^{7}) \oplus M(c_{\proj}^{14}) \nonumber \\
& \longrightarrow M(c_{\proj}^{8}) \oplus M(c_{\proj}^{10})
\longrightarrow M(c_{\proj}^{9}) \longrightarrow S(9)=M(c_{\simp}^{9}) \longrightarrow 0, \nonumber
\end{align}
thus $\pdim S(9) = 4$. Similarly, we get
$$ \pdim S(3)=1, \ \pdim S(10)=3. $$
Consequently we conclude that $\mathrm{gl.dim}A = 4$. It should be pointed out that
we can get $\mathrm{gl.dim}A = 4$ by Theorem \ref{thm. gl.dim} directly.
\end{example}

\begin{example}\label{exam-7.3} \rm
Let $A$ be a gentle algebra whose marked ribbon surface $\Surf_A$ is given in Example \ref{Ex. GProj.}.
Then $\mathrm{gl.dim}A = \infty$ by Theorem \ref{thm. Gdim A} because $\Surf_A$ has an $\infty$-elementary polygon
with the arc set $\{1, 2, 3\}$.
Indeed, the minimal projective resolution of $S(1)$ is
$$\cdots \to  P(3) \to  P(2) \to P(1) \to  P(3) \to  P(2) \to P(1)\to S(1)\to 0 $$
by Proposition \ref{prop. pdM and idM}, and so $\pdim S(1)=\infty$.
\end{example}

\subsection{Some examples for calculating self-injective dimension of gentle algebras}
In this subsection, we provide some examples to calculate projective (respectively, injective)
dimension of indecomposable injective (respectively, projective) modules,
and the self-injective dimension of gentle algebras.

\begin{example} \rm \label{exp. 3}
Consider the gentle algebra $A$ given in Example \ref{exp. 2}. Then $\gldim A=4$.
Since $I(8)={{}^{{}^{^{13}_{12}}_{11}}}{_{8}}{^{9}}=M(c_{\inj}^8)$ is an indecomposable injective module
with the top $M(c_{\simp}^{13}) \oplus M(c_{\simp}^{9}) \cong S(13)\oplus S(9)$, its projective cover is
$M(c_{\proj}^{13}\oplus c_{\proj}^9)$.
Consider the elementary polygons $\Delta_{9}$ with sides $\{9,10,14,13,12\}$ and $\Delta_{12}$ with side $\{12\}$.
We have $\mathfrak{C}(\Delta_{9}) = 5$ and $\mathfrak{C}(\Delta_{12}) = 1$, thus $\pdim I(8) = \max\{5, 1\} -1 = 4$.
Furthermore, by the proof of Proposition \ref{prop. pdim. of inj. mod.}, the minimal projective resolution of $I(8)$ is
as follows:
$$0 \rightarrow M(c_{\proj}^{12}) \longrightarrow M(c_{\proj}^{13}) \longrightarrow M(c_{\proj}^{14}) \nonumber
\longrightarrow M(c_{\proj}^{8})\oplus M(c_{\proj}^{10}) \longrightarrow M(c_{\proj}^{13})\oplus M(c_{\proj}^{9})
\longrightarrow I(8) \rightarrow 0. \nonumber$$
Thus we have $\idim_AA = \idim A_A = \pdim D(A) = 4 = \gldim A$.
%\item[(2)] For the projective module $P(9)$, we have $\idim P(9)=2$. Indeed, its minimal injective resolution is as follows:
%5\[ 0 \longrightarrow P(9) \longrightarrow M(c_{\inj}^{8})\oplus M(c_{\inj}^{4})
%\longrightarrow M(c_{\inj}^{11})\oplus M(c_{\inj}^9)\oplus M(c_{\inj}^{5})\longrightarrow M(c_{\inj}^{6}) \longrightarrow 0. \]
%\item[(3)] We have $\idim_AA = \idim A_A = \pdim D(A) = 4 = \gldim A$.
%Indeed, by Proposition \ref{prop. GProj}, we have that
%$P$ is projective if and only if $P$ is Gorenstein projective in this example,
%because $A$ has no oriented cycle with full relations.
%\end{enumerate}
\end{example}

In Example \ref{exp. 3}, for the gentle algebra $A$, its global and self-injective dimensions are identical.
Now we give two examples of gentle algebras with infinite global dimension.

\begin{example} \rm \label{exp. 4}
\begin{enumerate}
\item[]
\item[(1)] Let $A$ be the gentle algebra given in Example \ref{Ex. GProj.}.
Then $\mathrm{gl.dim}A = \infty$ by Example \ref{exam-7.3}.
Since $\max\limits_{\Delta_j\in\finiteDelta}\mathfrak{C} (\Delta_j)=2$,
we have $\idim_AA = \idim A_A = \max\limits_{\Delta_i\in\finiteDelta}\{1, \mathfrak{C} (\Delta_i)-1\} = 1$ by Theorem \ref{thm. Gdim A}.

\item[(2)] Let $A$ be the gentle algebra whose marked ribbon surface is shown in the \Pic \ref{Ex. Gdim proj} I.
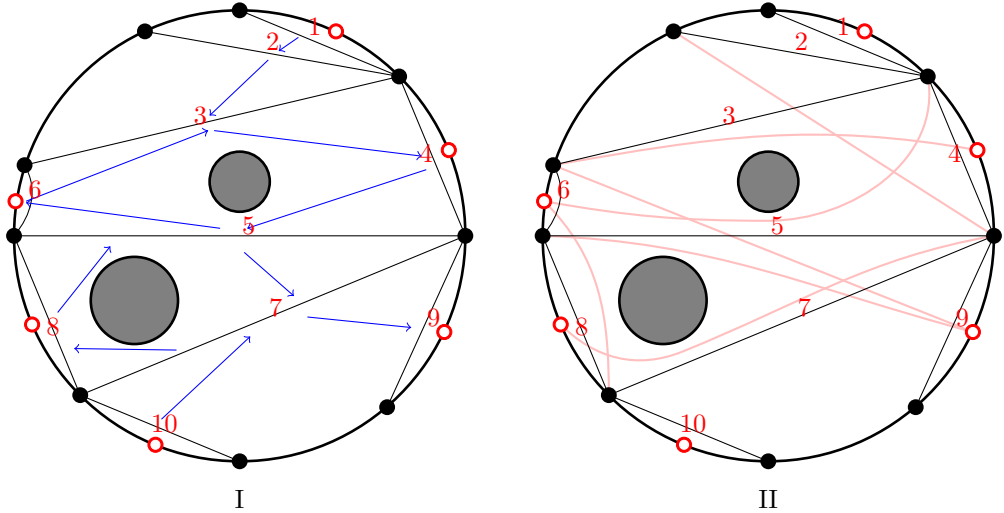
\begin{figure}[htbp]
\begin{center}
\begin{tikzpicture}[scale=2]
% boundaries
\draw [line width=1pt] (0,0) circle (1.5cm);
\filldraw[gray] (0,0.36) circle (0.2cm);
\draw [line width=1pt] (0,0.36) circle (0.2cm);
\filldraw[gray] (-0.7,-0.43) circle (0.29cm);
\draw [line width=1pt] (-0.7,-0.43) circle (0.29cm);
% arcs
\draw (0.98,-1.14)-- (1.5,0);
\draw (1.5,0)-- (1.06,1.06);
\draw (1.06,1.06)-- (0,1.5);
\draw (0,-1.5)-- (-1.06,-1.06);
\draw (-1.06,-1.06)-- (1.5,0);
\draw (-1.06,-1.06)-- (-1.5,0);
\draw (-1.5,0)-- (1.5,0);
\draw (1.06,1.06)-- (-1.43,0.47);
\draw (1.06,1.06)-- (-0.63,1.36);
\draw (-1.43,0.47) to[out=-60,in=50] (-1.5,0);
\draw [->] [blue] (0.39,1.32) -- (0.26,1.23);
\draw [->] [blue] (0.19,1.17) -- (-0.2,0.8);
\draw [->] [blue] (-0.17,0.7) -- (1.21,0.53);
\draw [->] [blue] (1.24,0.44) -- (0.05,0.05);
\draw [->] [blue] (-1.21,-0.51) -- (-0.86,-0.07);
\draw [->] [blue] (-0.42,-0.76) -- (-1.1,-0.75);
\draw [->] [blue] (0.03,-0.11) -- (0.36,-0.4);
\draw [->] [blue] (0.45,-0.54) -- (1.14,-0.61);
\draw [->] [blue] (-0.52,-1.22) -- (0.07,-0.67);
\draw [->] [blue] (-0.13,0.05) -- (-1.42,0.22);
\draw [->] [blue] (-1.42,0.23) -- (-0.21,0.70);
\fill [color=black] (-1.06,-1.06) circle (1.5pt);
\fill [color=black] (0,-1.5) circle (1.5pt);
\fill [color=black] (0.98,-1.14) circle (1.5pt);
\fill [color=black] (1.5,0) circle (1.5pt);
\fill [color=black] (1.06,1.06) circle (1.5pt);
\fill [color=black] (0,1.5) circle (1.5pt);
\fill [color=black] (-1.5,0) circle (1.5pt);
\fill [color=black] (-1.43,0.47) circle (1.5pt);
\fill [color=black] (-0.63,1.36) circle (1.5pt);
\draw[red] (0.5,1.39) node {$1$};
\draw[red] (0.22,1.29) node {$2$};
\draw[red] (-0.26,0.8) node {$3$};
\draw[red] (1.24,0.55) node {$4$};
\draw[red] (0.06,0.07) node {$5$};
\draw[red] (0.24,-0.48) node {$7$};
\draw[red] (-1.36,0.31) node {$6$};
\draw[red] (-1.24,-0.6) node {$8$};
\draw[red] (1.29,-0.54) node {$9$};
\draw[red] (-0.5,-1.25) node {$10$};
% extra points
\filldraw[color=red] ( 0.64, 1.36) circle (0.05); \filldraw[color=white] ( 0.64, 1.36) circle (0.03);
\filldraw[color=red] ( 1.39, 0.57) circle (0.05); \filldraw[color=white] ( 1.39, 0.57) circle (0.03);
\filldraw[color=red] ( 1.36,-0.64) circle (0.05); \filldraw[color=white] ( 1.36,-0.64) circle (0.03);
\filldraw[color=red] (-0.56,-1.39) circle (0.05); \filldraw[color=white] (-0.56,-1.39) circle (0.03);
\filldraw[color=red] (-1.38,-0.59) circle (0.05); \filldraw[color=white] (-1.38,-0.59) circle (0.03);
\filldraw[color=red] (-1.49, 0.23) circle (0.05); \filldraw[color=white] (-1.49, 0.23) circle (0.03);
\draw (0,-1.75) node {\text{I}};
\end{tikzpicture}
\ \ \ \ \ \
\begin{tikzpicture}[scale=2]
% Gorenstein projective modules
\draw [line width=0.8pt,pink] ( 1.50, 0.00) to (-0.63, 1.36);
\draw [line width=0.8pt,pink] (-1.43, 0.47) to[out=10, in=170] ( 1.39, 0.57);
\draw [line width=0.8pt,pink] (-1.43, 0.47) to ( 1.36,-0.64);
\draw [line width=0.8pt,pink] (-1.49, 0.23) to[out=-10, in= 180] (0,0.1) to[out=0,in=-80] ( 1.06, 1.06);
\draw [line width=0.8pt,pink] (-1.06,-1.06) to[out= 90, in=-45 ] (-1.49, 0.23);
\draw [line width=0.8pt,pink] (-1.50, 0.00) to[out= -2, in= 165] ( 1.36,-0.64);
\draw [line width=0.8pt,pink] (-1.38,-0.59) to[out=-45, in=-158] (-0.57,-0.77) to [out=22,in=-170] ( 1.50, 0.00);
% boundaries
\draw [line width=1pt] (0,0) circle (1.5cm);
\filldraw[gray] (-0.7,-0.43) circle (0.29cm);
\draw [line width=1pt] (-0.7,-0.43) circle (0.29cm);
\filldraw[gray] (0,0.36) circle (0.2cm);
\draw [line width=1pt] (0,0.36) circle (0.2cm);
%arcs
\draw (0.98,-1.14)-- (1.5,0);
\draw (1.5,0)-- (1.06,1.06);
\draw (1.06,1.06)-- (0,1.5);
\draw (0,-1.5)-- (-1.06,-1.06);
\draw (-1.06,-1.06)-- (1.5,0);
\draw (-1.06,-1.06)-- (-1.5,0);
\draw (-1.5,0)-- (1.5,0);
\draw (1.06,1.06)-- (-1.43,0.47);
\draw (1.06,1.06)-- (-0.63,1.36);
\draw (-1.43,0.47) to[out=-60,in=50] (-1.5,0);
\draw[red] (0.5,1.39) node {$1$};
\draw[red] (0.22,1.29) node {$2$};
\draw[red] (-0.26,0.8) node {$3$};
\draw[red] (1.24,0.55) node {$4$};
\draw[red] (0.06,0.07) node {$5$};
\draw[red] (0.24,-0.48) node {$7$};
\draw[red] (-1.36,0.31) node {$6$};
\draw[red] (-1.24,-0.6) node {$8$};
\draw[red] (1.29,-0.54) node {$9$};
\draw[red] (-0.5,-1.25) node {$10$};
% marked points
\fill [color=black] (-1.06,-1.06) circle (1.5pt);
\fill [color=black] (0,-1.5) circle (1.5pt);
\fill [color=black] (0.98,-1.14) circle (1.5pt);
\fill [color=black] (1.5,0) circle (1.5pt);
\fill [color=black] (1.06,1.06) circle (1.5pt);
\fill [color=black] (0,1.5) circle (1.5pt);
\fill [color=black] (-1.5,0) circle (1.5pt);
\fill [color=black] (-1.43,0.47) circle (1.5pt);
\fill [color=black] (-0.63,1.36) circle (1.5pt);
% extra points
\filldraw[color=red] ( 0.64, 1.36) circle (0.05); \filldraw[color=white] ( 0.64, 1.36) circle (0.03);
\filldraw[color=red] ( 1.39, 0.57) circle (0.05); \filldraw[color=white] ( 1.39, 0.57) circle (0.03);
\filldraw[color=red] ( 1.36,-0.64) circle (0.05); \filldraw[color=white] ( 1.36,-0.64) circle (0.03);
\filldraw[color=red] (-0.56,-1.39) circle (0.05); \filldraw[color=white] (-0.56,-1.39) circle (0.03);
\filldraw[color=red] (-1.38,-0.59) circle (0.05); \filldraw[color=white] (-1.38,-0.59) circle (0.03);
\filldraw[color=red] (-1.49, 0.23) circle (0.05); \filldraw[color=white] (-1.49, 0.23) circle (0.03);
\draw (0,-1.75) node {\text{II}};
\end{tikzpicture}
\caption{The marked ribbon surface of gentle algebra given in Example \ref{exp. 4} (subfigure I) and
all permissible curves corresponding to indecomposable non-projective Gorenstein
projective modules (subfigure II). }
\label{Ex. Gdim proj}
\end{center}
\end{figure}
Then $\gldim A=\infty$ by Theorem \ref{thm. gl.dim} (because the elementary polygon with sides $\{3,4,5,6\}$
and the elementary polygon with sides $\{5,7,8\}$ are $\infty$-elementary), and $\idim_AA = \idim A_A=2$ by Theorem \ref{thm. Gdim A}.
Furthermore, since its AG-invariant $\phi_A = [(9,4), (0,4), (0,3)]$, we have
$$\sharp \ind(\nonprojGP(A)) = 4\phi_A(0,4) + 3\phi_A(0,3) = 4\cdot 1+3\cdot 1 = 7$$
by Proposition \ref{prop. CM-corresponding}.
Indeed, we have $\ind(\nonprojGP(A)) = \{3,\ 4,\ {^{_5}_{^7_9}},\ {^7_9},\ 8,\ {^5_6},\ 6\}$ by Proposition \ref{prop. GProj}.
\end{enumerate}
\end{example}

\section*{\textbf{Acknowledgement}}
The authors thank the referees for very useful and detailed suggestions.


\begin{thebibliography}{99}\setlength{\itemsep}{- 2mm}
\newcommand{\arxiv}[1]{\href{http://arxiv.org/abs/#1}{arXiv:#1}}
\definecolor{arxivlink}{rgb}{0,0.7,0.4}

\bibitem{AB} Auslander, M., Bridger, M.:
        {\rm Stable module theory},
        {\it Mem. Amer. Math. Soc.} {\bf 94}, Amer. Math. Soc., Providence, RI, 1969.

\bibitem{AAG2008} Avella-Alaminos, D., Gei{\ss}, C.:
        {\rm Combinatorial derived invariants for gentle algebras},
        {\it J. Pure Appl. Algebra} \textbf{212} (2008), 228{--}243.

\bibitem{ALP2016} Arnesen, K. K., Laking, R., Pauksztello, D.:
        {\rm Morphisms between indecomposable objects in the derived category of a gentle algebra},
        {\it J. Algebra} \textbf{467} (2016), 1{--}46.

\bibitem{APS2019} Amiot, C., Plamondon, P.-G., Schroll, S.:
        {\rm A complete derived invariant for gentle algebras via winding numbers and Arf invariants},
        \arxiv{1904.02555} [math. RT] (2019).

\bibitem{AS1987} Assem, I., Skowronski, A.:
        {\rm Iterated tilted algebrs of type $\tilde{\mathbb{A}}_n$},
        {\it Math. Z.} \textbf{195} (1987), 269{--}290.

\bibitem{BM2003} Bekkert, V., Merklen, H. A.:
        {\rm Indecomposables in derived categories of gentle algebras},
        {\it Algebr. Represent. Theory} \textbf{6} (2003), 285{--}302.

\bibitem{BS2018} Baur, K., Coelho Sim\~oes, R.:
        {\rm A geometric model for the module category of a gentle algebra},
        {\it Int. Math. Res. Not. IMRN} 2021, no. 15, 11357--11392.

\bibitem{BR87} Butler, M. C. R., Ringel, C. M.:
        {\rm Auslander-Reiten sequences with few middle terms and applications to string algebras},
        {\it Comm. Algebra} \textbf{15} (1987), 145{--}179.

\bibitem{BZ2011} Br\"{u}stle, T., Zhang, J.:
        {\rm On the cluster category of a marked surface without punctures},
        {\it Algebra and Number Theory} \textbf{5} (2011), 529{--}566.

\bibitem{CL2019} Chen, X., Lu, M.:
        {\rm Cohen-Macaulay Auslander algebras of gentle algebras},
        {\it Comm. Algebra} \textbf{47} (2019), 3597{--}3613.
        %doi: 10.1080/00927872. 2019. 1570225.

\bibitem{C1989} Crawley-Boevey, W. W.:
        {\rm Maps between representations of zero-relation algebras},
        {\it J. Algebra} \textbf{126} (1989), 259{--}263.

\bibitem{Erd1990} Erdmann, K.:
        {\rm Blocks of Tame Representation Type and Related Algebras},
        {\it Lectures Notes in Math.} {\bf 1428}, Springer-Verlag, Berlin-Heidelberg, 1990.

\bibitem{EJ1995} Enochs, E. E., Jenda, O. M. G.:
        {\rm Gorenstein injective and projective modules},
        {\it Math. Z.} \textbf{220} (1990), 611{--}633.

\bibitem{GR2005} Gei{\ss}, C., Reiten, I.:
        {\rm Gentle algebras are Gorenstein},
        in {\it Representations of Algebras and Related Topics},
        Proceedings from the 10th International Conference, Toronto, Canada, July 15--August 10, 2002,
        Amer. Math. Soc., Providence, RI, Fields Institute Comm. {\bf 45}, pp.129--133, 2005.

\bibitem{HKK2017} Haiden, F., Katzarkov, L., Kontsevich, M.:
        {\rm Flat surfaces and stability structures},
        {\it Publ. Math. Inst. Hautes \'{E}tudes Sci.}, \textbf{126} (2017) 247{--}318.

\bibitem{H1991} Happel, D.:
        {\rm On Gorenstein algebras},
        in {\it Representation Theory of Finite Groups and Finite-dimensional Algebras},
        Proceedings of the Conference at the University of Bielefeld from May 15--17, 1991,
        Prog. Math. {\bf 95}, pp.389--404, 1991.


\bibitem{HZZ2020} He, P., Zhou, Y., Zhu, B.:
        {\rm A geometric model for the module category of a skew-gentle algebra},
        \arxiv{2004.11136} (v3) [math. RT] (2020).

\bibitem{Ho} Hoshino, M.: {\rm Algebras of finite self-injective dimension},
{\it Proc. Amer. Math. Soc.} {\bf 112} (1991), 619--622.

\bibitem{HuH} Huang, C. H., Huang, Z. Y.: {\rm Torsionfree dimension of modules
and self-injective dimension of rings}, {\it Osaka J. Math.} {\bf 49} (2012), 21--35.

\bibitem{K2015} Kalck, M.: {\rm Singularity categories of gentle algebras},
        {\it Bull. London Math. Soc.} \textbf{47} (2015), 65{--}74.
        %doi: 10.1112/blms/bdu093.

\bibitem{Kra91} Krause, H.: Maps between tree and band modules,
        {\it J. Algebra} \textbf{137} (1991), 186{--}194.

\bibitem{LF2009a} Labardini-Fragoso, D.:
        {\rm Quiver with potentials associated to triangulated surfaces},
        {\it Proc. Lond Math. Soc.} \textbf{98} (2009), 797{--}839.
        %MR2010b: 16033 Zbl 05551833.

\bibitem{LZ2021} Liu, Y.-Z., Zhang, C.:
        {\rm The geometric model of gentle one-cycle algebras},
        {\it Bull. Malaysian Math. Sci. Soc.} {\bf 44} (2021), 2489{--}2505.
       % DOI: 10.1007/s40840-021-01078-y.

\bibitem{LF2009b} Labardini-Fragoso, D.:
        {\rm Quiver with potentials associated to triangulated surfaces, II: arc representations},
        \arxiv{0909.4100} [math. RT] (2009).

\bibitem{OPS2018} Opper, S., Plamondon, P.-G., Schroll, S.:
        {\rm A geometric model for the derived category of gentle algebras},
        \arxiv{1801.09659} (v5) [math. RT] (2018).

\bibitem{QZ2017} Qiu, Y., Zhou, Y.:
        {\rm Cluster categories for marked surfaces: punctured case},
        {\it Compositio Math.} \textbf{153} (2017), 1779{--}1819.

\bibitem{R97} Ringel, C. M.:
        {\rm The repetitive algebra of a gentle algebra},
        {\it Bol. Soc. Mat. Mexicana} \textbf{3} (1997) 235{--}253.

\bibitem{WW1985} Wald, B., Waschb\"{u}sch, J.:
        {\rm Tame biserial algebras},
        {\it J. Algebra} \textbf{95} (1985), 480{--}500.

\end{thebibliography}
\end{document}